\documentclass[toc=bibliography]{scrartcl}

\usepackage[utf8]{inputenc}
\usepackage[T1]{fontenc}

\usepackage{graphicx}
\usepackage{amsmath,amssymb,amsthm}
\usepackage{esint} 
\usepackage{bbm}
\usepackage{enumitem,color}
\usepackage{comment}
\usepackage{hyperref}

\newtheorem{thm}{Theorem}[section]
\newtheorem{prop}[thm]{Proposition}
\newtheorem{lem}[thm]{Lemma}
\newtheorem{cor}[thm]{Corollary}
\theoremstyle{definition}

\theoremstyle{remark}
\newtheorem{rem}[thm]{Remark}

\numberwithin{equation}{section}
\numberwithin{figure}{section}

\newcommand{\fig}[2]{\includegraphics[width=#1\textwidth]{#2}}

\usepackage{xstring}

\newcommand{\al}{\mathtt{a}}
\newcommand{\bl}{\mathtt{b}}
\newcommand{\cl}{\mathtt{c}}

\newcommand{\vl}{\mathtt{v}}
\newcommand{\wl}{\mathtt{w}}

\newcommand{\td}{\tilde{d}}
\newcommand{\tdel}{\tilde{\partial}}

\newcommand{\ov}{\overline}
\newcommand{\un}{\underline}
\newcommand{\Z}{\mathbb{Z}}
\newcommand{\N}{\mathbb{N}}
\newcommand{\R}{\mathbb{R}}

\newcommand{\E}{\mathbb{E}}
\newcommand{\abs}[1]{\lvert #1 \rvert}

\newcommand{\CC}{\mathcal{C}}
\newcommand{\LL}{\mathcal{L}}
\renewcommand{\AA}{\mathcal{A}}
\newcommand{\RR}{\mathcal{R}}
\newcommand{\FF}{\mathcal{F}}
\newcommand{\BB}{\mathcal{B}}
\newcommand{\DD}{\mathcal{D}}
\newcommand{\GG}{\mathcal{G}}
\newcommand{\MM}{\mathcal{M}}
\newcommand{\QQ}{\mathcal{Q}}
\newcommand{\TT}{\mathcal{T}}

\renewcommand{\P}{\mathbb{P}}

\title{Bijective enumeration of planar bipartite maps with three
  tight boundaries, or how to slice pairs of pants}%
\author{J\'er\'emie Bouttier%
  \thanks{Université Paris-Saclay, CNRS, CEA, Institut de physique
    théorique, 91191, Gif-sur-Yvette, France} %
  \thanks{Univ Lyon, Ens de Lyon, Univ Claude Bernard, CNRS,
    Laboratoire de Physique, F-69342 Lyon} \and %
  Emmanuel Guitter\footnotemark[1] \and %
  Grégory Miermont\thanks{ENS de Lyon, UMPA, CNRS UMR 5669, 46 allée d’Italie, 69364 Lyon Cedex 07, France}}%
\date{\today}

\begin{document}
\maketitle

\begin{abstract}
  We consider planar maps with three boundaries, colloquially called
  pairs of pants. In the case of bipartite maps with controlled face
  degrees, a simple expression for their generating function was found
  by Eynard and proved bijectively by Collet and Fusy. In this paper,
  we obtain an even simpler formula for \emph{tight} pairs of pants,
  namely for maps whose boundaries have minimal length in their
  homotopy class. We follow a bijective approach based on the slice
  decomposition, which we extend by introducing new fundamental
  building blocks called bigeodesic triangles and diangles, and by
  working on the universal cover of the triply punctured sphere. We
  also discuss the statistics of the lengths of minimal separating
  loops in (non necessarily tight) pairs of pants and annuli, and
  their asymptotics in the large volume limit. 
\end{abstract}

\newpage
\tableofcontents

\section{Introduction}
\label{sec:intro}

\paragraph{Context and motivations.}

The study of maps (graphs embedded into surfaces) is an active field
of research, at the crossroads between combinatorics, theoretical
physics and probability theory. The combinatorial theory of maps
started with the pioneering work of Tutte in the
1960's~\cite{Tutte68}, and we refer to the recent review by
Schaeffer~\cite{Schaeffer15} for an account of its many developments
ever since. In theoretical physics, maps are intimately connected with
matrix models and two-dimensional quantum gravity: see for instance
the review by Di Francesco, Ginsparg and Zinn-Justin~\cite{DGZ95}, the
book by Ambjørn, Durhuus and Jonsson~\cite{ADJ97}, and the book by
Eynard~\cite{Eynard16} for more recent mathematical advances including
the theory of topological recursion. Probability theory aims at
understanding the geometric properties of large random maps and their limits:
this topic is covered in several sets of lecture
notes~\cite{LGMi12,MiermontSFnotes,BuddRMGFFnotes,CurienSFnotes}, and
we also mention the review by Miller~\cite{Miller18} which discusses
the connection with Liouville quantum gravity, a rigorous approach to
two-dimensional quantum gravity.

A key tool in the study of maps is the \emph{bijective approach},
which consists in finding correspondences between different families
of maps, or with other combinatorial objects such as trees or lattice
walks. Bijections often yield elementary derivations of enumerative
results, but are also useful to understand properties of maps such as
distances (see the aforementioned references). There exists by now
several general bijective frameworks, and in this paper we focus on a
specific one, called the slice decomposition.

Colloquially speaking, the~\emph{slice decomposition} consists in
performing a canonical decomposition of maps, by cutting them along
leftmost geodesics. It was first mentioned in the
papers~\cite{pseudoquad,hankel,constfpsac}, mostly as a reformulation
of the decomposition of mobiles~\cite{mobiles}. Its real significance
was highlighted in the paper~\cite{irredmaps}---see
also~\cite[Chapter~2]{BouttierHDR} for a recent exposition---which
considers so-called irreducible maps for which bijections were not
known before.  The slice decomposition also passes naturally to the
scaling limit~\cite{LeGall13,BeMi17}.  However, it has so far been
understood only in the case of \emph{disks} and \emph{annuli}, namely
planar maps with one or two boundaries. Our purpose is to understand
the case of maps with other topologies, with the long-term goal of
developing a bijective approach to topological recursion.

In this paper, we make a first step in this direction, by considering
\emph{pairs of pants}, namely planar maps with three boundaries. For
simplicity, we restrict to the case of bipartite maps with controlled
face degrees (also known as Boltzmann maps), though we believe that
our treatment can be extended to the non bipartite or to the
irreducible settings as considered in \cite{hankel,irredmaps}.  A
simple explicit expression of the generating function of bipartite
pairs of pants was given by Eynard~\cite[Proposition~3.3.1]{Eynard16}
and derived bijectively by Collet and Fusy~\cite{CoFu12}. We note that
equivalent formulas appeared previously in the physics literature, see
for instance  \cite[Equation (45)]{AJM} or \cite[Equation
(4.94)]{ADJ97}. 
Here, we
obtain an even simpler formula for \emph{tight} pairs of pants, namely
for maps whose boundaries have minimal length in their homotopy
class. As we shall see, our formula is equivalent to the
Eynard-Collet-Fusy formula, but our derivation is fundamentally
different.

\begin{figure}[t]
  \centering
  \includegraphics[width=.9\textwidth]{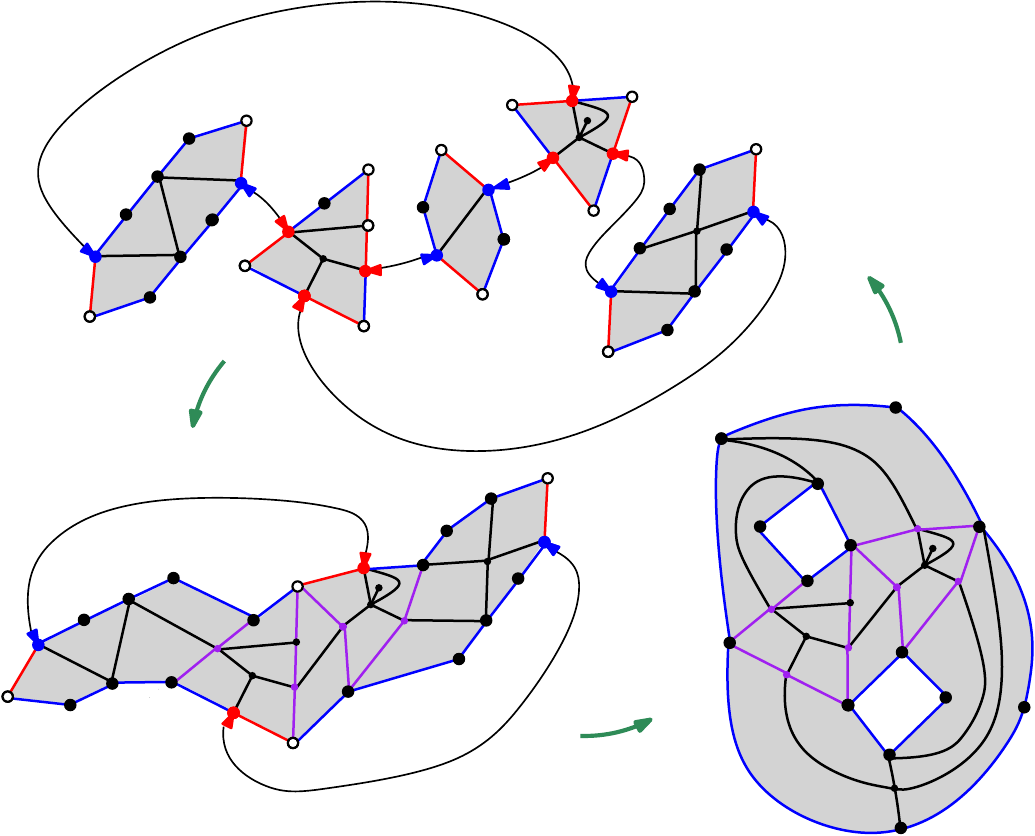}
  \caption{Illustration of the main bijective construction discussed in this
    paper. Starting from elementary pieces, namely two bigeodesic
    triangles and three bigeodesic diangles (top), one builds a pair of
    pants with three tight boundaries (bottom right). To better
    visualize the construction, we pass through an intermediate
    partial assembling (bottom left). Conversely, the building blocks
  can be recovered by cutting along appropriately defined bigeodesics,
here displayed in purple.}
  \label{fig:assemblingsample}
\end{figure}

Our approach, whose general idea is displayed in Figure \ref{fig:assemblingsample}, consists in decomposing tight pairs of pants into
geometric pieces which we call (bigeodesic) diangles and
triangles. While the former are, in a sense, generalizations of the
existing notion of slices, the second are new objects (although they
appear implicitly in earlier work \cite{threepoint}, see the
discussion in Appendix \ref{sec:wlm}). As was pointed to us by Bram
Petri, the way in which the elementary pieces are assembled is very
much reminiscent of classical constructions of pairs of pants in
hyperbolic geometry from ideal hyperbolic triangles, see for instance
\cite[Section 3.4]{Thurston97}. In particular, some notions of
importance in this paper, which we refer to as ``equilibrium
vertices'' in triangles and ``exceedances'' in diangles, have natural
analogs in hyperbolic geometry: the equilibrium vertices correspond
to tangency points of the inner circles of the ideal triangles, and
the exceedances correspond to the invariants $d(v)$ in
\cite{Thurston97}.

We believe that many other connections exist between these classical
concepts and our work. In particular, in the context of the
classification of Riemann and hyperbolic surfaces~\cite{ImTa92}, pants
decompositions play a fundamental role. It is therefore natural to
expect that similar decompositions should exist in the context of
maps. In particular, the tightness constraint which we introduce
should be an important ingredient: indeed, it should translate the
natural idea of cutting surfaces along closed geodesics, in order to
obtain canonical decompositions. Such pants decompositions will be
explored in future research, but provide one of our main motivations
for the present paper.

\paragraph{Overview.}

A \emph{planar map} is a connected multigraph embedded into the sphere
without edge crossings, and considered up to homeomorphism.
It consists of \emph{vertices}, \emph{edges},
\emph{faces} and \emph{corners}, see~\cite{Schaeffer15} for precise
definitions. Until further notice, we only consider finite maps, i.e.\
maps with a finite number of edges (hence of vertices, faces and
corners). A \emph{path} on a map is a sequence of consecutive edges,
and the \emph{length} of a path is its number of edges. Given a face,
its \emph{contour} is the closed path formed by its incident edges,
and its \emph{degree} is the length of the contour. A planar map is
\emph{bipartite} if all its faces have even degree.

A \emph{boundary} is either a marked face or a marked vertex on the
map. We will use the denominations \emph{boundary-face} and
\emph{boundary-vertex} when we wish to specify the nature of a
boundary. We define the \emph{length} of a boundary as being equal to
its degree in the case of a boundary-face, and to zero in the case of
a boundary-vertex. Faces which are not boundaries are called
\emph{inner faces}. A map is said \emph{essentially bipartite} if all
its inner faces have even degree. The sum of the lengths of the
boundaries of an essentially bipartite map is necessarily even.

\begin{figure}
  \centering
  \includegraphics[width=.5\textwidth]{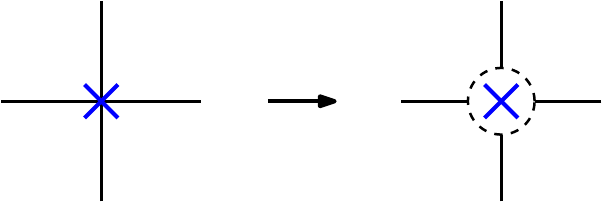}
  \caption{An intuitive way to think of a boundary-vertex: we remove a
    small disk around it, and keep all incident edges connected along
    a circle made of special (dashed) edges, which are considered as
    having length zero. Note that a path passing through a
    boundary-vertex may ``circumvent'' it in two ways, which will
    correspond to different homotopy classes when there are other
    boundaries.}
  \label{fig:boundaryvertex}
\end{figure}

We intuitively think of boundaries as representing punctures on the
sphere. This is rather natural in the case of a boundary-face (we just
remove its interior from the surface), but slightly less in the case
of boundary-vertex: see Figure~\ref{fig:boundaryvertex}. A path on the
map, together with a choice of circumventing direction when passing
through a boundary-vertex, corresponds to a path on the punctured
sphere. Two closed paths are said to be in the same \emph{homotopy
  class}, or \emph{freely homotopic}, if they can be continuously
deformed into one another on the punctured sphere.  A boundary-face is
said \emph{tight} if its contour has minimal length in its homotopy
class (if the boundary-face is incident to a boundary-vertex, the
contour should be considered as the contour of the corresponding face
in the map modified as in Figure~\ref{fig:boundaryvertex}). A
boundary-vertex is by convention always tight.

We are interested in essentially bipartite planar maps with three
distinct boundaries which are labeled (distinguishable). Such maps
cannot have symmetries, and therefore we do not root (i.e., mark a
corner on) the boundaries.  Two situations may occur: either all the
boundary lengths are even, and the planar map is truly bipartite, or
two lengths are odd and the third is even, and following~\cite{CoFu12}
we say that the map is \emph{quasi-bipartite}.  We may now state our
main enumerative result:

\begin{thm}
  \label{thm:Tabc}
  Let $a$, $b$ and $c$ be integers or half-integers such that $a+b+c$
  is an integer. Then, the generating function $T_{a,b,c}$ of
  essentially bipartite planar maps with three labeled distinct tight
  boundaries of lengths $2a$, $2b$, $2c$, counted with a weight $t$
  per vertex different from a boundary-vertex and, for all $k\geq 1$,
  a weight $g_{2k}$ per inner face of degree $2k$, is equal to
  \begin{equation}
    \label{eq:Tabc}
    T_{a,b,c} = R^{a+b+c} \frac{d \ln R}{dt} - t^{-1} \mathbf{1}_{a=b=c=0}
  \end{equation}
  where $R$ is the formal power series in $t,g_2,g_4,\ldots$
  determined by
  \begin{equation}
    \label{eq:Req}
    R = t + \sum_{k \geq 1} \binom{2k-1}{k} g_{2k} R^k
  \end{equation}
  and where $\mathbf{1}_P$ is equal to $1$ if $P$ is true, and to $0$
  otherwise.

\end{thm}

\begin{figure}
  \centering
  \includegraphics{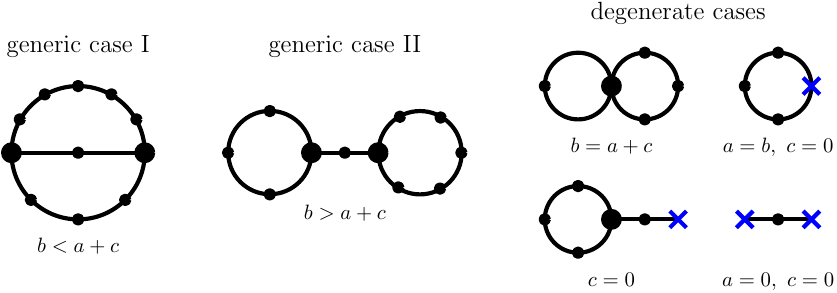}
  \caption{The possible types of maps with three tight boundaries
    and no inner face. Tightness implies
    that there are no vertices of degree one, except possibly
    boundary-vertices, indicated here by blue crosses.  To identify the different types, assume
    without loss of generality that the boundary lengths are
    $b \geq a \geq c$. There exists two generic situations, denoted I
    and II, corresponding to the cases $b<a+c$ and $b>a+c$
    respectively. The degenerate cases correspond to having $b=a+c$
    or/and some lengths equal to zero. We may check that there are
    $a+b+c-1$ vertices different from boundary-vertices in all cases.}
  \label{fig:gzero}
\end{figure}

It is not difficult to check that the right-hand side
of~\eqref{eq:Tabc} is indeed a well-defined power series in
$t,g_2,g_4,\ldots$. It is useful to look first at the case where
$g_2,g_4,\ldots$ all vanish. This corresponds to maps without inner
faces. In this case, $T_{0,0,0}$ vanishes while, for $a,b,c$ not all
zero, $T_{a,b,c}$ is equal to $t^{a+b+c-1}$: this means that there
should exist exactly one such map, with $a+b+c-1$ vertices different
from boundary-vertices. This is indeed true, as illustrated on
Figure~\ref{fig:gzero}.

Our expression for $T_{a,b,c}$ is very similar to the aforementioned
Eynard-Collet-Fusy (ECF) formula for maps with three boundaries that
are not necessarily tight. In fact, the ECF formula simply differs by
some binomial factors. As we will see in Section~\ref{sec:ecf}, the
two formulas are equivalent, by a canonical decomposition which
consists in cutting a map with three arbitrary boundaries along
outermost minimal separating loops, resulting in three annular maps
and one tight pair of pants. However, our expression for $T_{a,b,c}$ being even
simpler than the ECF formula, it is desirable to have a direct
bijective proof of it, and this is the main objective of the present
paper.

Our results have interesting consequences for the statistics of large
random planar maps, which are explored in Section
\ref{sec:scalinglimits}. There, for simplicity, we restrict our attention to
the case of quadrangulations. The aforementioned canonical decomposition of
a map with three boundaries into three annular maps and a tight pair
of pants allows one to define the exterior areas, corresponding to the number of
faces in the annular maps, the interior area, corresponding to the
number of faces in the tight pair of pants, and the minimal separating cycle
lengths, corresponding to the lengths of the three boundaries of the
tight pair of pants. In Theorem \ref{thpants}, we give a detailed limit theorem
for the joint law of these quantities in large quadrangulations with
three boundaries. We also provide an analogous statement for random
annular quadrangulations
in Theorem \ref{thcylinders}, which relies on the results obtained in
\cite{irredmaps}.

Our strategy to prove Theorem~\ref{thm:Tabc} is the following.  For
$a=b=c=0$, the right-hand side of~\eqref{eq:Tabc} can be rewritten as
$d\ln(R/t)/dt$ which, by results from~\cite{mobiles}, is already known
to be equal to the generating function $T_{0,0,0}$ of triply pointed
bipartite maps (for completeness, we provide a slice-theoretic
rederivation of this fact in
Appendix~\ref{sec:slicetp}).
Then, we will exhibit a bijection implying, as
Corollary~\ref{cor:main} below, that we have for any $a,b,c$
\begin{equation}
  \label{eq:Tabcter}
  T_{a,b,c} + t^{-1} \mathbf{1}_{a=b=c=0} = R^{a+b+c} \frac{X^3 Y^2}{t^6}
\end{equation}
where $R$, $X$ and $Y$ are the generating functions of combinatorial
objects called respectively elementary slices, bigeodesic diangles,
and bigeodesic triangles (all these series are equal to $t$ when
$g_2,g_4,\ldots$ all vanish). These combinatorial objects will be
defined in Section~\ref{sec:basic}. The notations are chosen to be
consistent with those of~\cite{threepoint,loops}: as we discuss in
Appendix~\ref{sec:wlm}, bigeodesic diangles and triangles are the
slice-theoretic equivalents of objects appearing in the decomposition
of well-labeled maps (the slice interpretation of $R$ being already
known). This makes a connection with the bijective approach developed
in~\cite{Miermont09,AmBu13,gen2p}. Comparing~\eqref{eq:Tabcter} with
the known expression for $T_{0,0,0}$, we get $X^3 Y^2/t^6=d\ln R/dt$
and Theorem~\ref{thm:Tabc} follows.

\paragraph{Outline.} Section~\ref{sec:basic} introduces the basic
building blocks of our approach, namely tight slices, bigeodesic
diangles and bigeodesic triangles, and derives some elementary
enumeration results for those pieces.  Section~\ref{sec:gluing}
explains how these pieces can be assembled to produce a map with three
tight boundaries. This allows us to state our main bijective result,
Theorem~\ref{thm:main}. To prove this theorem, the difficult part is
to decompose a map with three tight boundaries back into basic
building blocks.  This decomposition takes place on the universal
cover of the map, which is a periodic infinite map which we describe
in Section~\ref{sec:universal-cover-maps}.  Section~\ref{sec:decface}
then presents the decomposition of a map with three tight boundaries,
by first introducing the important geometric tool of Busemann
functions associated with infinite geodesics, and finishes the proof
of Theorem~\ref{thm:main} by showing that this decomposition is indeed
the inverse of the assembling procedure. Section~\ref{sec:ecf}
discusses how to recover the ECF formula from our results and a
decomposition of pairs of pants into annular maps and tight pairs of
pants, and then states and proves our probabilistic applications on
the statistics of minimal separating cycles and areas in large random
quadrangulations with three boundaries. Concluding remarks and
discussion on future directions are gathered in
Section~\ref{sec:conc}. Finally, we recall in
Appendix~\ref{sec:slicetp} how to obtain the classical recursion
relation~\eqref{eq:Req} for slices, as well as the reason why
$d \ln(R/t)/ dt$ is the generating function of triply pointed maps,
and in Appendix~\ref{sec:wlm} we present another approach to
bigeodesic diangles and triangles in the case of quadrangulations,
based on a bijection with labeled trees.

\paragraph{Acknowledgements.} We thank Marie Albenque, Timothy Budd,
Vincent Delecroix, Marco Mazzucchelli and Bram Petri for valuable
discussions. We also thank the two anonymous referees for suggesting
useful improvements to the paper. This project results from an
institutional collaboration between CEA and ENS de Lyon, and was
initiated at the occasion of the \emph{Séminaire de combinatoire de
  Lyon à l'ENS} which is funded by the Labex Milyon
(ANR-10-LABX-0070).  The work of JB is partly supported by the Agence
Nationale de la Recherche via the grants ANR-18-CE40-0033 ``Dimers''
and ANR-19-CE48-0011 ``Combiné''.

\section{Basic building blocks}
\label{sec:basic}

In this section we introduce the fundamental building blocks of our
approach. We start with some preliminary definitions.

\subsection{Preliminaries:  geodesics and related concepts}
\label{sec:geoddefs}

In a map, a \emph{geodesic} between two vertices $v_1$ and $v_2$ is a
path of minimal length connecting them. This minimal length is by
definition the (graph) \emph{distance} $d(v_1,v_2)$ between the
vertices $v_1$ and $v_2$. Maps are assumed to be connected, so geodesics
between any two given vertices always exist and $d(v_1,v_2)$ is a
finite integer.

A \emph{geodesic vertex} between $v_1$ and $v_2$ is a vertex $v$
belonging to a geodesic between them. Clearly, $v$ is a geodesic
vertex if and only if
\begin{equation}
  d(v,v_1) + d(v,v_2) = d(v_1,v_2).
\end{equation}

\paragraph{Bigeodesics.}

In general, there may exist many geodesics between $v_1$ and
$v_2$. But, using the (local) planar structure of a
map, it is often possible to single out a canonical one. The previous
works on slice decomposition were using the notion of leftmost
geodesic determined by the choice of an initial direction at
$v_1$. Here, we will need a related but slightly different notion,
which is that of leftmost bigeodesic determined by the choice of a
geodesic vertex between $v_1$ and $v_2$.

A \emph{bigeodesic} between two vertices $v_1$ and $v_2$ is a triple
made of a geodesic vertex $v$ between them and of two geodesics, one
between $v$ and $v_1$ and one between $v$ and $v_2$. Clearly the
concatenation of these two geodesics is a geodesic between $v_1$ and
$v_2$, so that a bigeodesic between $v_1$ and $v_2$ is entirely
specified by the data of a geodesic between $v_1$ and $v_2$ and of a
vertex $v$ along this geodesic.

Viewing the bigeodesic as ``launched'' from the geodesic vertex $v$
towards $v_1$ and $v_2$ respectively, we may introduce the notion of
\emph{leftmost bigeodesic} as follows. Assume that
$d(v_1,v_2) \geq 2$ and that $v$ is distinct from $v_1$ and $v_2$. We
may partition the set of edges incident to $v$ into three types:
\begin{itemize}\setlength{\itemsep}{0pt}
\item[(i)] those leading to a vertex strictly closer to $v_1$,
\item[(ii)] those leading to a vertex strictly closer to $v_2$,
\item[(iii)] those leading to a vertex that is neither strictly closer to $v_1$ nor to $v_2$.
\end{itemize}
Ignoring the edges of type (iii), it is easily seen that, by
planarity, there exists an edge $e_1$ of type (i) and an edge $e_2$ of
type (ii) such that, when turning clockwise around $v$, all edges of
type (i) appear between $e_1$ and $e_2$ and all edges of type (ii)
appear between $e_2$ and $e_1$. We then consider the leftmost geodesic
from $v$ to $v_1$ starting with $e_1$, i.e.\ the geodesic whose first
step goes along $e_1$ from $v$ to its neighbor at distance
$d(v,v_1)-1$ from $v_1$, and at each step, goes along the leftmost
edge (as viewed from the previous edge) among all those going from the
currently attained vertex to a vertex closer to $v_1$, until $v_1$ is
eventually reached. We similarly pick the leftmost geodesic from $v$
to $v_2$ starting with $e_2$. Concatenating these two geodesics, we
obtain a bigeodesic between $v_1$ and $v_2$ launched from $v$, which
is by definition the leftmost bigeodesic we are looking for. Observe
that the leftmost bigeodesic is not well-defined if $v=v_1$ or
$v=v_2$, and in particular if $d(v_1,v_2)<2$.

\paragraph{Geodesic boundary intervals.}

Consider a planar map with one boundary-face, which we denote $f_0$
and which we choose as the external face in the planar representation
of the map. Let $c$ and $c'$ be two corners incident to $f_0$. These
corners split the contour of $f_0$ in two portions, which we call
\emph{boundary intervals}. When turning counterclockwise around the
map (i.e., when walking along the contour with $f_0$ on the right),
the portion that starts at $c$ and ends at $c'$ is denoted
$[c,c']$. It forms a path on the map, which may not be simple in
general. In the particular case where it forms a geodesic between the
vertices incident to $c$ and $c'$, then we say that the boundary
interval $[c,c']$ is \emph{geodesic} (the path is necessarily simple
in this case). Furthermore, if there exists no other geodesic in the
map with the same endpoints, then $[c,c']$ is said \emph{strictly
  geodesic}. In the figures, we will often use the graphical
convention of representing geodesic boundary intervals in blue, and
strictly geodesic boundary intervals in red.

\subsection{Tight slices}\label{sec:tightslices}

\begin{figure}
  \centering
  \fig{.7}{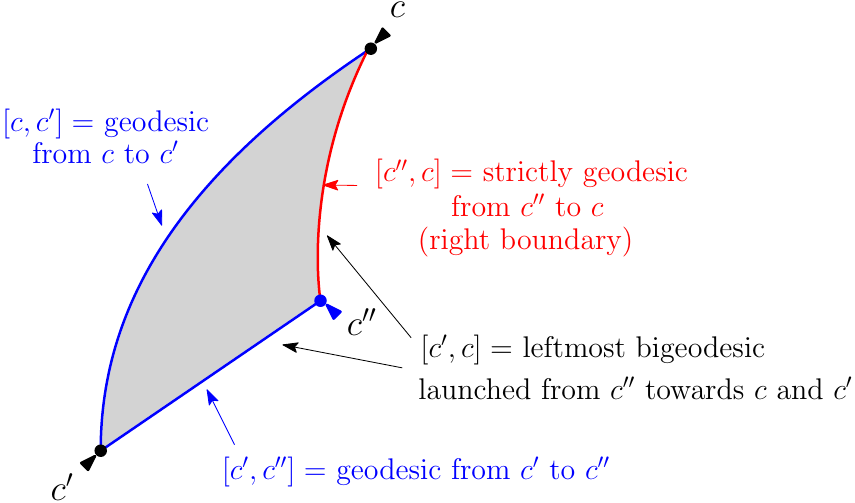}
  \caption{Generic structure of a tight slice (the boundary-face is
    the outer face). In this figure and the following, we represent
    geodesic boundary intervals in blue and strictly geodesic ones in
    red. The corner $c''$ is shown in blue since it is an intermediate
    corner on the geodesic boundary interval $[c',c]$. The width of
    the slice is the length of $[c',c'']$.}
  \label{fig:tightslice}
\end{figure}

Our first building block is what we call a \emph{tight slice}, defined
as a planar bipartite map with one boundary-face having three
distinguished (not necessarily distinct) incident corners $c$, $c'$
and $c''$ appearing counterclockwise around the map
such that:
\begin{itemize}[label=$\scriptstyle{\circ}$]\setlength{\itemsep}{0pt}
\item  the boundary intervals $[c,c']$ and $[c',c]$ are geodesic,
\item the boundary interval $[c'',c]$, called the \emph{right
    boundary}, is strictly geodesic,
\item  the intervals $[c'',c]$ and $[c,c']$ share only a common vertex at $c$.
\end{itemize}
See Figure~\ref{fig:tightslice} for an illustration.
Note that the constraints imply that $[c',c'']$ is also geodesic, and
the length of this interval is called the \emph{width} of the
slice. The only tight slice of width zero is equal to the
\emph{vertex-map}, reduced to a single vertex and a single face both of
degree zero. Note that, if $c\neq c''$ (hence $c' \neq c''$), then
$[c',c]$ is nothing but the leftmost bigeodesic launched from the
vertex incident to $c''$ towards those incident to $c$ and $c'$
respectively. A tight slice of width $1$ is called an \emph{elementary slice}.

\begin{prop}
  \label{prop:proptightslice}
  The generating function of tight slices of width $\ell$, counted
  with a weight $t$ per vertex not incident to the right boundary and
  a weight $g_{2k}$ per inner face of degree $2k$ for all $k \geq 1$,
  is equal to $R^\ell$, with $R$ defined as in Theorem~\ref{thm:Tabc}.
\end{prop}

\begin{proof}
  It is known that $R$ counts elementary slices (for
  completeness, we provide a proof in
  Appendix~\ref{sec:slicetp}). Given a tight slice of arbitrary width,
  let us consider all the vertices incident to the boundary interval
  $]c',c''[$ (with endpoints excluded), and the leftmost bigeodesics
  launched from them towards the vertices incident to $c'$ and
  $c$. These bigeodesics necessarily follow the boundary towards
  $c'$, but may enter inside the map towards $c$. Cutting also these
  bigeodesics splits the map into a $\ell$-tuple of elementary slices, and
  it is straightforward to check that the decomposition is bijective.
\end{proof}

\subsection{Bigeodesic diangles}\label{sec:bigeo-diangles}

\begin{figure}
  \centering
  \fig{.25}{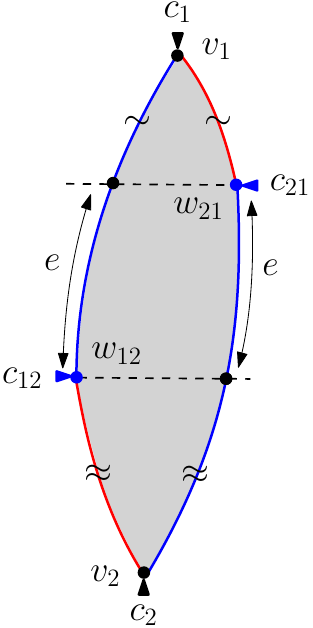}
  \caption{Schematic picture of a bigeodesic diangle of nonnegative
    exceedance $e$.  The boundary intervals with the $\sim$ label have
    the same length and meet only at $v_1$, and similarly for the
    $\approx$ label. The remaining two other boundary intervals, of
    length $e$, may however share common vertices as soon as $e>0$.}
  \label{fig:diangle}
\end{figure}

Our second building block is what we call a \emph{bigeodesic
  diangle}, or a \emph{diangle} for short, which is again a planar
bipartite map with one boundary-face, with now four distinguished (not
necessarily distinct) incident corners $c_1$, $c_{12}$, $c_2$,
$c_{21}$ appearing counterclockwise around the map, and having the
following properties:
\begin{itemize}[label=$\scriptstyle{\circ}$]\setlength{\itemsep}{0pt}
\item the boundary intervals $[c_1,c_2]$ and $[c_2,c_1]$ are geodesic, 
\item the boundary intervals $[c_{12},c_2]$ and $[c_{21},c_1]$ are strictly geodesic,
\item $[c_{21},c_1]$ and $[c_1,c_2]$ share only a common vertex
at the vertex $v_1$ incident to $c_1$ and similarly, $[c_{12},c_2]$ and $[c_2,c_1]$ share only a common vertex
at the vertex $v_2$ incident to $c_2$.
\end{itemize}
See Figure \ref{fig:diangle} for an illustration.
Note that the boundary intervals $[c_1,c_{12}]$ and $[c_2,c_{21}]$ are necessarily geodesic.
Let us denote by $w_{12}$ and $w_{21}$ the vertices incident to $c_{12}$ and $c_{21}$. If $w_{12}$ is different from $v_1$ and $v_2$, then $[c_1,c_2]$ is the leftmost bigeodesic launched from $w_{12}$ towards $v_1$ and $v_2$. Similarly, if $w_{21}$ is different from $v_1$ and $v_2$, then $[c_2,c_1]$ is the leftmost bigeodesic launched from $w_{21}$. The corners $c_{12}$ and $c_{21}$ (or the vertices $w_{12}$ and $w_{21}$ depending on the context)
will be referred to as the \emph{attachment points} of the diangle,
for reasons which will become clear in the next section. 

\medskip
The \emph{exceedance} $e$ of the diangle is defined as
 \begin{equation}
  e=d(w_{12},v_1) - d(w_{21},v_1) = d(w_{21},v_2)-d(w_{12},v_2).
\end{equation}
A diangle is said \emph{balanced} if its exceedance $e$ is $0$, that
is if $w_{12}$ and $w_{21}$ are at the same distance from $v_1$, say.
Note that the vertex-map again satisfies all the required criteria for
a balanced diangle. Apart from this trivial case, any other balanced
diangle has $d(v_1,v_2)\geq 2$, with $w_{12}$ and $w_{21}$ different
from $v_1$ and $v_2$ and from each other, with $[c_1,c_2]$ and
$[c_2,c_1]$ meeting only at $v_1$ and $v_2$. This latter property is
not necessarily true for unbalanced diangles, as indicated in the caption of
Figure~\ref{fig:diangle}.

\medskip
We denote by $X$ the generating function of balanced bigeodesic diangles, where the boundary-face and the vertices incident
to the strictly geodesic intervals $[c_{12},c_2]$ and $[c_{21},c_1]$ other than $w_{12}$ and $w_{21}$ receive no weight.
Note that the vertex-map contributes a weight $t$ to $X$. 

\begin{figure}
  \centering
  \fig{.8}{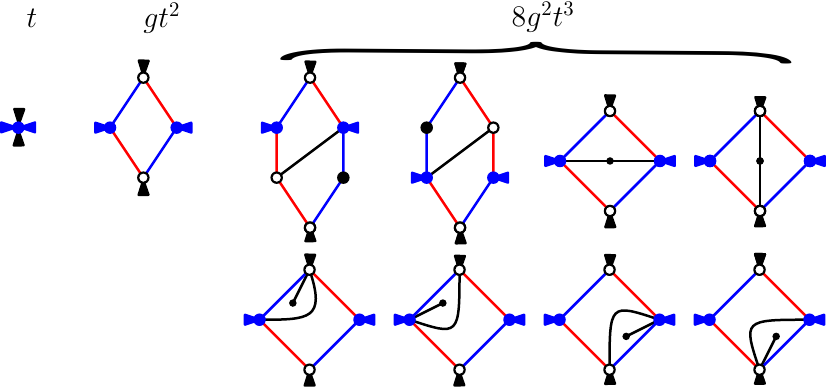}
  \caption{The first terms in the expansion of $X$ for quadrangulations ($g_{2k}=g\, \delta_{k,2}$). The unfilled vertices (circles) receive no weight.}
  \label{fig:firstX}
\end{figure}
Even though we shall not need any precise expression for $X$, let us
mention that, by the results of~\cite{threepoint}, a very explicit
formula can be given in the case of quadrangulations
($g_{2k}=g\, \delta_{k,2}$), see equation~\eqref{eq:XYform} in
Appendix~\ref{sec:wlm}. It yields the expansion
$X=t+gt^2 +8 g^2 t^3 +73 g^3 t^4 +711 g^4 t^5 +\ldots$, and the maps
corresponding to the first terms of this expansion are displayed in
Figure~\ref{fig:firstX}.

\medskip
With the same weighting convention as for balanced diangles, we have the following property:
\begin{prop}
  \label{prop:diangleenum}
  The generating function of bigeodesic diangles with nonnegative
  exceedance $e$ is equal to $R^e X$.
\end{prop}
\begin{proof}
\begin{figure}
  \centering
  \fig{.65}{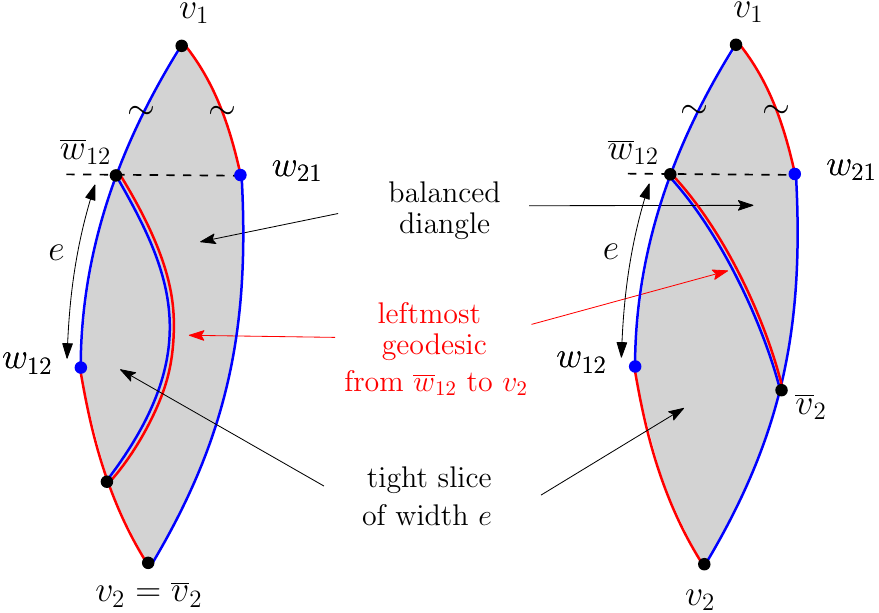}
  \caption{Decomposition 
of a bigeodesic diangle of exceedance $e$ into a pair made of a balanced bigeodesic diangle
and a tight slice of width $e$. Different situations occur according to which boundary is first hit by the lefmost
geodesic from $\overline{w}_{12}$ to $v_2$.  
}
  \label{fig:diangledecomp}
\end{figure}
The property is obvious for $e=0$, so we may assume $e>0$.
Let $\overline{w}_{12}$ denote the vertex along $[c_1,c_2]$ which is at the same distance from $v_1$ as $w_{21}$. We have $d(w_{12},\overline{w}_{12})=e$.
Consider the leftmost bigeodesic launched from $\overline{w}_{12}$
towards $v_1$ and $v_2$: its part towards $v_1$ follows the boundary,
while its part towards $v_2$ may enter inside the diangle. We denote
by $\overline{v}_2$ the first vertex common to this part and
$[c_2,c_1]$. Then, as illustrated in Figure~\ref{fig:diangledecomp},
the bigeodesic splits the map into two pieces. One of them is
a balanced diangle with distinguished corners incident
to $v_1$, $\overline{w}_{12}$, $\overline{v}_2$ and $w_{21}$,
and the other is a tight slice of width $e>0$.
The decomposition is clearly a bijection and implies the wanted
expression, by Proposition~\ref{prop:proptightslice} (note that the
weight $t$ for $w_{12}$ must be transfered to $\overline{w}_{12}$ in the tight slice).
\end{proof}

When all face weights are set to zero, $R^e X$ is equal to $t^{e+1}$,
which accounts for the diangle made of a chain of $e+1$ vertices and $e$ edges, with
$v_1=w_{21}$ at one extremity and $v_2=w_{12}$ at the other.

It is interesting to note that a tight slice of (positive) width $e$ is nothing but a diangle of
exceedance $e$ for which $c_{12}=c_2$ (with the correspondance $c=c_1$, $c'=c_{12}=c_2$, $c''=c_{21}$).
Note however that the weighting conventions differ slightly (there is an extra weight $t$ for diangles).

We will not consider diangles with negative exceedance in this paper,
since their enumeration is more subtle.

\subsection{Bigeodesic triangles}\label{sec:bigeo-triangle}

\begin{figure}
  \centering
  \fig{.45}{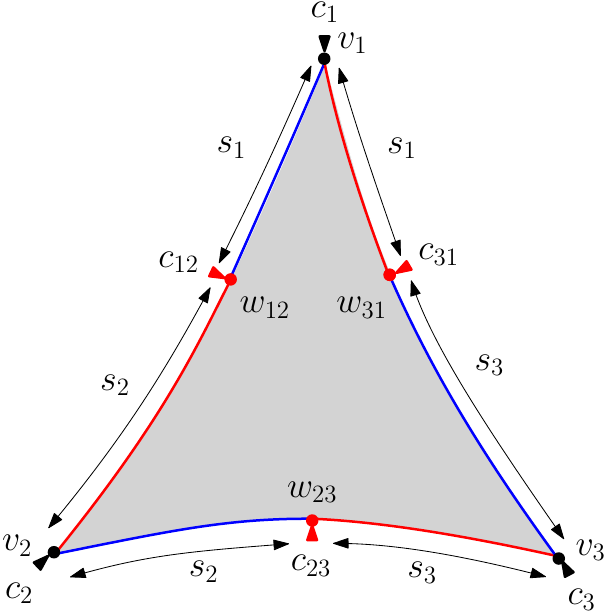}
  \caption{Schematic picture of a bigeodesic triangle. The vertices $w_{12}$, $w_{23}$ and $w_{31}$ are colored in red to indicate that
any geodesic from $v_1$ to $v_2$ (respectively from $v_2$ to $v_3$, from $v_3$ to $v_1$) must pass via $w_{12}$ (respectively $w_{23}$, $w_{31}$).
}
  \label{fig:triangle}
\end{figure}

Our third and final building block is what we call a \emph{bigeodesic
  triangle}, or \emph{triangle} for short, which is again a planar bipartite map with one boundary-face.  It
now has six distinguished incident corners $c_1$, $c_{12}$, $c_2$,
$c_{23}$, $c_3$, $c_{31}$ appearing counterclockwise around the map,
and having the following properties:
\begin{itemize}[label=$\scriptstyle{\circ}$]\setlength{\itemsep}{0pt}
\item the boundary intervals $[c_1,c_2]$, $[c_2,c_3]$ and $[c_3,c_1]$ are geodesic, with no common vertex except
at their endpoints $v_1$, $v_2$, and $v_3$ (incident to $c_1$, $c_2$ and $c_3$, respectively), 
\item the boundary intervals $[c_{12},c_2]$, $[c_{23},c_3]$ and $[c_{31},c_1]$ are strictly geodesic,
\item the boundary intervals $[c_1,c_{12}]$ and $[c_{31},c_1]$ (respectively $[c_2,c_{23}]$ and $[c_{12},c_2]$, $[c_3,c_{31}]$ and $[c_{23},c_3]$)
have the same length $s_1$ (respectively $s_2$, $s_3$).
\item any geodesic from $v_1$ to $v_2$ (respectively from $v_2$ to $v_3$, from $v_3$ to $v_1$) passes via $w_{12}$ (respectively $w_{23}$, $w_{31}$),
the vertex incident to $c_{12}$ (respectively $c_{23}$, $c_{31}$).
\end{itemize}
See Figure~\ref{fig:triangle} for an illustration. Note that, if two
corners among $c_1$, $c_{12}$, $c_2$, $c_{23}$, $c_3$, $c_{31}$ are
equal, then the above properties imply that the triangle is reduced to
the vertex-map.  Note also that the boundary intervals $[c_1,c_{12}]$,
$[c_2,c_{23}]$ and $[c_3,c_{31}]$ are necessarily geodesic and that
$d(v_1,v_2)=s_1+s_2$, $d(v_2,v_3)=s_2+s_3$ and $d(v_3,v_1)=s_3+s_1$.
As in the case of diangles, the interval 
$[c_1,c_2]$ is the leftmost bigeodesic launched from $w_{12}$ towards $v_1$ and $v_2$.
Note also that a geodesic from $v_1$ to $v_2$, which has to pass via $w_{12}$, necessarily sticks to $[c_{12},c_2]$
between $w_{12}$ and $v_2$ since $[c_{12},c_2]$ is strictly geodesic. Similar properties hold under cyclic permutations
of the indices $1,2,3$.
The corners $c_{12},c_{23},c_{31}$ (or the vertices $w_{12},w_{23},w_{31}$ depending on the context)
will be referred to as the \emph{attachment points} of the triangle. 

\medskip
We call $Y$ the generating function of bigeodesic triangles, where the boundary-face  and the vertices incident
to the strictly geodesic intervals $[c_{12},c_2]$, $[c_{23},c_3]$ and $[c_{31},c_1]$ other than $w_{12}$, $w_{23}$ and $w_{31}$ receive no weight.
The vertex-map contributes a term $t$ to $Y$.
\begin{figure}
  \centering
  \fig{.8}{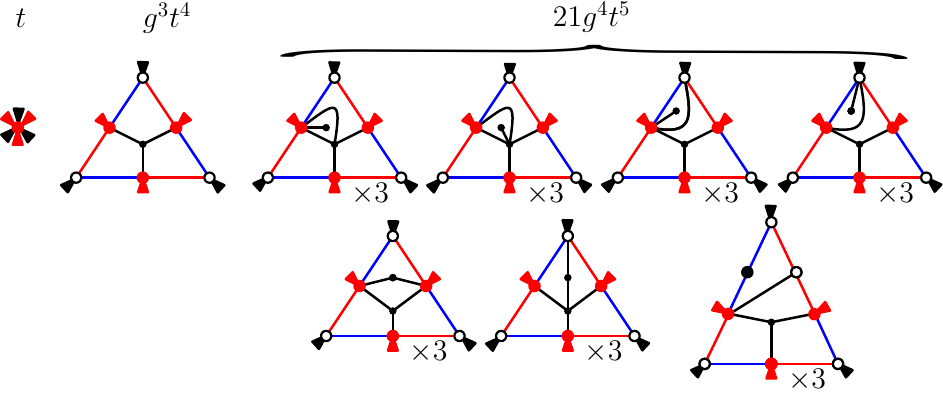}
  \caption{The first terms in the expansion of $Y$ for quadrangulations ($g_{2k}=g\, \delta_{k,2}$). The unfilled vertices (circles) receive no weight.}
  \label{fig:firstY}
\end{figure}

As was the case for the generating function $X$ of diangles, we shall
not need any precise expression for $Y$, even though the results
of~\cite{threepoint} provide an explicit formula in the case of
quadrangulations ($g_{2k}=g\, \delta_{k,2}$), see again
equation~\eqref{eq:XYform} in Appendix~\ref{sec:wlm}. It yields the
expansion $Y=t+g^3 t^4 +21 g^4 t^5 +324 g^5 t^6+\ldots$, and the maps
corresponding to the first terms of this expansion are displayed in
Figure~\ref{fig:firstY}.

\section{Assembling the building blocks}
\label{sec:gluing}

We now explain how to assemble a map with three tight boundaries
from the basic building blocks introduced in the previous section. We
start from a quintuple consisting of the following pieces:
\begin{itemize}[label=$\scriptstyle{\circ}$]\setlength{\itemsep}{0pt}
\item{three bigeodesic diangles of nonnegative exceedances denoted by $e_1$, $e_2$ and $e_3$,}
\item{two bigeodesic triangles.}
\end{itemize}
Recall that the boundaries of bigeodesic diangles and triangles are
conventionally colored in red and blue: a boundary edge is colored red
if it belongs to a boundary interval that is constrained to be
strictly geodesic, and is colored blue otherwise (i.e.\ it belongs to
a geodesic boundary interval which is not necessarily strictly
geodesic). Generally speaking, the assembling procedure consists in
gluing the boundaries of the different pieces together, a red edge
being always glued to a blue edge. Some blue edges will possibly
remain unmatched and form the boundaries of the assembled map.

We shall discuss in fact two alternative assembling procedures,
hereafter numbered I and II, which are complementary in the sense that
each of them generates only a strict subset of the set of maps with
three tight boundaries but, taken together, they generate the full
set.

\begin{figure}
  \centering
  \fig{.8}{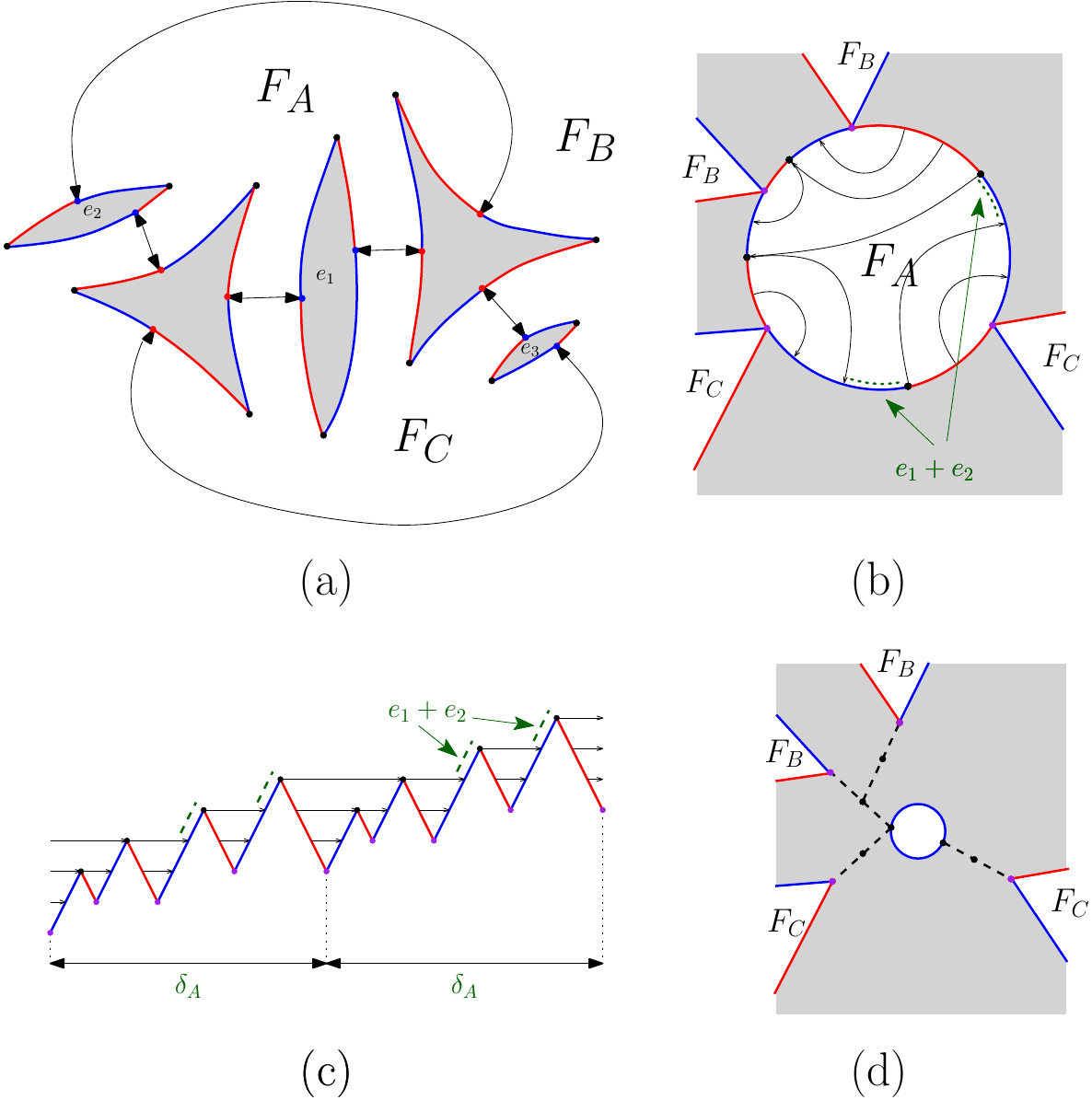}
  \caption{The assembling procedure of type I. (a) We first identify
    the attachment points of the two triangles and the three diangles
    as shown so as to create a planar map with three special faces
    $F_A$, $F_B$ and $F_C$. (b) The red-to-blue gluing around the
    special face $F_A$: the arrows indicate the resulting
    identification between vertices (the identified attachment points
    are shown in purple). Here, $e_1+e_2=2$ blue edges remain
    unmatched. (c) Alternate representation of the gluing process: the
    cyclic sequence of edges counterclockwise around $F_A$ may be
    coded by the quasi-periodic lattice path with a unit up
    (respectively down) step for each blue (respectively red) edge
    (the increments of the path are denoted $\epsilon_n$ in the
    text). Each down (red) step is then matched to the next up (blue)
    step on its right which returns to the same height. A number
    $e_1+e_2$ of blue edges remain unmatched for each period
    $\delta_A=\mathrm{deg}(F_A)$. (d) Result of the gluing of $F_A$:
    we obtain a boundary of length $e_1+e_2$ (the glued edges are
    represented by dashed lines). The red-to-blue gluing of $F_B$ and
    $F_C$ is performed similarly.}
  \label{fig:assemblingI}
\end{figure}

\subsection{Description of the assembling procedures I and II}
\label{sec:gluingdesc}

In a nutshell, the assembling procedures consist of two operations,
which we call \emph{attachment} and \emph{red-to-blue gluing}. We
start by describing these operations in detail in the case of
procedure I, referring again to Figure \ref{fig:assemblingsample} for an illustration, then turn to procedure II which only differs at the level
of the attachment operation.

\paragraph{Procedure I.}

Recall from Sections~\ref{sec:bigeo-diangles} and
\ref{sec:bigeo-triangle} that the attachment points of diangles and
triangles are the vertices from which their bigeodesic boundaries are
launched. The \emph{attachment} operation consists in identifying the
attachment points as shown on Figure~\ref{fig:assemblingI}(a). The
resulting object is a planar map which, in addition to the inner faces
of the initial triangles and diangles, has three extra \emph{special}
faces, hereafter denoted $F_A$, $F_B$ and $F_C$. Each special face is
incident to four attachment points (after identification) and its
contour is made of alternating blue and red intervals, four of each
color, with an excess of blue edges ($e_1+e_2$ for $F_A$, $e_2+e_3$
for $F_B$, $e_3+e_1$ for $F_C$).

The second operation is performed independently on each special face,
and consists in gluing all its incident red edges to blue edges so as
to form a face of smaller degree with only blue incident edges (thus
the term~\emph{red-to-blue gluing}). More precisely, consider a
special face, say $F_A$, and follow its contour keeping the face on
the left (i.e.\ we turn counterclockwise around $F_A$, when it is
represented as a bounded face in the plane): each red edge immediately
followed by a blue edge is glued to it, and we repeat the process
until no red edge is left. A convenient global description of this
operation can be given as follows: number the edges along the contour
of $F_A$ by integers, starting at an arbitrary position and in the
same direction as before. We set $\epsilon_n=1$ if the edge numbered
$n$ is blue, and $\epsilon_n=-1$ if it is red. The sequence
$(\epsilon_n)_n$ is naturally defined for all $n \in \Z$ by
periodicity, with a period $\delta_A$ equal to the degree of
$F_A$. Then, a red edge at position $k$ will be matched and glued to
the blue edge at position $\ell$, where $\ell$ is the smallest integer
larger than $k$ such that $\sum_{n=k}^\ell \epsilon_n=0$. Such $\ell$
necessarily exists since $\epsilon_k=-1$ and
$\sum_{n=k}^{k+\delta_A-1} \epsilon_n =e_1+e_2 \geq 0$.  See
Figure~\ref{fig:assemblingI} for an illustration. After the gluing
step, a number $e_1+e_2$ of blue edges remain unmatched. If
$e_1+e_2>0$, these edges form the contour of a boundary-face. If
$e_1+e_2=0$, we instead obtain a boundary-vertex, which corresponds to
the vertex preceding any edge $k$ such that
$\sum_{n=k}^\ell \epsilon_n \leq 0$ for all $\ell \geq k$.  Indeed,
such vertices corresponds to the maxima of the lattice path of
Figure~\ref{fig:assemblingI}(c)---which is periodic when
$e_1+e_2=0$---and all of them are identified by gluing.  All in all,
$F_A$ becomes a boundary of length $e_1+e_2$, and performing the same
operation on $F_B$ and $F_C$, these become boundaries of lengths
$e_2+e_3$ and $e_3+e_1$ respectively.

The assembling procedure is trivially adapted to the case where a
triangle is reduced to the vertex-map, by equipping its unique vertex
with three attachment points (dividing the surrounding corner in three
sectors). Similarly, when one of the exceedances is $0$ and the
corresponding (balanced) diangle is reduced to the vertex-map, we
equip its unique vertex with two attachment points. If $e_1=e_2=e_3=0$
and all the triangles/diangles are reduced to the vertex-map, the
object resulting from the assembling procedure is the vertex-map
itself. Besides this pathological case, we always obtain a map in
which the three boundaries are distinct elements of the map.

\paragraph{Procedure II.}

\begin{figure}
  \centering
  \fig{.65}{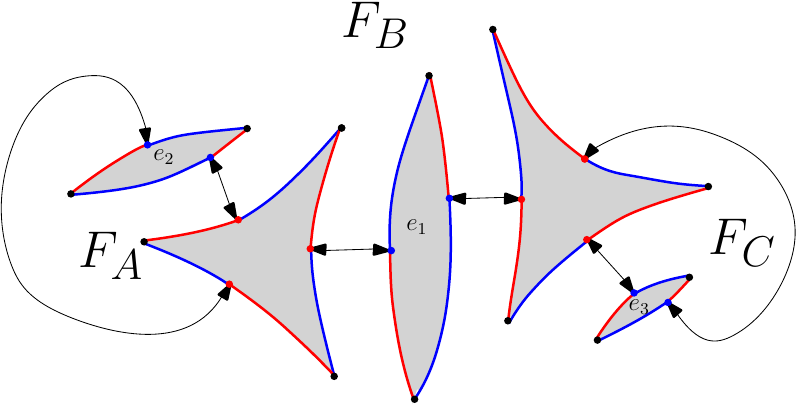}
  \caption{The first step of the assembling procedure of type II,
    creating a planar map with three special faces $F_A$, $F_B$ and $F_C$.} 
  \label{fig:assemblingII}
\end{figure}

It differs from the previous one only by the identification between
attachment points, following now the prescription of
Figure~\ref{fig:assemblingII}.  The resulting object is still a planar
map with three special faces: two of them ($F_A$ and $F_C$) are now
incident to two attachment points (after identification) and the third
one ($F_B$) to eight attachment points. They all have a boundary made
of alternating blue and red intervals.  We then repeat the red-to-blue
gluing operation described in the case I, creating three boundaries of
respective lengths $e_2$ (from $F_A$), $e_3$ (from $F_C$) and
$2e_1+e_2+e_3$ (from $F_B$).  Again, the assembling procedure is
trivially extended to the case where some of the triangles/diangles
are reduced to the vertex-map, and the three boundaries are distinct
as soon as at least one of the building blocks is non
trivial.

\subsection{Properties of the assembling procedures}
\label{sec:gluingprops}

Let us first
observe that, for both assembling procedures, the resulting map is
essentially bipartite: indeed its inner faces are those of the initial
triangles and diangles, with no modification of the degrees.  The
boundary lengths are given by:
\begin{equation}
  \label{eq:boundlengths}
  \text{procedure I: }
  \begin{cases}
    2a=e_1+e_2 \\
    2b=e_2+e_3 \\
    2c=e_3+e_1
  \end{cases}
  \qquad
  \text{procedure II: }
  \begin{cases}
    2a=e_2 \\
    2b=2e_1+e_2+e_3  \\
    2c=e_3
  \end{cases}
\end{equation}
Here, $a$, $b$ and $c$ may either be integers or half-integers. They
are all integers (i.e., the map is bipartite) if $e_1$, $e_2$ and
$e_3$ have the same parity in case I, or if $e_2$ and $e_3$ are even
in case II. Otherwise, two of them are half-integers and the third is
an integer, and the map is quasi-bipartite.

We also see that, since $e_1$, $e_2$ and $e_3$ are assumed
nonnegative, we have the ``triangle inequality'' $b\leq c+a$ in case
I, while we have $b \geq c+a$ in case II (in both cases, we have
$a\leq b+c$ and $c\leq a+b$). Upon permuting $a$, $b$ and $c$ in case
II, it is possible to obtain any possible triple of boundary lengths
in a bipartite or quasi-bipartite map with three boundaries.

This suggests to introduce the following definition: a map with three
boundaries is said \emph{of type I (respectively II)} if the largest
of the three boundary lengths is smaller than or equal to
(respectively larger than or equal to) the sum of the two other
boundary lengths.  Clearly a map with three boundaries is either of
type I or type II, and may be both in the equality case. We may now
state the main bijective result of this paper, illustrated by Figure
\ref{fig:assemblingsample} in the case of procedure I. 

\begin{thm}
  \label{thm:main}
  For $e_1$, $e_2$ and $e_3$ fixed nonnegative integers, the
  assembling procedure I (respectively II) is a bijection between the
  set of quintuples made of two bigeodesic triangles and three
  bigeodesic diangles of nonnegative exceedances $e_1$, $e_2$ and
  $e_3$, where at least one element differs from the vertex-map, and
  the set of essentially bipartite planar maps with three tight
  boundaries of type I (respectively II), where the lengths of the
  boundaries are given by~\eqref{eq:boundlengths}.
\end{thm}

\begin{figure}
  \centering
  \fig{.55}{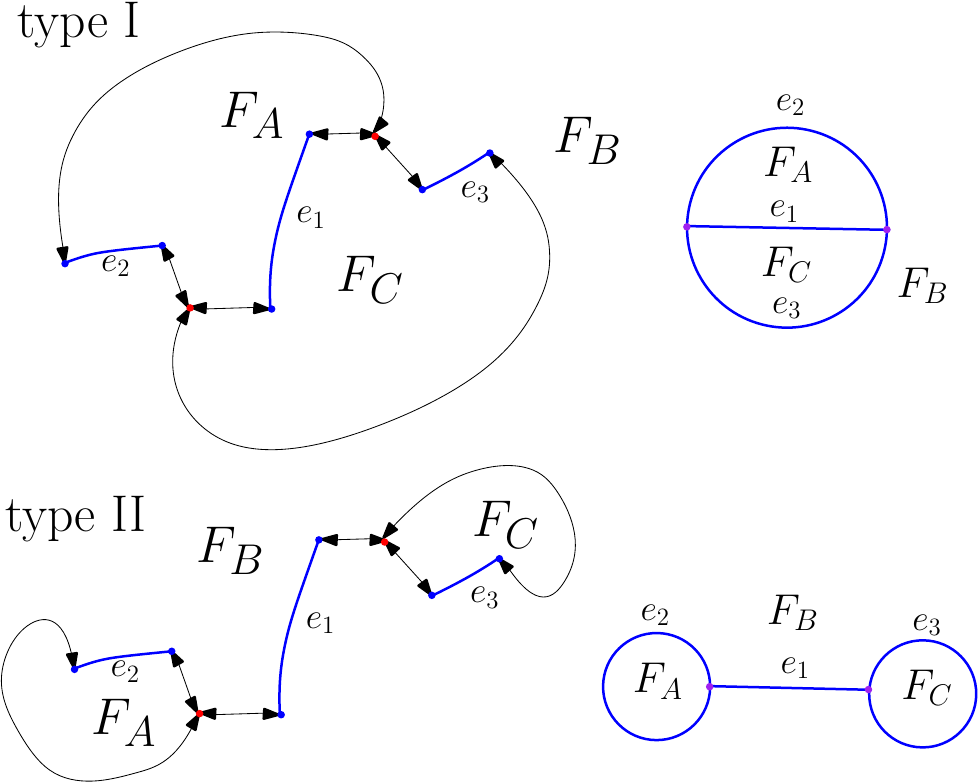}
  \caption{When the triangles and diangles have no inner faces, the
    assembling procedures I and II generate maps corresponding to the
    generic cases I and II of Figure~\ref{fig:gzero}.}
  \label{fig:assembling_noface}
\end{figure}

It is instructive to examine the case where the triangles and diangles
have no inner faces: the triangles are then reduced to the vertex-map,
while the diangles are segments of lengths $e_1$, $e_2$ and $e_3$ (see
Figure~\ref{fig:assembling_noface}).
When these lengths are all positive, we recover precisely from the
assembling procedures I and II the two generic maps I and II displayed
in Figure~\ref{fig:gzero}. The degenerate cases correspond to having
one or two lengths vanish, and the case $e_1=e_2=e_3=0$ (all diangles
reduced to the vertex-map) is pathological when there are no inner faces.

To establish Theorem~\ref{thm:main}, two statements remain to be
proved. First, we need to show that the boundaries of the maps
resulting from the assembling procedures are indeed tight: this will
be done in Section~\ref{sec:universal-cover-maps}, see
Proposition~\ref{prop:tightness}. Second, we must check that both
procedures are bijections: for this, we will exhibit the inverse
bijections in Section~\ref{sec:decface}, see
Propositions~\ref{prop:typeIinv} and \ref{prop:typeIIinv}. Both proofs
are most conveniently performed by considering the universal cover of
a map with three boundaries, which we introduce in the next
section. Before this, let us explain why Theorem~\ref{thm:main}
implies Theorem~\ref{thm:Tabc}, and also sketch the inverse of
procedure I in the simpler case of maps with three boundary-vertices
($e_1=e_2=e_3=0$).

\subsection{Enumerative consequences}
\label{sec:gluingenum}

Recall that $T_{a,b,c}$ denotes the generating function of essentially
bipartite planar maps with three tight boundaries of lengths $2a$,
$2b$, $2c$, with the weighting convention of
Theorem~\ref{thm:Tabc}. As discussed in Section~\ref{sec:intro}, this
theorem is implied by the following corollary of
Theorem~\ref{thm:main}.

\begin{cor}
  \label{cor:main}
  We have
  \begin{equation}
    \label{eq:corenum}
    T_{a,b,c} =R^{a+b+c} \frac{X^3 Y^2}{t^6} - t^{-1} \mathbf{1}_{a=b=c=0}
  \end{equation}
  where $X$ and $Y$ are the generating functions of bigeodesic
  balanced diangles and triangles, respectively, as defined in
  Section~\ref{sec:basic}.
\end{cor}
 
\begin{proof}
  By Proposition~\ref{prop:diangleenum}, the generating function of
  quintuples made of two bigeodesic triangles and three bigeodesic
  diangles of nonnegative exceedances $e_1$, $e_2$ and $e_3$, is equal
  to $R^{e_1+e_2+e_3} X^3 Y^2$. We exclude the quintuple with all
  elements reduced to the vertex-map by subtracting a term
  $t^5 \mathbf{1}_{e_1=e_2=e_3=0}$.

  We then apply Theorem~\ref{thm:main} and note that, by
  \eqref{eq:boundlengths}, we have $e_1+e_2+e_3=a+b+c$ and
  $\mathbf{1}_{e_1=e_2=e_3=0}=\mathbf{1}_{a=b=c=0}$ both in procedures
  I and II. The claim then follows from the fact that the weight of a
  quintuple, defined according to the conventions of
  Section~\ref{sec:basic}, is equal to $t^6$ times that of the
  corresponding map, defined as in Theorem~\ref{thm:Tabc}. This is
  clear for face weights since the inner faces are not modified by the
  bijection. For vertex weights, the corrective factor $t^6$ comes
  from the identifications between attachment points, see again
  Figures~\ref{fig:assemblingI}(a) and \ref{fig:assemblingII}: there
  are generically twelve attachment points in a quintuple, and they
  are identified in pairs so lead to six vertices in the assembled
  map.  We still obtain a difference of six in the situations where
  some diangles or triangles are reduced to the vertex-map. Note that
  the gluing between blue and red edges does not require corrective
  factors, since by convention the vertices which are incident to red
  edges and which are not attachment points receive no weight. In
  particular, potential boundary-vertices are obtained from such
  unweighted vertices, as wanted.
\end{proof}

\subsection{Disassembling a triply pointed map}
\label{sec:gluinginvtp}

We now sketch the proof
of Theorem~\ref{thm:main} in the case of procedure I with
$e_1=e_2=e_3=0$, which generates a triply pointed map (three
boundary-vertices).  Boundary-vertices are always considered tight, so
we only have to exhibit the inverse bijection. The key point is to
identify the attachment points: once we know them, the decomposition
into diangles and triangles is done by cutting along leftmost
bigeodesics.

\begin{figure}
  \centering
  \fig{.4}{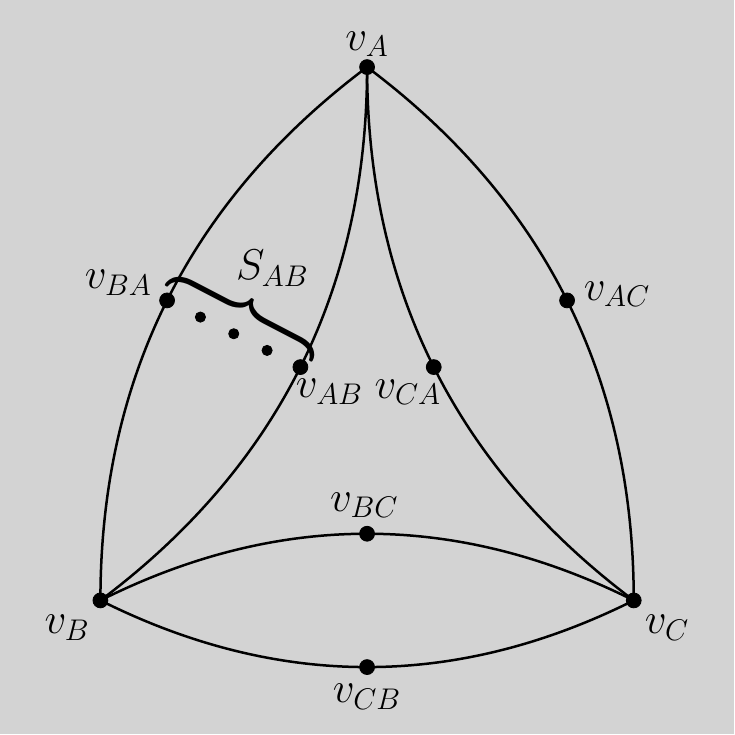}
  \caption{Sketch of the decomposition of a triply pointed map, with
    marked vertices $v_A$, $v_B$ and $v_C$. The set $S_{AB}$ consists
    of geodesic vertices between $v_A$ and $v_B$ at distance $r_A$
    from $v_A$. We pick its extremal elements $v_{AB}$ and $v_{BA}$,
    from which the two leftmost bigeodesics launched towards $v_A$ and
    $v_B$ delimit a maximal diangle containing all other elements of
    $S_{AB}$, but not $v_C$. Similarly, we construct two other
    disjoint diangles, formed by lefmost bigeodesics launched from
    vertices $v_{BC}$ and $v_{CB}$, $v_{CA}$ and $v_{AC}$
    respectively. The complementary region consists of two geodesics
    triangles (one inside and one outside).}
  \label{fig:threepointdecomp}
\end{figure}

Consider a planar bipartite map with three marked distinct vertices
$v_A$, $v_B$ and $v_C$ (as there are no boundary-faces, being
essentially bipartite is the same as being bipartite). Let us denote
by $d_{AB}$, $d_{BC}$ and $d_{CA}$ their mutual distances. By the
triangle inequalities and bipartiteness, there exists three
nonnegative integers $r_A$, $r_B$ and $r_C$, at most one of them
vanishing, such that
\begin{equation}
  \label{eq:rdtp}
  d_{AB} = r_A + r_B, \quad d_{BC}= r_B+r_C, \quad d_{CA}=r_C+r_A.
\end{equation}
We then consider the set $S_{AB}$ of geodesic vertices between $v_A$
and $v_B$ which are at distance $r_A$ from $v_A$ (hence distance $r_B$
from $v_B$). In the generic situation where $r_A,r_B,r_C$ are all
nonzero, we may single out canonically two ``extremal'' elements
$v_{AB}$ and $v_{BA}$ of $S_{AB}$ as follows. For $v$ and $v'$ in
$S_{AB}$, consider the two leftmost bigeodesics towards $v_A$ and
$v_B$ which are launched from $v$ and $v'$: these bigeodesics delimit
two regions, which turn out to be balanced bigeodesic diangles. Then,
$v_{AB}$ and $v_{BA}$ are chosen in such a way that the diangle not
containing $v_C$ is the largest possible (and contains in particular
all other elements of $S_{AB}$), and such that $v_A$, $v_{AB}$, $v_B$
and $v_{BA}$ appears in clockwise order around it. See
Figure~\ref{fig:threepointdecomp} for an illustration. We similarly
define the vertices $v_{BC}$, $v_{CB}$, $v_{CA}$ and $v_{AC}$. In this
way, we obtain three balanced bigeodesic diangles (the specific choice
of $r_A,r_B,r_C$ ensures that these are disjoint), and the
complementary region forms two bigeodesic triangles, thereby giving
the quintuple we are looking for. Note that some diangles or triangles
may be reduced to the vertex-map, for instance it may happen that
$v_{AB}=v_{BA}$ (resp.\ $v_{AB}=v_{BC}=v_{CA}$), in which case the
corresponding diangle (resp.\ triangle) is equal to the vertex-map.

In the situation where, say, $r_C$ vanishes, $v_C$ is actually an
element of $S_{AB}$: we simply cut the map along the leftmost
bigeodesic launched from $v_C$ towards $v_A$ and $v_B$, which
transforms the map into a single balanced bigeodesic diangle. This may
be seen as a degeneration of the generic situation, upon taking
$v_C=v_{AB}=v_{BA}=v_{BC}=v_{CB}=v_{CA}=v_{AC}$, all other elements of
the quintuple being equal to the vertex-map.

This informal discussion overlooks some important details, such as the
fact that, in general, the leftmost bigeodesics merge before reaching
$v_A$, $v_B$ or $v_C$. We shall be more precise below when discussing
the general inverse bijection, which applies to both types of
boundaries (vertices or faces). Still, the core idea will be the same:
upon defining in a suitable way the distances $d_{AB},d_{BC},d_{CA}$
between the three boundaries, we will define some parameters
$r_A,r_B,r_C$ via the \emph{equilibrium conditions}~\eqref{eq:rdtp},
and we will then construct some \emph{equilibrium vertices}, analogous
to the vertices $v_{AB},v_{BA},v_{BC},v_{CB},v_{CA},v_{AC}$ defined
above. Once these equilibrium vertices are constructed, the
decomposition is done by cutting along leftmost bigeodesics. The main
change induced by the presence of boundary-faces is that all this
construction must be performed on the universal cover of the map,
which we will define now.

\section{The universal cover of a map with three boundaries}\label{sec:universal-cover-maps}

In this section we introduce the universal cover of a map with three
boundaries, which will be one of our main topological tools. Its
construction and some of its properties are discussed in
Section~\ref{univcovsec}. We explain in
Section~\ref{sec:univcov_assembling} how to visualize the assembling
procedures directly on the universal cover, before proving in
Section~\ref{tightsec} that the maps resulting from the assembling
procedures have tight boundaries.

\subsection{Construction and properties of the universal cover}\label{univcovsec}

Let $M$ be a map with three boundaries
$\partial_A,\partial_B,\partial_C$, that are either boundary-faces or
boundary-vertices. We view $M$ as a graph embedded in a
topological sphere $S'$, and we let $x_A,x_B,x_C$ be three
distinguished points of $S'$, such that $x_i$ belongs to the $i$-th
boundary component, that is $x_i=\partial_i$ in the case of a
boundary-vertex, or $x_i\in \partial_i$ in the case of a
boundary-face. We let $S$ be the triply punctured sphere
$S'\setminus \{x_A,x_B,x_C\}$.

\paragraph{The universal cover of the triply punctured sphere.}

The universal cover of the surface $S$ is the data of a simply
connected topological space $\tilde{S}$ and of a mapping
$p:\tilde{S}\to S$ such that for every $x\in S$, there is a
neighborhood $U$ of $x$ in $S$ such that $p^{-1}(U)$ is a disjoint
union of open sets $U_i$, with $p|_{U_i}:U_i\to U$ a
homeomorphism. The pair $(\tilde{S},p)$ is unique up to the natural
notion of isomorphism.

\begin{figure}[t]
\center
\includegraphics[scale=.8]{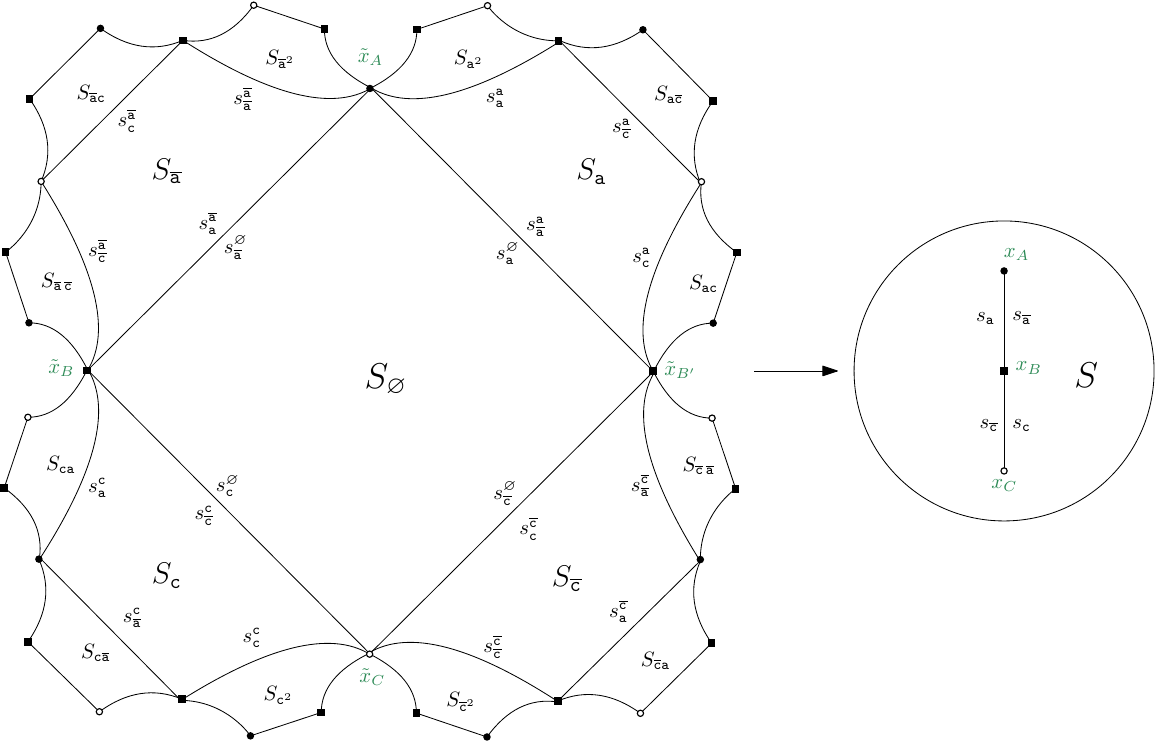}
\caption{A part of the universal cover of the triply punctured
  sphere. The ideal corners
  $\tilde{x}_A,\tilde{x}_B,\tilde{x}_C,\tilde{x}_{B'}$ correspond
  respectively to the punctures $x_A,x_B,x_C$ and $x_B$ again.}
\label{fig:Univcovcarre}
\end{figure}

It will be useful to consider the following concrete construction. Let
$S_0$ denote the unit square with its four corners removed, and with
its four sides denoted by
$s_{\al},s_{\overline{\al}}, s_\cl,s_{\overline{\cl}}$ in
counterclockwise direction. Note that the space obtained by gluing
$S_0$ along its sides by identifying $s_\al$ with $s_{\overline{\al}}$
and $s_\cl$ with $s_{\overline{\cl}}$ (with a head to tail matching of
their orientations so that the resulting surface is orientable) is
homeomorphic to the triply punctured sphere $S$, and we let
$p:S_0\to S$ be the resulting projection.

Let $F=\langle \al,\cl\rangle$ be the free group with two generators
$\al,\cl$, that is the set of reduced finite words $\wl$ made of the
four letters $\al,\overline{\al},\cl,\overline{\cl}$, where
$\overline{\al}=\al^{-1},\overline{\cl}=\cl^{-1}$ are the inverses of
$\al,\cl$ (we thus let
$\overline{\overline{\al}}=\al,\overline{\overline{\cl}}=\cl$). Here,
we say that a word is reduced if it does not contain an occurence of
any letter followed immediately by its inverse. The group operation is
defined by letting $\vl \wl$ be the reduced word obtained from the
concatenation of $\vl$ and $\wl$. The empty word $\varnothing$ is the
neutral element of $F$. We identify $F$ with its Cayley graph with
generators $\{\al,\overline{\al},\cl,\overline{\cl}\}$, which is the
infinite $4$-regular tree with root $\varnothing$, and we let $|\wl|$
be the length of the word $\wl$, which is also its distance in the
tree to the root.

For every $\wl\in F$, we let $S_\wl$ be a copy of $S_0$, that we can
view as $\{(x,\wl):x\in S_0\}$, with its four sides denoted by
$s^\wl_\al,s^\wl_{\overline{\al}},s^\wl_\cl,s^\wl_{\overline{\cl}}$,
again oriented counterclockwise around $S_\wl$. We consider the space
$\tilde{S}$ obtained by gluing the spaces $S_\wl$ along their sides,
in such a way that $s^\wl_{\mathtt{l}}$ is glued with
$s^{\wl\mathtt{l}}_{\ov{\mathtt{l}}}$ for every letter
$\mathtt{l}\in \{\al,\overline{\al},\cl,\overline{\cl}\}$ (with a head
to tail matching of their orientations). See Figure
\ref{fig:Univcovcarre}. We then extend the projection $p$ to a mapping
$p:\tilde{S}\to S$ by letting $p(x,\wl)=p(x)$ for every $x\in S_0$ and
$\wl\in F$. This is easily seen to be the universal cover of $S$.

The universal cover comes with its group of automorphisms
$\mathrm{Aut}(p)$, that are the homeomorphisms
$u:\tilde{S}\to \tilde{S}$ such that $p\circ u=p$. This group is
isomorphic to the free group $F$ via the natural action of $F$ on
$\tilde{S}$ defined by $\wl\cdot (x,\vl)=(x,\wl\vl)$.  Note that the
element of $\mathrm{Aut}(p)$ corresponding to $\wl$ sends $S_\vl$ to
$S_{\wl \vl}$, which may be arbitrarily far apart. We denote by $A$
and $C$ the elements of $\mathrm{Aut}(p)$ corresponding to the action
of $\al$ and $\cl$ respectively. It will be also convenient to
introduce the automorphisms $B:=A^{-1} C^{-1}$ and $B':=C^{-1}A^{-1}$
corresponding the respective actions of the elements
$\bl:=\overline{\al}\, \overline{\cl}$ and
$\bl':=\overline{\cl}\, \overline{\al}$ of $F$.

Classically, $\tilde{S}$ is a topological space homeomorphic to the
open unit disk. In fact, the universal cover of the triply punctured
sphere can also be constructed via hyperbolic geometry, see
e.g.~\cite[Section~5.3]{Stillwell2012}. Figure~\ref{fig:UnivcovPoincare}
(left) displays the connection with our present construction: the
gluing of the squares $(S_\wl)_{\wl \in F}$ can be realized as a
regular tiling in the hyperbolic plane made of ideal quadrangles.

For our purposes, it will be convenient to augment $\tilde{S}$ by
adding back the corners of the squares $S_\wl$, and we denote the
resulting space by $\tilde{S}'$.  Note that, after gluing, a corner is
common to infinitely many squares: for instance, the corner denoted
$\tilde{x}_A$ on Figure~\ref{fig:Univcovcarre} is common to all the
squares $S_{\al^n}$ for $n \in \Z$, similarly $\tilde{x}_C$ is common
to all the $S_{\cl^n}$, while $\tilde{x}_B$ (resp.\ $\tilde{x}_{B'}$)
is common to all the $S_{\bl^n}$ and $S_{\bl^n \cl}$ (resp.\
$S_{(\bl')^n}$ and $S_{(\bl')^n \al}$). The corners form a subset of
the `ideal boundary' of $\tilde{S}$ (which corresponds to the boundary
of the disk in Figure~\ref{fig:UnivcovPoincare}), and we therefore
call them \emph{ideal corners}. The projection $p: \tilde{S} \to S$
extends to a continuous mapping\footnote{Note however that, if we
  identify $\tilde{S}$ with the open unit disk as in
  Figure~\ref{fig:UnivcovPoincare}, then $p$ \emph{does not extend} to
  a continuous function on the closed unit disk. It only admits
  non-tangential limits at the ideal corners, which form a dense
  countable subset of the unit circle. This is consistent with the
  topology of $\tilde{S}'$ resulting from the gluing of squares: in
  the hyperbolic plane picture, given an ideal corner $\tilde{x}$, a
  neighborhood basis for $\tilde{x}$ consists of the interiors of
  horocycles of center $\tilde{x}$.}  from $\tilde{S}'$ to the
sphere $S'$, and the group of automorphisms $\mathrm{Aut}(p)$ acts
naturally on $\tilde{S}'$, each ideal corner being left invariant by
an infinite cyclic subgroup.

Finally, we recall that for any continuous path $\gamma:[0,1]\to S$
and any $\tilde{\gamma}(0)\in p^{-1}(\gamma(0))$, there is a unique
continuous path $\tilde{\gamma}:[0,1]\to \tilde{S}$ starting at
$\tilde{\gamma}(0)$ and such that $p\circ\tilde{\gamma}=\gamma$. This
path is called the \emph{lift} of $\gamma$ starting at
$\tilde{\gamma}(0)$. The lift remains well-defined if the endpoint
$\gamma(1)$ is one of the punctures $x_A,x_B,x_C$, in that case
$\tilde{\gamma}(1)$ is an ideal corner. Lifting a path joining two
punctures may be done by splitting the path at an intermediate point,
and choosing a preimage for that point.

\paragraph{Lifting the map $M$ on the universal cover.}

\begin{figure}[t]
\center
\includegraphics[scale=.8]{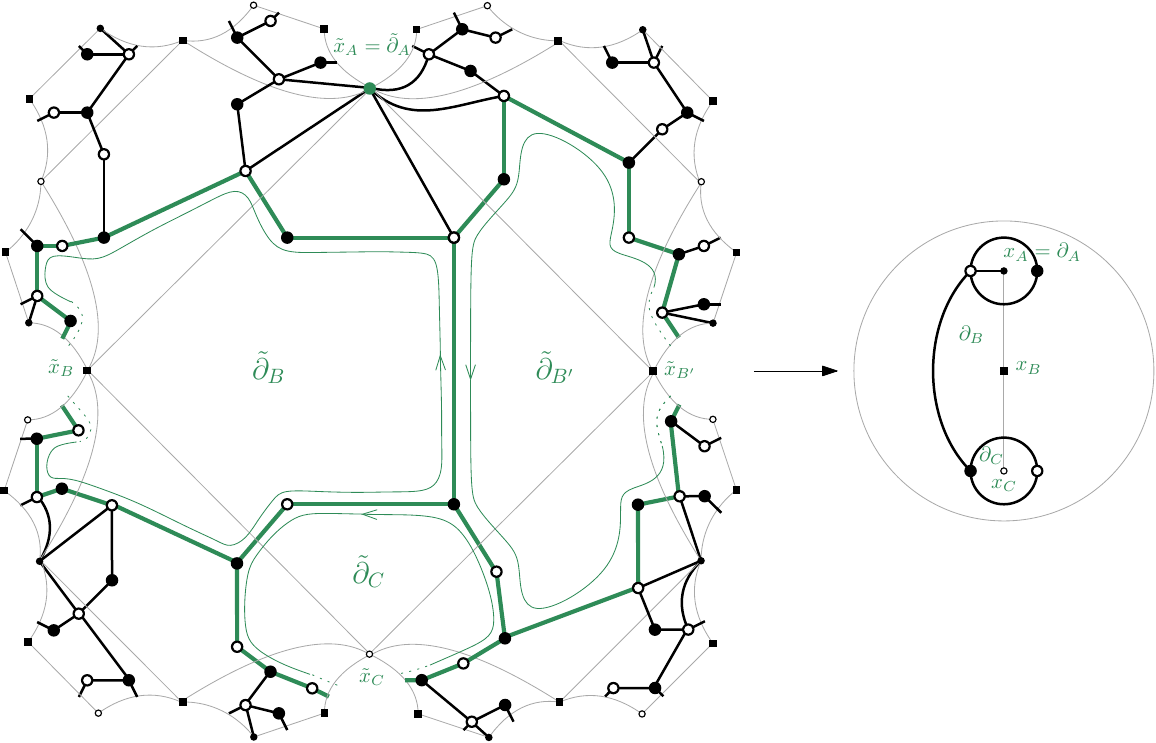}
\caption{A part of the universal cover of a dumbbell-shaped map,
  having two boundary-faces $\partial_B$ and $\partial_C$ and one
  boundary-vertex $\partial_A=x_A$. We display in green the
  distinguished ideal vertex $\tdel_A=\tilde{x}_A$, and the contours
  (oriented counterclockwise by convention) of the distinguished ideal
  faces $\tdel_B$, $\tdel_C$ and $\tdel_{B'}$.  The embedding is good,
  according to the definition given in Section~\ref{sec:goodfund}.}
\label{fig:Univcovcarremap}
\end{figure}

\begin{figure}
  \centering
  \includegraphics[scale=.5]{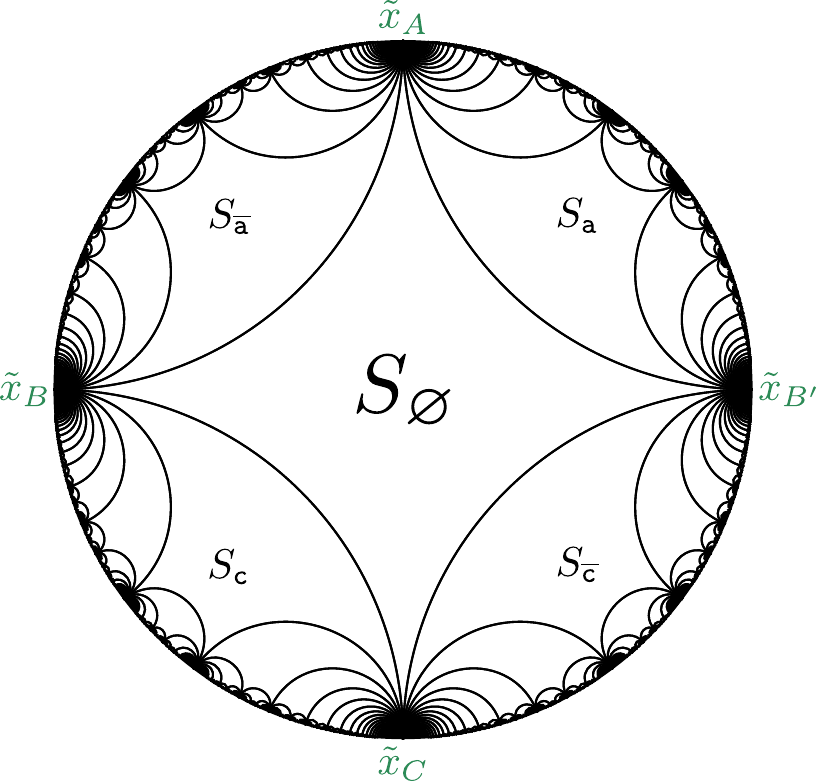}
  \hspace{1em}
  \includegraphics[scale=.5]{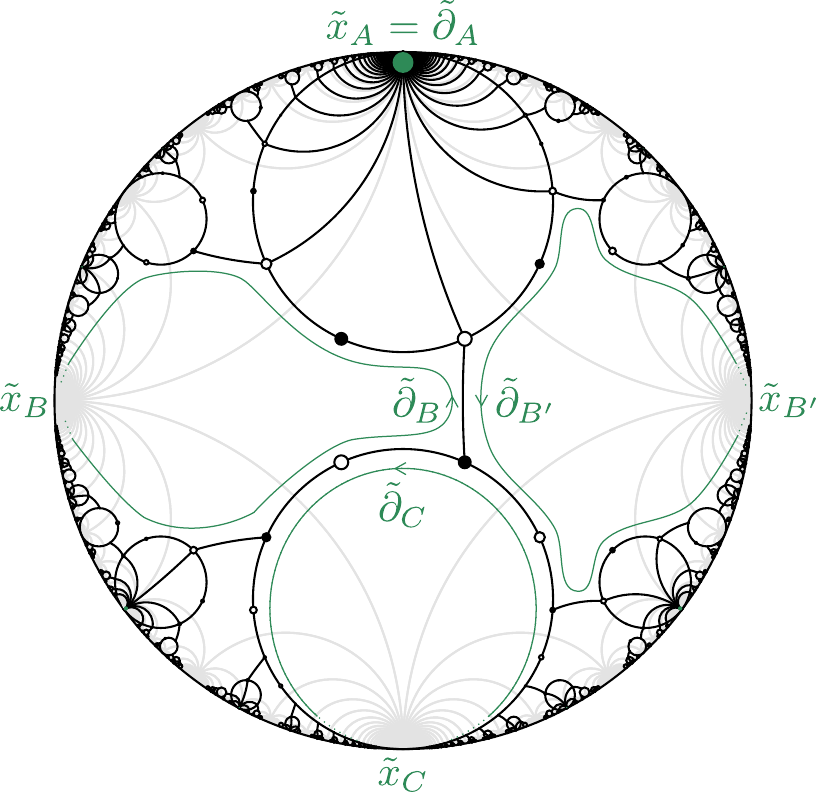}
  \caption{Left: the regular tiling of the hyperbolic plane (here
    represented in the Poincaré disk model) realizing the gluing of
    squares of Figure~\ref{fig:Univcovcarre}. Right: the corresponding
    representation of the universal cover of the dumbbell-shaped map
    of Figure~\ref{fig:Univcovcarremap}. The green vertex at the ideal
    corner $\tilde{x}_A$ has infinite degree. The contour of
    $\tilde{\partial}_C$ follows a horocycle with center the ideal
    corner $\tilde{x}_C$.}
  \label{fig:UnivcovPoincare}
\end{figure}

We now explain how the above considerations interact with the map $M$.
By viewing the (oriented) edges of $M$ as paths in $S'$ parametrized
by $[0,1]$, we can consider the lifts of these edges in $\tilde{S}'$,
which form an embedded graph $\tilde{M}$ in $\tilde{S}'$.

The resulting embedded graph $\tilde{M}$ is an infinite map in the
non-compact surface $\tilde{S}'$, with some faces or vertices of
infinite degrees (see Figure~\ref{fig:Univcovcarremap} and the right
of Figure~\ref{fig:UnivcovPoincare} for an example). More precisely,
$\tilde{M}$ has two possible types of vertices and faces:
\emph{regular} vertices and faces, which are the preimages of the
non-boundary vertices and faces of $M$ (keeping the same finite
degree), and \emph{ideal} vertices and faces, which have infinite
degree, project to the boundaries of $M$, and are in bijection with
the ideal corners of $\tilde{S}'$.

Let us provide some elements of justification to this
dichotomy. First, since $p:\tilde{S}\to S$ is a cover, it is immediate
that a non-boundary vertex of $M$, being placed at a point of $S$,
lifts to vertices of $\tilde{M}$ with the same finite degree.  Next,
if $f$ is a face of $M$ which is not a boundary-face, then (at least
when $f$ is not incident to a boundary-vertex, otherwise we have to
adapt slightly the argument) its contour $\partial f$
is homotopic in $S$ to a point, so that its lifts form closed paths in
$\tilde{S}$ bounding the preimages of $f$ by $p$ (which therefore keep
the same finite degree). Now, if $f$ is a boundary-face of $M$, then
its contour $\partial f$ is not homotopic to a point, so that its
preimages in $\tilde{S}$ are domains bounded by lifts of $\partial f$,
which form infinite paths of edges, resulting in ideal faces with infinite
degree.  Finally, a boundary-vertex $v$ of $M$, being placed at a
puncture, lifts to an ideal corner $\tilde{x}$ of $\tilde{S}'$, and
has infinite degree since $\tilde{x}$ is common to infinitely many
squares $S_\wl$ contributing at least one edge incident
to $\tilde{x}$.

We call the infinite map $\tilde{M}$ the \emph{universal cover} of the
map $M$.  We endow the set of its vertices $V(\tilde{M})$ with the
graph distance $\tilde{d}$, as for finite maps. The automorphism group
of $\tilde{M}$ is the same as that of $\tilde{S}$,
$\mathrm{Aut}(p)$. It acts freely on the regular vertices and faces,
but each ideal vertex or face is left invariant by an infinite cyclic
subgroup, precisely the same as that fixing the corresponding ideal
corner of $\tilde{S}'$. We point out that, in the concrete construction
of $\tilde{S}$ done above, it comes endowed with a distinguished
fundamental domain $S_\varnothing$, which in turn distinguishes four
ideal corners $\tilde{x}_A$, $\tilde{x}_B$, $\tilde{x}_C$ and
$\tilde{x}_{B'}$ (see again Figure~\ref{fig:Univcovcarre}), and
therefore four ideal vertices or faces $\tdel_A$, $\tdel_B$, $\tdel_C$
and $\tdel_{B'}$ of $\tilde{M}$ (see again
Figure~\ref{fig:Univcovcarremap}). We emphasize that this ``rooting''
results from our construction of $\tilde{S}$ and \emph{not} from any
extra data on $M$ other than its distinguished boundaries.

Note that there is an important flexibility in our
construction. Indeed, we can choose in an arbitrary way the embedding
of $M$ in $S$ or the projection $p:S_0\to S$: the resulting
$\tilde{S}$, $\tilde{M}$ and extension of $p$ to $\tilde{S}$ will
always be the same, up to isomorphisms. This observation will be
useful in Section \ref{sec:decface}.

Finally, we record the important observation that, if $M$ is
essentially bipartite, then its universal cover $\tilde{M}$ is
bipartite. Indeed, every simple cycle in $\tilde{M}$ is contractible
and encloses a finite number of faces of finite even degree, thus has
even length.

\subsection{Visualizing the assembling procedures on the universal cover}
\label{sec:univcov_assembling}
\begin{figure}
  \centering
  \fig{}{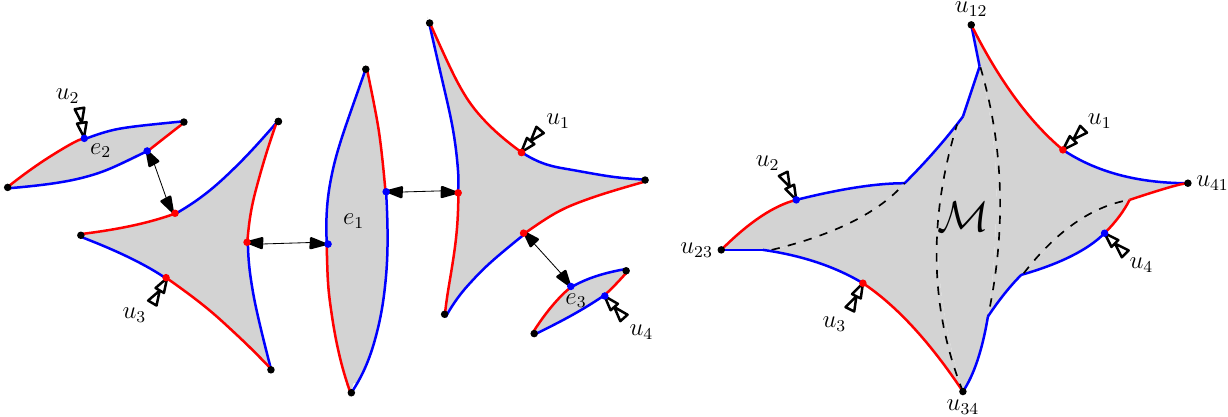}
  \caption{The partially assembled map $\MM$ obtained from three
    bigeodesic diangles and two bigeodesic triangles.  See again Figure
\ref{fig:assemblingsample} for a concrete example. }
  \label{fig:assembling_partial}
\end{figure}

We now explain how the assembling procedures I and II can be
visualized on the universal cover of the triply punctured sphere.  Let
us start again from three (bigeodesic) diangles and two triangles, as
at the beginning of Section~\ref{sec:gluing}. In the description given
in Section~\ref{sec:gluingdesc}, the assembling was done in two
successive operations, attachment then red-to-blue gluing. Here, it is
convenient to give an alternative (but equivalent) description in
which the assembling is done even more progressively, by doing first a
partial attachment and a partial red-to-blue gluing, thereby giving a
\emph{partially assembled map}, and then completing the assembling by
another round of attachment and red-to-blue gluing.

Precisely, the partially assembled map, which we denote by $\MM$, is
constructed as displayed on
Figure~\ref{fig:assembling_partial}. Namely, we identify some of the
attachment points of the diangles and triangles together, but some of
them, denoted $u_1$, $u_2$, $u_3$ and $u_4$, remain unattached for
now. Note that the attachments that we perform are exactly those which
are common to procedures I and II, see again
Figures~\ref{fig:assemblingI}(a) and~\ref{fig:assemblingII}. This
results in a map with one unique special face. We then perform
red-to-blue gluing counterclockwise around the special face (i.e.\
with this face on the left), with the important prescription that we
do not perform gluings which would require passing over the attachment
points $u_1,\ldots,u_4$. For instance, a red edge preceding the
attachment point $u_1$ on the right triangle remains unglued, since it
would be glued with a blue edge beyond $u_1$. The partially assembled
map $\MM$ is a planar bipartite map with one boundary-face and eight
distinguished incident corners: the attachment points
$u_1,u_2,u_3,u_4$, at which we switch from blue to red when turning
counterclockwise around $\MM$, and the corners
$u_{12},u_{23},u_{34},u_{41}$ at which we switch back from red to
blue. In the situation displayed on
Figure~\ref{fig:assembling_partial}, $u_{12}$ corresponds to a corner
of the right triangle, but it could also be, say, a corner of the
middle diangle, if the latter were longer.  It is straightforward to
check that $[u_1,u_{12}]$ (resp.\ $[u_{41},u_{12}]$) is a strictly
geodesic (resp.\ geodesic) boundary interval of $\MM$, as defined in
Section~\ref{sec:geoddefs}, and similarly for the other intervals
around $\MM$. Furthermore, by the definition of exceedances, we see
that
\begin{equation}
  \label{eq:exccount}
  \begin{split}
    d(u_1,u_{12})+e_1+e_2&=d(u_{12},u_2), \\
    d(u_2,u_{23})+e_2&=d(u_{23},u_3), \\
    d(u_3,u_{34})+e_1+e_3&=d(u_{34},u_4), \\
    d(u_4,u_{41})+e_3&=d(u_{41},u_1),
  \end{split}
\end{equation}
where $d$ denotes the graph distance in $\MM$ and where, by a slight abuse,
we identify corners with their incident vertices. Note that the
boundary-face of $\MM$ is not necessarily simple, as there may be
contacts between the different blue intervals (as said in the caption
of Figure~\ref{fig:diangle}, such contacts may exist in a diangle of
positive exceedance, and may subsist in the partial gluing, for
instance if all the remaining pieces are equal to the vertex-map).

In order to complete the assembling procedures, we have to identify
the remaining attachment points $u_1,\ldots,u_4$ together: let
$\MM_{\text{I}}$ (resp.\ $\MM_{\text{II}}$) be the map obtained from
$\MM$ by identifying $u_1$ with $u_2$ and $u_3$ with $u_4$ (resp.\
$u_1$ with $u_4$ and $u_2$ with $u_3$). Note that these
identifications are exactly those which are specific to
Figure~\ref{fig:assemblingI}(a) (resp.\
Figure~\ref{fig:assemblingII}). In the terminology of
Section~\ref{sec:gluingdesc}, the maps $\MM_{\text{I}}$ and
$\MM_{\text{II}}$ have three ``special'' faces, and we complete the
assembling procedures by performing red-to-blue gluing in each special
face. It is straightforward to check that, for both procedure I and
procedure II, we obtain the same final result as with the previous
construction.

\begin{figure}
  \centering
  \fig{}{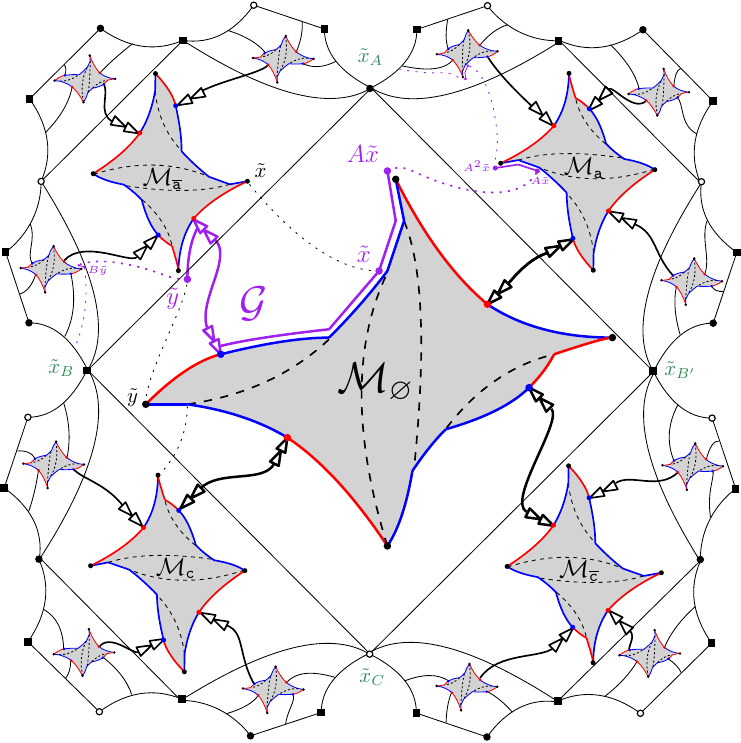}
  \caption{A visualization of the assembling procedure I in the universal
    cover of the triply punctured sphere. Inside each square $S_\wl$
    of the tiling of Figure~\ref{fig:Univcovcarre}, we place a copy
    $\MM_\wl$ of the partial gluing $\MM$ of
    Figure~\ref{fig:assembling_partial}. By attaching these maps
    together via their attachment points as displayed with the arrows,
    and deforming appropriately to perform the identification of
    vertices, we obtain the universal cover of the map
    $\MM_{\text{I}}$. To complete the assembling, we have to match and
    glue the red and blue edges of the ideal faces together (some
    identification between vertices are shown with dotted lines),
    thereby leaving only blue unmatched edges incident to ideal
    faces. Procedure II works just the same, except that we have to
    rotate each map $\MM_\wl$ by a quarter-turn clockwise. We display
    in purple the path $\GG$ and its vertices
    $\tilde{x},A\tilde{x},A^2\tilde{x},\tilde{y},B\tilde{y}$, as
    defined in Section~\ref{tightsec} (in the displayed situation we
    have $\Delta_\al,\Delta_{\overline{\cl}}>0$ but
    $\Delta_{\overline{\al}}=0$). By Lemma~\ref{lem:minpathnew}, $\GG$
    is a geodesic in $\tilde{M}$.}
  \label{fig:Univcovcarre_assembling}
\end{figure}

The interest of our alternative description is that it is now easier
to visualize the universal cover of the resulting assembled
maps. Precisely, we consider the square tiling $(S_\wl)_{\wl \in F}$
of Section~\ref{univcovsec}, and inside each square $S_\wl$ we place a
\emph{copy} $\MM_\wl$ of the partially assembled map $\MM$. We then
identify the attachment points of these copies together as follows. In
the case of procedure I displayed on
Figure~\ref{fig:Univcovcarre_assembling}, the attachment point
$u_1^\wl$ of $\MM_\wl$ is identified with the attachment point
$u_2^{\wl \al}$ of $\MM_{\wl \al}$, and $u_3^\wl$ is identified with
$u_4^{\wl \cl}$, for all $\wl \in F$. Similarly, in the case of
procedure II, looking again at Figure~\ref{fig:assemblingII}, we see
that the attachment point $u_2^\wl$ of $\MM_\wl$ is identified with
the attachment point $u_3^{\wl \al}$ of $\MM_{\wl \al}$, and $u_4^\wl$
is identified with $u_1^{\wl \cl}$. Of course, these identifications
require to deform the copies, and we can place the identified
attachment vertices on the sides of the squares if we want.  At this
stage, we obtain infinite maps denoted by $\tilde{\MM}_{\text{I}}$ and
$\tilde{\MM}_{\text{II}}$, which are the respective universal covers
of $\MM_{\text{I}}$ and $\MM_{\text{II}}$, upon seeing their three
special faces as the three boundaries.

To complete the assembling procedure in the universal cover, we have
to perform a final round of red-to-blue gluing in each special (ideal)
face. To better understand how this works, it is useful to set up some
notations. We first introduce the shorthand notations
\begin{equation}
  \label{eq:Idefs}
  I_1:=[u_{41},u_{12}], \quad I_2:=[u_{12},u_{23}], \quad I_3:=[u_{23},u_{34}], \qquad
  I_4:=[u_{34},u_{41}]
\end{equation}
for the sides of $\MM$ (so that $I_j$ is a bigeodesic launched from
the assembling point $u_j$). We also introduce the unified notations
\begin{equation}
  \label{eq:Idefsbis}
  \text{procedure I: }
  \begin{cases}
    I_\al := I_1 \\ I_{\overline{\al}} := I_2 \\ I_\cl := I_3 \\ I_{\overline{\cl}} := I_4
  \end{cases}
  \qquad
  \text{procedure II: }
  \begin{cases}
    I_\al := I_2 \\ I_{\overline{\al}} := I_3 \\ I_\cl := I_4 \\ I_{\overline{\cl}} := I_1
  \end{cases}
\end{equation}
and, for $\wl \in F$ and
$\mathtt{l} \in \{\al,\overline{\al},\cl,\overline{\cl}\}$, we let
$I^\wl_{\mathtt{l}}$ be the corresponding side of the copy
$\MM_\wl$. The interest of the unified notations is that the
attachment point of $I^\wl_{\mathtt{l}}$ is identified with that of
$I^{\wl {\mathtt{l}}}_{\overline{\mathtt{l}}}$, regardless of whether
we apply procedure I or II. As is apparent on~\eqref{eq:Idefsbis},
switching from procedure I to procedure II amounts on the universal
cover to rotating each $\MM_\wl$ by a quarter-turn
clockwise. Furthermore, the red part of $I^\wl_{\mathtt{l}}$ and the
blue part of $I^{\wl {\mathtt{l}}}_{\overline{\mathtt{l}}}$ are
incident to the same special face, and appear successively
counterclockwise around it. Thus, in the red-to-blue gluing operation,
the red edges in $I^\wl_{\mathtt{l}}$ will first look for blue matches
in $I^{\wl {\mathtt{l}}}_{\overline{\mathtt{l}}}$. Observe that, in
all this discussion, we can return from the universal cover to the
sphere by dropping the copy superscript.

Let us first consider the case $\mathtt{l}=\al$:
by~\eqref{eq:exccount}, we see that the red part of $I_\al$ has
$\Delta_\al$ fewer edges than the blue part of $I_{\overline{\al}}$,
with
\begin{equation}
  \label{eq:Deltaal}
  \text{procedure I: } \quad \Delta_\al := e_1 + e_2, \qquad
  \text{procedure II: } \quad \Delta_\al := e_2.
\end{equation}
Since $\Delta_\al$ is always nonnegative, when we complete the
assembling of $\MM$, the red part of $I_\al$ is thus completely glued
to the beginning of the blue part of $I_{\overline{\al}}$ (starting at
the attachment point). If $\Delta_\al>0$, there remains an unmatched
blue part of length $\Delta_\al$ which forms the contour of the
boundary face $\partial_A$; if $\Delta_\al=0$, $\partial_A$ is instead
a boundary-vertex as explained in
Section~\ref{sec:gluingdesc}. Translated in the universal cover, the
red part of $I_\al^\wl$ is completely glued to the beginning of the
blue part of $I_{\overline{\al}}^{\wl \al}$ (starting at the
attachment point).  If $\Delta_\al>0$, the remaining unmatched blue
part of $I_{\overline{\al}}^{\wl \al}$ forms a lift of the contour of
the boundary face $\partial_A$; if $\Delta_\al=0$, we only get an
``exposed'' blue vertex which we can place at an ideal corner of
$\tilde{S}'$ in order to form an ideal vertex of the map.

The case $\mathtt{l}=\cl$ is entirely similar, the red part of $I_\cl$
having $\Delta_\cl$ fewer edges than the blue part of
$I_{\overline{\cl}}$, with
\begin{equation}
  \text{procedure I: } \quad \Delta_\cl := e_1 + e_3, \qquad
  \text{procedure II: } \quad \Delta_\cl := e_3.
\end{equation}

The cases $\mathtt{l}=\overline{\al}$ and $\mathtt{l}=\overline{\cl}$
are slightly more involved since they must be considered altogether to
construct the boundary $\partial_B$. Let $\Delta_{\overline{\al}}$
(resp.\ $\Delta_{\overline{\cl}}$) be the difference between the
length of the blue part of $I_\al$ (resp.\ $I_\cl$) and that of the
red part of $I_{\overline{\al}}$ (resp.\ $I_{\overline{\cl}}$). The
relations~\eqref{eq:exccount} imply that
$\Delta_\bl:=\Delta_{\overline{\al}}+\Delta_{\overline{\cl}}$ is given
by
\begin{equation}
  \text{procedure I: } \quad \Delta_\bl := e_2 + e_3, \qquad
  \text{procedure II: } \quad \Delta_\bl := 2 e_1 + e_2 + e_3
\end{equation}
and is therefore always nonnegative, but the signs of
$\Delta_{\overline{\al}}$ and $\Delta_{\overline{\cl}}$ themselves are
not fixed since they depend on the sizes of different diangles and
triangles. We therefore have three generic situations:
\begin{itemize}
\item[(i)] $\Delta_{\overline{\al}}$ and $\Delta_{\overline{\cl}}$ are both nonnegative,
\item[(ii)] $\Delta_{\overline{\al}}$ is negative (hence
  $\Delta_{\overline{\cl}}$ is positive),
\item[(iii)] $\Delta_{\overline{\cl}}$ is negative (hence
  $\Delta_{\overline{\al}}$ is positive).
\end{itemize}

\begin{figure}[t!]
\centering
\fig{.85}{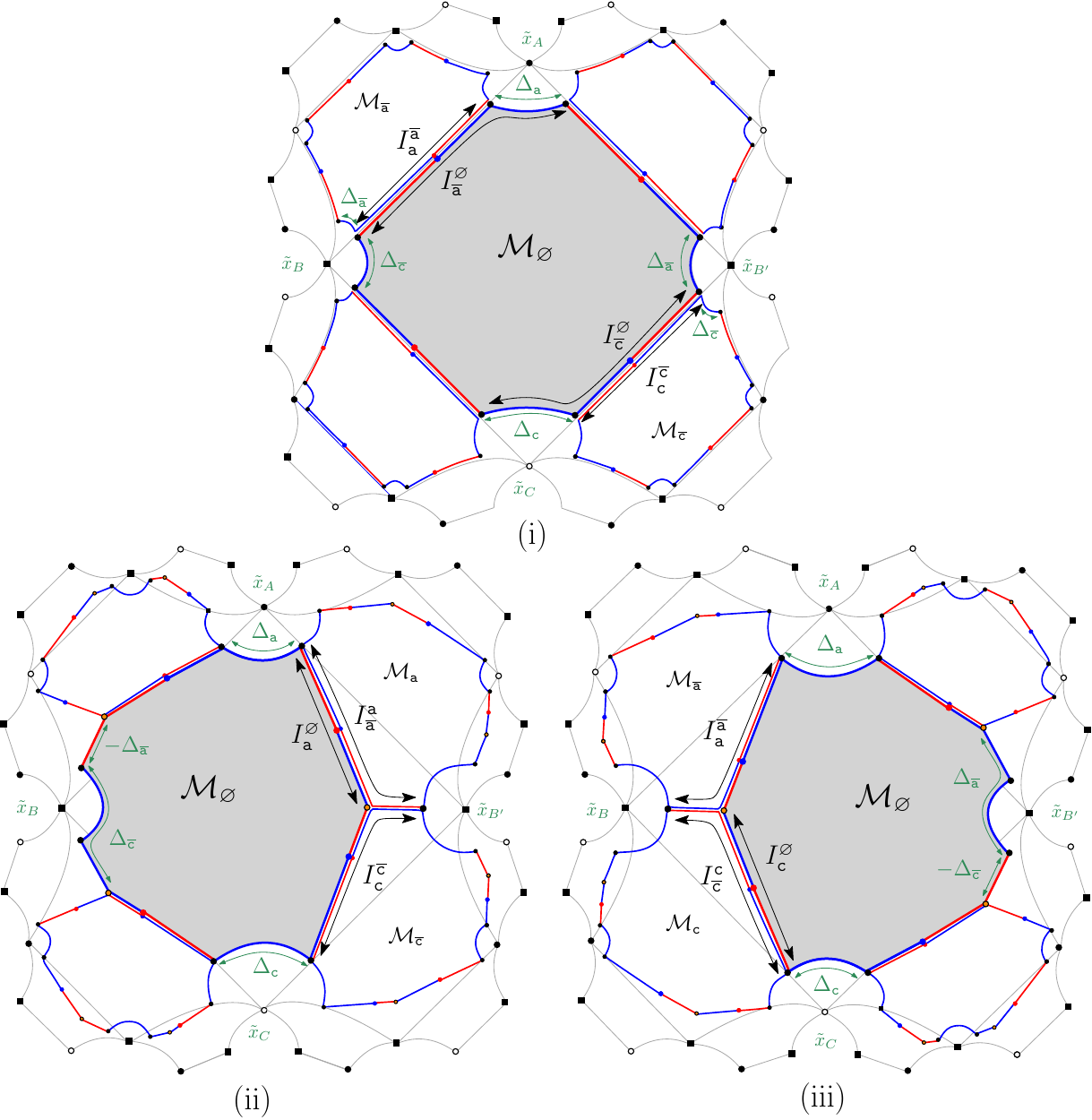}
\caption{The three generic situations corresponding to the possible
  signs of $\Delta_{\overline{\al}}$ and $\Delta_{\overline{\cl}}$:
  (i) $\Delta_{\overline{\al}},\Delta_{\overline{\cl}} \geq 0$, (ii)
  $\Delta_{\overline{\al}}<0$ and $\Delta_{\overline{\cl}}\geq 0$,
  (iii) $\Delta_{\overline{\al}}\geq 0$ and
  $\Delta_{\overline{\cl}}<0$. In case (ii), some of the red edges of
  $I_{\overline{\al}}^\al$ are matched to blue edges of
  $I_\cl^{\al\bl}=I_\cl^{\overline{\cl}}$, and similarly for case
  (iii) mutatis mutandis.}
  \label{fig:Univcovmap_cases}
\end{figure}

Let us discuss these three situations, illustrated on
Figure~\ref{fig:Univcovmap_cases}.
In the situation (i), all the red edges of $I_{\overline{\al}}$
(resp.\ $I_{\overline{\cl}}$) are matched to blue edges of $I_\al$
(resp.\ $I_\cl$). This implies two things about the universal cover,
one good and one bad. The good thing is that we only have ``nearest
neighbor'' gluings, as in the cases $\mathtt{l}=\al,\cl$ discussed
before: all the red edges of $I_{\overline{\al}}^\wl$ (resp.\
$I_{\overline{\cl}}^\wl$) are matched to blue edges of
$I_\al^{\wl \overline{\al}}$ (resp.\ $I_\cl^{\wl
  \overline{\cl}}$). The bad thing is that, since there are some
unmatched blue edges in both $I_\al$ and $I_\cl$ (unless we are in the
degenerate situation where $\Delta_{\overline{\al}}$ or
$\Delta_{\overline{\cl}}$ vanishes), the lifts of the contour of
$\partial_B$ do not remain in a single copy $\MM_\wl$, but we have to
consider two neighboring copies, say $\MM_\wl$ and
$\MM_{\wl \overline{\al}}$, to obtain such a lift as the concatenation
of the unmatched blue parts of $I_{\overline{\al}}^\wl$ and
$I_{\overline{\cl}}^{\wl \overline{\al}}$.

In the situation (ii), some red edges of $I_{\overline{\al}}$ do not
find blue matches in $I_\al$, and therefore find them in $I_\cl$. This
implies two things about the universal cover, one bad and one
good. The bad thing is that there are now ``next nearest neighbor''
gluings in the universal cover, namely some red edges of
$I_{\overline{\al}}^\wl$ are matched to blue edges of
$I_\cl^{\wl \bl}$, where we recall that
$\bl=\overline{\al}\, \overline{\cl}$ is an element of length $2$ in
$F$. The good thing is, since all unmatched blue edges are found on
$I_\cl$, the corresponding unmatched blue part of $I_\cl^\wl$ forms a
complete lift of the contour of $\partial_B$ that remains in the
single copy $\MM_\wl$.

The situation (iii) is entirely similar to the situation (ii), upon
interchanging the roles of $\al$ and $\cl$, and therefore changing
$\bl$ into $\bl'=\overline{\cl}\, \overline{\al}$.

\subsection{Tightness of the boundaries resulting from the assembling procedures}\label{tightsec}

We are now in position to prove that the maps resulting from the
assembling procedures have tight boundaries. Starting from three
diangles and two triangles, we denote by $\MM$ the partially assembled
map as defined in the previous subsection, and by $M$ the completely
assembled map (done according to procedure I or II). As we have seen,
the universal cover $\tilde{M}$ of $M$ can be constructed directly by
gluing infinitely many copies of $\MM$ along their boundaries.

A key ingredient in our proof is the path $\GG$ displayed on
Figure~\ref{fig:Univcovcarre_assembling} and defined as follows. It
consists of two parts which are both ``launched'' from the attachment
point between the copies $\MM_\varnothing$ and $\MM_{\overline{\al}}$,
and which go ``towards'' the ideal corners $\tilde{x}_A$ and
$\tilde{x}_B$, respectively, by following the blue boundaries. More
precisely, the part towards $\tilde{x}_A$ starts with the blue part of
$I_{\overline{\al}}^\varnothing$, then continues with the unmatched
blue part of $I_{\overline{\al}}^\al$, then that of
$I_{\overline{\al}}^{\al^2}$, $I_{\overline{\al}}^{\al^3}$, etc. Let
us denote by $\tilde{x}$ the last vertex on the blue part of
$I_{\overline{\al}}^\varnothing$ which is glued to the red part of
$I_\al^{\overline{\al}}$. Then, our path passes through the vertices
$A^n \tilde{x}$ for all $n \geq 0$ (the unmatched blue part of
$I_{\overline{\al}}^{\al^n}$ going from $A^n \tilde{x}$ to
$A^{n+1} \tilde{x}$). In the case $\Delta_\al=0$, all the vertices
$A^n \tilde{x}$ are identified and placed at the ideal corner
$\tilde{x}_A$. In this case, $\tilde{x}_A$ is reached in a finite
number of steps by our path. As soon as $\Delta_\al$ is positive, our
path is infinite but, as visible on
Figure~\ref{fig:Univcovcarre_assembling}, it still ``tends'' to
$\tilde{x}_A$ in a sense, as it meets at $\tilde{x}$ the contour of
the ideal face containing $\tilde{x}_A$, and follows it
counterclockwise onwards. The description of the part towards
$\tilde{x}_B$ is similar and just a bit more involved, for the reasons
discussed at the end of Section~\ref{sec:univcov_assembling}: assume
for instance that we are in case (i) with
$\Delta_{\overline{\al}},\Delta_{\overline{\cl}}\geq 0$, then the path
with the blue part of $I_\al^{\overline{\al}}$, continues with the
unmatched blue part of $I_\cl^\bl$, then that of
$I_{\al}^{\bl \overline{\al}}$, $I_{\cl}^{\bl^2}$,
$I_{\al}^{\bl^2 \overline{\al}}$, etc. If $\Delta_\bl>0$, the path
meets at a vertex denoted $\tilde{y}$ the contour of the ideal face
containing $\tilde{x}_B$, and follows it counterclockwise onwards,
passing in particular through $B^n \tilde{y}$ for all $n \geq 0$. If
$\Delta_\bl=0$, the path reaches $\tilde{x}_B$ in a finite number of
steps, and we set $\tilde{y}:=\tilde{x}_B$. Then, the key property of
$\GG$ is the following:

\begin{lem}
  \label{lem:minpathnew}
  The path $\GG$ is a geodesic in $\tilde{M}$.
\end{lem}

\begin{proof}
  Intuitively, the reason is that, in
  Figure~\ref{fig:Univcovcarre_assembling}, we glue the copies along
  (bi)geodesics.  More formally, consider a path $\gamma$ in
  $\tilde{M}$ whose two endpoints are on $\GG$: we need to show
  that the length of $\gamma$ is at least that of the portion of $\GG$
  having the same endpoints.  Our proof is by induction on the number
  of copies visited by $\gamma$.

  Precisely, we say that $\gamma$ \emph{enters into $k$ copies} if it
  can be written as a concatenation of paths
  $\gamma_0 \gamma_1 \ldots \gamma_k \gamma_{k+1}$, where $\gamma_0$
  and $\gamma_{k+1}$ are portions of $\GG$ (which we do not include
  into the copy count) and, for all $i=1,\ldots,k$, $\gamma_i$ is a
  path in the copy $\MM_{\wl_i}$. Without loss of generality, we may
  assume that the sequence $(\wl_i)_{i=1}^k$ is consistent with the
  structure of gluings discussed at the end of
  Section~\ref{sec:univcov_assembling}, that is to say in the
  situation (i) we have
  $\wl_{i+1} \wl_{i}^{-1} \in
  \{\al,\overline{\al},\cl,\overline{\cl}\}$ for all $i=1,\ldots,k-1$
  and, in the situations (ii) and (iii), we must also allow the values
  $\bl,\overline{\bl}$ and $\bl',\overline{\bl}'$, respectively. Note
  that ideal vertices lead to identifications between vertices
  belonging to distant copies, but we can always assume that the
  $\wl_i$ satisfy the above constraint by considering that a
  long-range identification corresponds to a concatenation of several
  paths $\gamma_i$ of length zero. We may also assume that $\wl_1$ and
  $\wl_k$ are both of the form $\al^n$, $\bl^n$ or
  $\bl^n \overline{\al}$ for some $n \geq 0$, as these correspond to
  the copies ``carrying'' $\GG$.

  If $\gamma$ enters into $0$ copies, it is contained in $\GG$ and we
  are done. If it enters into $k \geq 1$ copies, we will show that we
  can modify it, without increasing its length, so that it enters into
  at most $k-1$ copies. For this, we pick an $i$ such that $|\wl_i|$
  is maximal.  Let
  $\mathtt{l}\in \{\al,\overline{\al},\cl,\overline{\cl}\}$ be the
  last letter of $\wl_i$, with the convention that $\mathtt{l}=\al$ if
  $\wl_i=\varnothing$. Then, from the structure of gluings discussed
  above, from the definition of $\GG$, and from the maximality of
  $|\wl_i|$, we see that $\gamma_i$ has its two endpoints on the
  boundary interval $I_{\overline{\mathtt{l}}}^{\wl_i}$ which is
  geodesic in $\MM_{\wl_i}$. Let $\gamma_i'$ be the portion of
  $I_{\overline{\mathtt{l}}}^{\wl_i}$ with the same endpoints as
  $\gamma_i$: $\gamma_i$ cannot be shorter than $\gamma'_i$, so
  $\gamma$ is not shorter than the path
  $\gamma':=\gamma_0 \ldots \gamma_{i-1} \gamma'_i \gamma_{i+1} \ldots
  \gamma_{k+1}$. We claim that $\gamma'$ enters into (at most) $k-1$
  copies. Indeed, for $1<i<k$, $\gamma'_i$ can be viewed as a path on
  the boundaries of the copies $\MM_{\wl_{i-1}}$ and
  $\MM_{\wl_{i+1}}$---which are necessarily the same in the situation
  (i), but could differ in the situations (ii) and (iii)---so that we
  may rewrite $\gamma_{i-1}\gamma'_i\gamma_{i+1}$ as the concatenation
  of (at most) two paths, each of them staying in one of these
  copies. The argument is the same for $i=1$ or $i=k$, upon
  understanding that $\MM_{\wl_0}$ and $\MM_{\wl_{k+1}}$ refer to the
  path $\GG$.
\end{proof}

\begin{rem}
  \label{rem:GGnotgeod}
  The projection of $\GG$ on the finite map $M$ consists of a path
  connecting the projection $x$ of $\tilde{x}$ to the projection $y$
  of $\tilde{y}$---this path is \emph{not} a geodesic in general---
  prolonged with infinitely many turns around $\partial_A$ and
  $\partial_B$ when these are boundary-faces.
\end{rem}

We will need another technical lemma about $\GG$ and its two
distinguished vertices $\tilde{x}$ and $\tilde{y}$: 

\begin{lem}
  \label{lem:tightn-bound-result-new}
  Suppose that $\partial_A$ (resp.\ $\partial_B$) is a boundary-face
  of $M$, and let $c$ be a cycle of $M$ freely homotopic in $S$
  to, i.e.\ which can be continuously deformed into, the
  contour of $\partial_A$ (resp.\ $\partial_B$), oriented
  counterclockwise. Then, $c$ can be lifted to a path $\tilde{c}$ in
  $\tilde{M}$ going from a vertex $\tilde{z}$ to the vertex
  $A \tilde{z}$ (resp.\ $B \tilde{z}$), with $\tilde{z}$ belonging to
  $\GG$ and situated between $\tilde{x}$ and $\tilde{y}$.
\end{lem}

\begin{proof}
  If $\partial_A$ is a boundary-face, we denote its contour by
  $\hat{\partial}_A$, which we orient counterclockwise and view as a
  loop rooted at the projection $x$ of $\tilde{x}$ on $M$. If $c$ is
  freely homotopic to $\hat{\partial}_A$, there exists a path $\gamma$
  going from $x$ to a point $z_0$ on $c$ such that
  $\gamma c \gamma^{-1}$ and $\hat{\partial}_A$ are homotopic as loops
  rooted at $x$.  Lifting these loops to $\tilde{M}$, we get paths
  with the same endpoints. Choosing the starting point to be
  $\tilde{x}$, the final point must be $A\tilde{x}$ (since it is the
  final point for the lift of $\hat{\partial}_A$), and we deduce
  that $\gamma c \gamma^{-1}$ lifts to
  $\tilde{\gamma} \tilde{c}_0 (A\tilde{\gamma})^{-1}$, where
  $\tilde{\gamma}$ is the lift of $\gamma$ going from $\tilde{x}$ to a
  vertex $\tilde{z}_0$, and $\tilde{c}_0$ is a lift of $c$ going from
  $\tilde{z}_0$ to $A\tilde{z}_0$.

  If $\tilde{z}_0$ lies on $\GG$ we are done, otherwise we can tweak
  it as follows. Note that the paths $A^n\tilde{c}_0,n\in \Z$ are all
  lifts of $c$, with the property that the endpoint of
  $A^n\tilde{c}_0$ is the starting point of $A^{n+1}\tilde{c}_0$. The
  concatenation of these paths is a doubly infinite path
  $\tilde{c}_\infty$ in $\tilde{M}$ (whose projection to $M$ circles
  around $c$ indefinitely) which, by planarity, necessarily crosses
  $\GG$ at a vertex $\tilde{z}$ situated between $\tilde{x}$ and
  $\tilde{y}$. Indeed, as visible on
  Figure~\ref{fig:Univcovcarre_assembling}, the portion of $\GG$
  between $\tilde{x}$ and $\tilde{y}$ separates the copies $\MM_\wl$
  with a reduced word $\wl$ starting with the letter $\overline{\al}$,
  from the others. And, for $n \geq 0$ large enough,
  $A^{-n} \tilde{z}_0$ will be in the former set of copies, and
  $A^n \tilde{z}_0$ in the latter.  The portion $\tilde{c}$ of
  $\tilde{c}_\infty$ between $\tilde{z}$ and $A\tilde{z}$ is the
  wanted lift of $c$.

  The reasoning in the case where $\partial_B$ is a boundary-face and
  $c$ is a cycle freely homotopic to its contour is entirely similar.
\end{proof}

\begin{prop}
  \label{prop:tightness}
  The map $M$, assembled according to procedure I or II, has tight
  boundaries.
\end{prop}

\begin{proof}
  Let us first prove that the boundary $\partial_A$ is tight. If it is
  a boundary-vertex this is a tautology, otherwise it is a
  boundary-face of length $\Delta_\al>0$, and we need to show that any
  cycle $c$ freely homotopic to its contour has length at least
  $\Delta_\al$. Let us lift $c$ to a path $\tilde{c}$ as in
  Lemma~\ref{lem:tightn-bound-result-new}. Then, the length of $c$ is
  at least $\td(\tilde{z},A\tilde{z})$ which, by the triangle
  inequality, satisfies
  \begin{equation}
    \td(\tilde{z},A\tilde{z}) \geq \td(\tilde{z},A\tilde{x}) - \td(A\tilde{x},A\tilde{z}).
  \end{equation}
  As $\GG$ is a geodesic by Lemma~\ref{lem:minpathnew}, and as
  $\tilde{x}$ lies between $\tilde{z}$ and $A\tilde{x}$ on $\GG$, we
  have
  \begin{equation}
    \td(\tilde{z},A\tilde{x}) = \td(\tilde{z},\tilde{x}) +\td(\tilde{x},A\tilde{x}) = \td(\tilde{z},\tilde{x}) + \Delta_\al.
  \end{equation}
  But we have $\td(\tilde{z},\tilde{x})=\td(A\tilde{x},A\tilde{z})$
  since $A$ is an automorphism of $\tilde{M}$, hence
  \begin{equation}
    \td(\tilde{z},A\tilde{z}) \geq \Delta_\al
  \end{equation}
  as wanted.  The proof for $\partial_B$ is entirely similar, using
  again Lemma~\ref{lem:tightn-bound-result-new}, and replacing in the
  above argument $\Delta_\al$ by $\Delta_\bl$, $A$ by $B$, and
  $\tilde{x}$ by $\tilde{y}$.  Finally, for $\partial_C$, we notice
  that throughout this section it plays a symmetric role to
  $\partial_A$, viewing all the figures upside down (i.e.\ rotated by
  $180$ degrees).
\end{proof}

To conclude this section, let us note that, symmetrically to the
definition of $\GG$, we may define another geodesic $\GG'$ that
consists of two parts launched from the attachment point between the
copies $\MM_\varnothing$ and $\MM_{\overline{\cl}}$ and going towards
the ideal corners $\tilde{x}_C$ and $\tilde{x}_{B'}$,
respectively. Then, the copy $\MM_\varnothing$ is delimited by the
four geodesics $\GG$, $A\GG$, $\GG'$ and $C\GG'$, upon removing their
pairwise common parts. Furthermore, $A\GG$ (resp.\ $C\GG'$) may be
viewed as a geodesic launched from the attachment point between
$\MM_\varnothing$ and $\MM_\al$ (resp.\ $\MM_\cl$).

\section{Decomposing a map with three tight boundaries}
\label{sec:decface}

In this section we complete the proof of Theorem~\ref{thm:main}:
starting from an essentially bipartite planar map $M$ with three tight
boundaries $\partial_A,\partial_B,\partial_C$ of respective lengths
$2a,2b,2c\geq 0$ on the triply punctured sphere $S$, we want to
disassemble $M$ into two bigeodesic triangles and three bigeodesic
diangles with nonnegative exceedances, in a way that inverts the
assembling procedure.

We already sketched in Section~\ref{sec:gluinginvtp} the decomposition
in the case of triply pointed maps ($a=b=c=0$). In order to generalize
it to the case $a,b,c\geq 0$, we have to find an analog of the
equilibrium conditions~\eqref{eq:rdtp}, which involve the distances
$d_{AB},d_{BC},d_{CA}$ between the boundary-vertices. When some
boundaries are faces, we need appropriate analogs of these
distances: it turns out that such analogs may be constructed using
so-called Busemann functions defined on the universal cover
$\tilde{M}$ of the map $M$. From this, we will obtain the equilibrium
vertices, from which we will launch leftmost bigeodesics giving the
decomposition we are looking for.

Our presentation is done in several steps. First, we discuss in
Section~\ref{sec:goodfund} some graph-theoretical properties of the
infinite map $\tilde{M}$, via the notion of good embedding. Busemann
functions, associated with infinite geodesics, are then introduced in
Section~\ref{sec:hyperbolicbuse}. We then explain in
Section~\ref{sec:hyperbolicgeod} how tight boundaries give rise to
specific Busemann functions on $\tilde{M}$. In
Section~\ref{sec:univcovbigeod}, we adapt the notion of leftmost
bigeodesic to $\tilde{M}$, and use it to state the crucial diangle
lemma. We construct equilibrium vertices in
Section~\ref{sec:trianglelemma}, and use them to state the no less
crucial triangle lemma. We apply all these tools in
Sections~\ref{sec:decI} and~\ref{sec:decII} to exhibit the inverses of
the assembling procedures I and II, respectively.

\subsection{Graph-theoretical properties of the universal cover}
\label{sec:goodfund}

Let us start by discussing some properties of the underlying graph of
$\tilde{M}$, which is infinite. For this, it is useful to introduce
the notion of ``good embedding''.

Recall from Section~\ref{univcovsec} that the map $\tilde{M}$ is
obtained by lifting the map $M$ in the universal cover $\tilde{S}$ of
the triply punctured sphere $S$, and that we constructed $\tilde{S}$
as a square tiling $(S_\wl)_{\wl \in F}$ dual to the infinite
$4$-regular tree $F$. Let us denote by $\Sigma$ the projection of the
boundaries of the square $S_\varnothing$ (hence of any square $S_\wl$)
on the sphere $S'$: it is a path connecting $x_A$ to $x_C$ via $x_B$,
see again Figure~\ref{fig:Univcovcarre}.

In general, the edges of $\tilde{M}$ may connect vertices belonging to
squares arbitrarily far from each other in $F$, since the edges of $M$
may cross $\Sigma$ in an arbitrarily complicated manner. We can
simplify the situation by an appropriate deformation of the embedding
of $M$ in $S$ (hence, of the embedding of $\tilde{M}$ in
$\tilde{S}$). More precisely, we say that we have a \emph{good
  embedding} if every edge of $M$, with its endpoints excluded, is
either entirely contained in $\Sigma$, or intersects it in at most one
point. Note that, in the latter case, the endpoints may also belong to
$\Sigma$.  See again Figure~\ref{fig:Univcovcarremap} for an example.

\begin{lem}
  Every map $M$ with three boundaries admits a good embedding.
\end{lem}

\begin{proof}
  Consider the first derived map $M'$ of $M$, which is
  defined~\cite{Tutte63} as the triangulation obtained by
  superimposing $M$ with its dual---which creates a quadrangulation
  called the derived map~\cite{Schaeffer15} of $M$---and splitting
  each quadrangle into two triangles by connecting each vertex of $M$
  to all neighboring dual vertices. The vertex set of $M'$ can be
  partitioned into $\{W,W^*,W^\dagger\}$, where $W$, $W^*$ and
  $W^\dagger$ correspond respectively to the vertices, faces and edges
  of $M$, and the edges of $M'$ correspond to the incidence relations
  in $M$. Note that the vertices of $M'$ have even degree, and those
  in $W^\dagger$ have degree four.

  The boundaries $\partial_A,\partial_B,\partial_C$ of $M$ become
  vertices in $M'$, which we can place at the punctures $x_A,x_B,x_C$
  of $S$. Let $P$ be a simple path on $M'$ connecting $x_A$ to $x_C$
  via $x_B$, and going ``straight'' at every vertex in
  $W^\dagger$. Such a path always exists and, by deforming $M'$ in
  such a way that $P$ coincides with $\Sigma$, we get a good
  embedding of $M$.
\end{proof}

When the embedding of $M$ is good, which we will assume from now on,
the lifted edges of $\tilde{M}$ remains ``local'' with respect to the
tiling $(S_\wl)_{\wl \in F}$, i.e.\ they may only connect vertices
lying in the same square or in neighboring squares (i.e.\ squares
$S_\wl,S_{\wl'}$ with
$\wl^{-1}\wl' \in \{\al,\overline{\al},\cl,\overline{\cl}\}$).
In particular, when we remove the finitely many vertices belonging to
the square $S_\varnothing$ (including those possibly placed on its
boundaries and at its ideal corners) and their incident edges, then
$\tilde{M}$ is disconnected in four infinite pieces.  It follows that
$\tilde{M}$ has infinitely many ends in the graph-theoretical
sense. These ends are in natural bijection with those of the infinite
$4$-regular tree $F$.

Recall that $\tilde{M}$ has two possible types of vertices, namely
regular vertices with finite degree, and ideal vertices (placed at
ideal corners of $\tilde{S}'$) with infinite degree. In the absence of
ideal vertices (i.e.\ when $M$ has only boundary-faces), the
underlying graph of $\tilde{M}$ is locally finite.  In the presence of
ideal vertices, we have the following weaker property:

\begin{lem}
  \label{lem:finvert}
  Let $v$ and $v'$ be two vertices of $\tilde{M}$, and $r$ a nonnegative
  integer. Then, the number of simple paths from $v$ to $v'$ having
  length at most $r$ is finite.
\end{lem}

\begin{proof}
  Since we have a good embedding, we may keep track of the squares
  visited by a path on $\tilde{M}$:
  \begin{itemize}
  \item when following an edge between regular vertices, we may
    either remain in the same square $S_\wl$, or \emph{move} to a
    neighboring square $S_{\wl\mathtt{l}}$,
    $\mathtt{l} \in \{\al,\overline{\al},\cl,\overline{\cl}\}$,
  \item when passing through an ideal vertex projecting to $x_A$ (if
    there are any), we may \emph{jump} from the square $S_\wl$ to any
    square of the form $S_{\wl \al^n}$, with $n \in \Z$ arbitrary,
  \item similarly, when passing through an ideal vertex projecting to
    $x_C$ (if there are any), we may jump from $S_\wl$ to any square
    $S_{\wl \cl^n}$, $n \in \Z$,
  \item finally, when passing through an ideal vertex projecting to
    $x_B$ (if there are any), we may jump from $S_\wl$ to squares
    of the form $S_{\wl \bl^n}$, $S_{\wl \al \bl^n}$,
    $S_{\wl \bl^n \cl}$ or $S_{\wl \al \bl^n \cl}$, with $n \in \Z$
    and $\bl=\overline{\al}\ \overline{\cl}$ (see again
    Figure~\ref{fig:Univcovcarre}).
  \end{itemize}

  Without loss of generality, we may assume that the initial vertex
  $v$ belongs to the square $S_\varnothing$. Then, we may reach after
  $r$ steps only squares of the form
  $ S_{\wl_1 \mathtt{l}_1^{n_1} \cdots \wl_s \mathtt{l}_s^{n_s}}$,
  where $\wl_1,\ldots,\wl_s$ are elements of $F$ such that
  $|\wl_1|+\cdots+|\wl_s|+s \leq 2r$,
  $\mathtt{l}_1,\ldots,\mathtt{l}_s$ are equal to $\al$, $\bl$ or
  $\cl$, and $n_1,\ldots,n_s$ are arbitrary integers. Hence, generally
  speaking, infinitely many squares may be reached.

  However, here we fix the endpoint $v'$, hence the final square 
  $ S_{\wl_1 \mathtt{l}_1^{n_1} \cdots \wl_s \mathtt{l}_s^{n_s}}$. We
  claim that, on any simple path from $v$ to $v'$ with length at most
  $r$, we may only visit squares of the form
  $ S_{\vl_1 \mathtt{l}_1^{m_1} \cdots \vl_s \mathtt{l}_s^{m_s}}$,
  where $\vl_1,\ldots,\vl_s$ are elements of $F$ at bounded distance
  from the neutral element $\varnothing$, and $m_1,\ldots,m_s$ are
  integers such that either $\abs{m_i}\leq r$, or
  $\abs{m_i-n_i} \leq r$. The reason is that, if we perform a ``big
  jump'' (larger than $r$) at an ideal vertex, then we cannot ``undo''
  the jump in less than $r$ steps since we cannot revisit the same
  ideal vertex again---here we use the fact that $F$ is a free group,
  whose Cayley graph is a tree. Therefore, the number of squares that
  may be visited is finite, and since each square contains finitely
  many vertices, the claim follows.
\end{proof}

Note that Lemma~\ref{lem:finvert} becomes false if, instead of simple
paths, we consider general paths, or even nonbacktracking
paths. Indeed, without the simplicity assumption, it is possible to do
arbitrarily big jumps, then undo them.

Note also that we have not used the fact that the boundaries of $M$
are tight, in fact all the discussion in this subsection remains valid
without this assumption.

\subsection{Infinite geodesics and Busemann functions}
\label{sec:hyperbolicbuse}

Next, we introduce the main tool that will allow us to decompose maps with tight boundary-faces. This tool is the notion of Busemann function, a classical 
object of metric geometry, see for instance Chapter 5 in \cite{BuBuIv01}. 
In the infinite map $\tilde{M}$, we define an \emph{infinite
  geodesic} (also called a geodesic ray in metric geometry) 
as an infinite sequence $(\gamma(t))_{t \in \N}$ of vertices such that
\begin{equation}
  \tilde{d}(\gamma(t),\gamma(t'))=\abs{t-t'}
\end{equation}
for all $t,t'$ (recall that $\tilde{d}$ denotes the graph distance in
$\tilde{M}$). A \emph{biinfinite geodesic} is defined in exactly the
same way, except that the sequence is indexed by $\Z$ instead of
$\N$. To simplify our discussion, we overlook the fact that the map
may be non simple, which would in all rigor require to specify which
edges are visited by the geodesic.

With an infinite geodesic $\gamma$, we may associate its \emph{Busemann
  function} $B_\gamma$ which assigns to a vertex $v$ the quantity
\begin{equation}
  \label{eq:Busedef}
  B_\gamma(v) = \lim_{t \to +\infty} \tilde{d}(v,\gamma(t))-t.
\end{equation}
This quantity is well-defined since the function
$t \mapsto \tilde{d}(v,\gamma(t))-t$ is nonincreasing, by virtue of
the triangle inequality, and bounded from below by
$-\tilde{d}(v,\gamma(0))$.  In fact, since we are in a discrete metric
space, we have $B_\gamma(v) = \tilde{d}(v,\gamma(t))-t$ for $t$
large enough. We also have $B_\gamma(\gamma(t))=-t$ for all $t$.

It is not difficult to check that $B_\gamma$ is a $1$-Lipschitz
function, changes parity along each edge since the map $\tilde{M}$ is bipartite, and
admits no local minimum: every vertex $v$ has an adjacent vertex $v'$
such that $B_\gamma(v')=B_\gamma(v)-1$.

For a biinfinite geodesic $\gamma$, we define its Busemann function
$B_{\gamma}$ in the same way. Note that a change of parametrization
$t \mapsto t+t_0$ of $\gamma$ changes $B_\gamma$ by a constant, so a
Busemann function should really be viewed as ``defined modulo a
constant''. Note however that the change of parametrization
$t \mapsto -t$ gives rise to a different Busemann function.

\begin{lem}
  \label{lem:gdamdef}
  Let $\gamma$, $\gamma'$ be two infinite geodesics, and suppose that
  there exists a finite set $V_0$ of vertices whose removal splits
  $\tilde{M}$ in several connected components, such that $\gamma$ and
  $\gamma'$ eventually belong to different connected components
  (i.e. $\gamma(t)$ remains in one connected component and
  $\gamma'(t)$ in another for large enough $t$). Then, the function
  $B_\gamma+B_{\gamma'}$ admits a global minimum, which is reached at some
  $v_0 \in V_0$.
  Furthermore, there exists at least one biinfinite geodesic along
  which $B_\gamma$ is strictly increasing and $B_{\gamma'}$ is
  strictly decreasing, and a vertex $v$, not necessarily in $V_0$,
  belongs to such a geodesic if and only if it is a global minimum of
  $B_\gamma+B_{\gamma'}$.
\end{lem}

\begin{proof}
  Let $v_0$ be a vertex at which $B_\gamma+B_{\gamma'}$ attains its
  minimum in the finite set $V_0$.  For a given $v$, take $t$ large
  enough so that $B_\gamma(v)=\tilde{d}(v,\gamma(t))-t$,
  $B_{\gamma'}(v)=\tilde{d}(v,\gamma'(t))-t$, and so that $\gamma(t)$
  and $\gamma'(t)$ are in different connected components after
  removing $V_0$. By the triangle inequality, we have
  $B_\gamma(v)+B_{\gamma'}(v)\geq \tilde{d}(\gamma(t),\gamma'(t))-2t$,
  and we have
  $\tilde{d}(\gamma(t),\gamma'(t))=\tilde{d}(\gamma(t),v_1)+\tilde{d}(v_1,\gamma'(t))$
  for some $v_1 \in V_0$ since a geodesic path from $\gamma(t)$ to
  $\gamma'(t)$ necessarily meets $V_0$ at some vertex $v_1$. We get
  \begin{equation}
    B_\gamma(v)+B_{\gamma'}(v)\geq (\tilde{d}(v_1,\gamma(t))-t)
    +(\tilde{d}(v_1,\gamma'(t))-t)\geq
    B_\gamma(v_1)+B_{\gamma'}(v_1),
  \end{equation}
  which is at least $B_\gamma(v_0)+B_{\gamma'}(v_0)$.  This proves the
  first claim.

  Consider now an arbitrary minimizer $v$ of
  $B_\gamma+B_{\gamma'}$. Since $B_\gamma$ has no local minimum, we
  can construct an infinite path starting at $v$ along which
  $B_\gamma$ is strictly decreasing, and similarly we can construct
  another infinite path starting at $v$ along which $B_{\gamma'}$ is
  strictly decreasing. It is straightforward to check that, reversing
  the direction of the second path and concatenating it with the first
  one, we get a biinfinite geodesic with the wanted property, and that
  conversely any such geodesic can only pass through minimizers of
  $B_\gamma+B_{\gamma'}$.
\end{proof}

Note that, in this subsection, we have not used the fact that
$\tilde{M}$ is planar. In fact, our discussion (including
Lemma~\ref{lem:gdamdef}) holds for an arbitrary infinite
graph.

\subsection{Busemann functions associated with tight
  boundaries}
\label{sec:hyperbolicgeod}

We will now exploit the assumption that $M$ has tight
boundaries. Consider a boundary of $M$, say $\partial_A$ of length
$2a$. As discussed in Section~\ref{univcovsec}, $\partial_A$ has a
distinguished lift $\tdel_A$ which is an ideal vertex if $a=0$, and an
ideal face if $a>0$. In this latter case, we let $\gamma_A$ be the
biinfinite path on $\tilde{M}$ obtained by following the contour of
$\tdel_A$ in the counterclockwise direction. Choosing a reference
point $\gamma_A(0)$ arbitrarily, $\gamma_A$ is parametrized by $\Z$.

\begin{lem}\label{sec:infin-geod-busem}
  If $\partial_A$ is a tight boundary-face in $M$, then $\gamma_A$ is
  a biinfinite geodesic in $\tilde{M}$.
 \end{lem}

\begin{proof}
  If we view the contour of $\partial_A$ as a biinfinite sequence of
  vertices $(\lambda_A(t),t\in \Z)$ obtained by cycling infinitely
  many times around $\partial_A$, then~\cite[Proposition
  2.5]{CdVEr10}---which is closely related to the wrapping lemma of
  \cite{irredmaps}---implies that any path of the form
  $(\lambda_A(t+l),0\leq t\leq m)$ for $l\in \Z$ and $m\in \N$ is
  shortest in its homotopy class with fixed endpoints (such a path is
  called `tight' in \cite{CdVEr10}, so that our terminology
  agrees). Then, as noted in \cite[Section 2.1]{CdVEr10}, this
  property is preserved by taking lifts in the universal cover (and in
  fact, even in arbitrary covers).  Since $\gamma_A$ is one of these
  lifts, and since $\tilde{S}$ is simply connected, any two paths
  between the same vertices in $\tilde{M}$ are homotopic, and we
  conclude that $\gamma_A$ is a geodesic between any pair of points it
  visits, which is the definition of a biinfinite geodesic.
\end{proof}

We may therefore define the function $\td_A:V(\tilde{M}) \to \Z$ as follows:
\begin{itemize}
\item if $a=0$, then we let $\td_A$ be the distance in $\tilde{M}$ to the vertex $\tilde{x}_A$,
\item if $a>0$, then we let $\td_A=B_{\gamma_A}$ be the Busemann function associated
  with $\gamma_A$.
\end{itemize}
Let us record the relation
\begin{equation}\label{eq:1}
  \td_A(Av)=\td_A(v)-2a
\end{equation}
valid for any vertex $v$ of $\tilde{M}$, where $A\in \mathrm{Aut}(p)$
is the automorphism defined in Section~\ref{univcovsec}. This relation
is immediate in the case $a=0$; for $a>0$ it follows from the
definition of the Busemann function $B_{\gamma_A}$ and the fact that
$A(\gamma_A(t))=\gamma_A(t+2a)$ for any $t$.

In a completely similar manner, we define the functions $\td_B$,
$\td_C$ and $\td_{B'}$, which obey relations similar to~\eqref{eq:1}
mutatis mutandis. For later use, we record the following:
\begin{lem}
  \label{lem:ABCbuse}
  There exists a constant $k \in \Z$ such that, for any vertex $v$ of
  $\tilde{M}$, we have
  \begin{equation}
    \label{eq:2}
    \td_{B'}(Av) = \td_B(v) + k
  \end{equation}
  and
  \begin{equation}
    \label{eq:3}
    \td_B(Cv) = \td_{B'}(v) - k + 2b.
  \end{equation}
\end{lem}

\begin{proof}
  The first relation is a straightforward consequence of the fact that
  $\tdel_{B'}=A\tdel_B$. The constant $k$ is equal to $0$ in the case
  $b=0$ while, for $b>0$, it satisfies $A\gamma_B(k)=\gamma_{B'}(0)$
  (we could thus set it to $0$ too by choosing the reference points on
  $\tdel_B$ and $\tdel_{B'}$ appropriately).

  For the second relation, we note that $\tdel_B=C\tdel_{B'}$ hence
  there exists a constant $k'$ such that $\td_B(Cv) = \td_{B'}(v)+k'$.
  But, by the relation $B=A^{-1} C^{-1}$, we have
  \begin{equation}
    \td_B(v)=\td_B(CABv)=\td_{B'}(ABv)+k'=
    \td_{B}(Bv)+k'+k =\td_B(v)+k'+k-2b.
  \end{equation}
  hence $k'=-k+2b$.
\end{proof}

\begin{rem}
  In this paper, we only consider the Busemann functions obtained by
  following the contours of ideal faces counterclockwise. We surmise
  that considering the clockwise orientation might be useful to study
  ``strictly tight'' boundaries, i.e.\ boundaries whose contours are
  the unique paths of minimal length in their homotopy class.
\end{rem}

\subsection{Leftmost bigeodesics and the diangle lemma}
\label{sec:univcovbigeod}

Throughout this section, we consider specifically the pair of ideal
vertices/faces $(\tdel_A,\tdel_B)$, but all the discussion can be
adapted to any other pair, e.g.\ $(\tdel_A,\tdel_C)$,
$(\tdel_B,\tdel_{B'})$, etc.  We start by adapting the concepts of
Section~\ref{sec:geoddefs} to the context of the infinite map
$\tilde{M}$ and of Busemann functions.

\paragraph{Geodesics and bigeodesics.}
A \emph{geodesic towards $\tdel_A$} is a path on $\tilde{M}$ along
which $\td_A$ is strictly decreasing, which stops at $\tdel_A$ if
$a=0$, and which continues forever if $a>0$ (so that $\td_A$ tends to
$-\infty$ along the path).  Such a path $\gamma$ is indeed a geodesic,
i.e.\ satisfies $\td(\gamma(t),\gamma(t'))=|t-t'|$ for all $t,t'$ in
its interval of definition.
A geodesic towards $\tdel_B$ is defined similarly.

A \emph{bigeodesic between $\tdel_A$ and $\tdel_B$} is a path which is
both a geodesic towards $\tdel_A$ in one direction, and a geodesic
towards $\tdel_B$ in the other direction. Such bigeodesics always
exist: this is clear when $b=0$ (start at the vertex $\tdel_B$ and
follow a path along which $\td_A$ decreases) or similarly when $a=0$;
when $ab>0$, we may invoke Lemma~\ref{lem:gdamdef}, with $V_0$ the set
of vertices lying in the square $S_\varnothing$ (given a good
embedding).  We define the \emph{distance $\td_{AB}$ between $\tdel_A$
  and $\tdel_B$} as
\begin{equation}
  \label{eq:dAB}
  \td_{AB} := \min_v \left( \td_A(v) + \td_B(v) \right).
\end{equation}
This denomination is consistent with the fact that, for $a=b=0$,
$\td_{AB}$ is precisely the graph distance in $\tilde{M}$ between
$\tilde{x}_A$ and $\tilde{x}_B$. For $ab>0$, it is nothing but the
minimal value of the function $B_{\gamma_A}+B_{\gamma_B}$ as
considered in Lemma~\ref{lem:gdamdef}.

In the Poincaré disk representation of $\tilde{S}$, a geodesic towards
$\tdel_A$ forms a simple path which ``ends'' at the ideal point
$\tilde{x}_A$, see again Figure~\ref{fig:UnivcovPoincare}. A
bigeodesic between, say, $\tdel_A$ and $\tdel_B$, forms a simple path
connecting the ideal points $\tilde{x}_A$ and $\tilde{x}_B$. By
planarity, this path splits the disk (hence $\tilde{M}$) in two
regions, which we may distinguish as \emph{left} and \emph{right},
given an orientation of the bigeodesic (say, from $\tdel_B$ to
$\tdel_A$). Note that the interiors of these regions may be
disconnected, if the bigeodesic passes through an ideal vertex on its
way.

\paragraph{Geodesic vertices.}

A \emph{geodesic vertex between $\tdel_A$ and $\tdel_B$} is a vertex
belonging to a bigeodesic between $\tdel_A$ and $\tdel_B$. It is
straightforward to check (see again Lemma~\ref{lem:gdamdef} in the
case $ab>0$) that $v$ is such a vertex if and only if
\begin{equation}
  \label{eq:geodvertcond}
  \td_A(v)+\td_B(v) = \td_{AB}.
\end{equation}
The quantity $\td_A(v)$ is called the \emph{$\td_A$-latitude} of the
geodesic vertex $v$ between $\tdel_A$ and $\tdel_B$.

\begin{lem}
  \label{lem:fingeodvert}
  For any $r \in \Z$, the set $I_{AB}(r)$ of geodesic vertices between
  $\tdel_A$ and $\tdel_B$ having $\td_A$-latitude $r$ is finite.
\end{lem}

\begin{proof}
  If $a=b=0$ then this is a corollary of
  Lemma~\ref{lem:finvert}. Suppose now that $a>0$. Given a good
  embedding, any bigeodesic between $\tdel_A$ and $\tdel_B$ must visit
  a vertex in the square $S_{\al^n}$, for any $n\geq 0$. Let us denote
  by $r_n$ the maximal $\td_A$-latitude of a geodesic vertex belonging
  to $S_{\al^n}$, then $r_{n+1}=r_n-2a<r_n$ by \eqref{eq:1}. Hence
  there exists an $m$ such that $r_{m}<r$. If $b=0$, we see that any
  geodesic vertex of $\td_A$-latitude $r$ belongs to a geodesic (hence
  simple) path between $\tdel_B$ and a vertex of $S_{\al^{m}}$, so
  there are finitely many of them by Lemma~\ref{lem:finvert}. If
  $b>0$, adapting the previous reasoning shows that there exists an
  $\ell$ such that the minimal $\td_A$-latitude of a geodesic vertex
  in $S_{\bl^{\ell}}$ is larger than $r$. Therefore, any geodesic
  vertex of $\td_A$-latitude $r$ belongs to a geodesic (hence simple)
  path between a vertex of $S_{\al^{m}}$ and a vertex of
  $S_{\bl^{\ell}}$, and again there are finitely many of them by
  Lemma~\ref{lem:finvert}.
\end{proof}

\paragraph{Leftmost bigeodesic via a geodesic vertex.}

Given a geodesic vertex $v$ between $\tdel_A$ and $\tdel_B$, we may,
as in Section \ref{sec:basic}, ``launch'' from $v$
the leftmost geodesics towards $\tdel_A$ and $\tdel_B$ (the planarity
of $\tilde{M}$ is used to identify the first edges of these leftmost
geodesics, as in the finite case). The only potential difficulty is
that these geodesics may encounter ideal vertices of infinite
degree. However, by Lemma~\ref{lem:fingeodvert}, only finitely many
edges through a given ideal boundary point make $\td_A$ and $\td_B$
decrease, therefore, when there are any, there is always a leftmost
one to be picked. 
Since $v$ is a geodesic vertex, the concatenation
of these two geodesics, oriented all the way from $\tdel_B$ to $\tdel_A$, forms 
a bigeodesic between $\tdel_A$ and $\tdel_B$ which we call 
the \emph{leftmost bigeodesic via $v$} and denote $\GG_{AB}(v)$. 
Note that the leftmost geodesic towards $\tdel_A$
eventually merges with it: this is obvious for $a=0$ since the ideal
vertex $\tilde{x}_A=\tdel_A$ is reached in finitely many steps; for
$a>0$, setting $\gamma_A$ as in Section~\ref{sec:hyperbolicgeod}, then
the leftmost geodesic from $v$ merges with $\gamma_A$ at the vertex
$\gamma_A(t)$ for the smallest value of $t$ such that
$\td_A(v)=\td(v,\gamma_A(t))-t$. Of course, a similar property holds
for the leftmost geodesic towards $\tdel_B$.

\paragraph{The diangle lemma.}

\begin{figure}[t]
  \centering 
  \fig{.6}{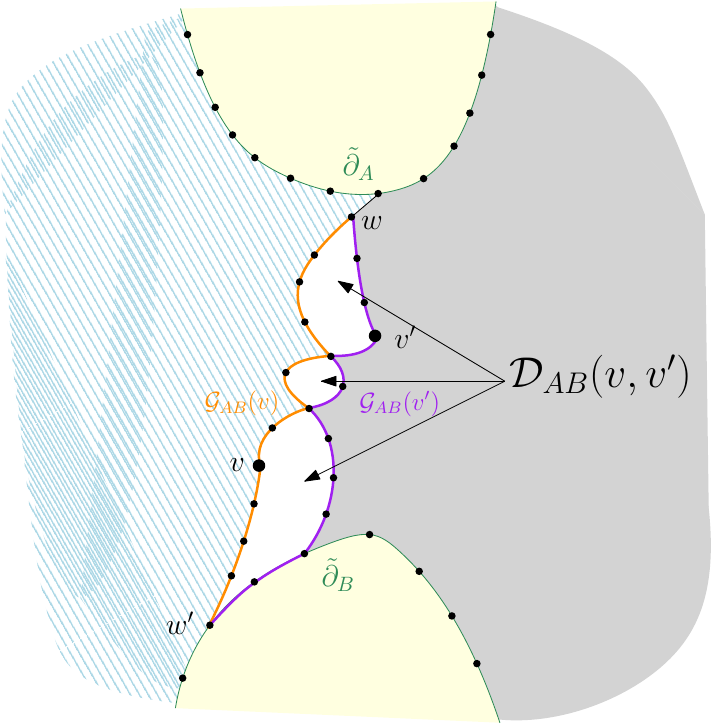}
  \caption{Illustration of the diangle lemma: we consider two geodesic
    vertices $v$ and $v'$ between $\tdel_A$ and $\tdel_B$, such that
    $v'$ lies on the right of the leftmost bigeodesic $\GG_{AB}(v)$
    (shown in orange), and such that $\td_A(v) \geq \td_A(v')$. Then,
    $\GG_{AB}(v)$ and $\GG_{AB}(v')$ (shown in purple) delimit a
    region $\DD_{AB}(v,v')$ (shown in white) which is a bigeodesic
    diangle of nonnegative exceedance $\td_A(v) - \td_A(v')$.}
  \label{fig:non-crossing-bigeod2}
\end{figure}

Consider two geodesic vertices $v$ and $v'$ between $\tdel_A$ and
$\tdel_B$, such that $v'$ lies on the right of the leftmost bigeodesic
$\GG_{AB}(v)$ (oriented from $\tdel_B$ to $\tdel_A$) or on it, and
such that $\td_A(v) \geq \td_A(v')$. This situation is illustrated on
Figure~\ref{fig:non-crossing-bigeod2}.

We claim that the leftmost bigeodesic $\GG_{AB}(v')$ remains on the
right of $\GG_{AB}(v)$. Indeed, as we start from $v'$ and follow a
geodesic towards $\tdel_A$, it is not possible to pass to the left of
$\GG_{AB}(v)$, since the latter consists of a leftmost geodesic
towards $\tdel_A$ and we started on its right. Furthermore, when we
actually follow the leftmost geodesic from $v'$ towards $\tdel_A$,
then we will eventually meet $\GG_{AB}(v)$ at a vertex $w$ (since all
leftmost geodesics towards $\tdel_A$ eventually merge with it), and
follow it onwards. Similarly, as we start from $v'$ and follow the
leftmost geodesic towards $\tdel_B$, it is again not possible to pass
to the left of $\GG_{AB}(v)$, hence $\GG_{AB}(v')$ stays to the right
of $\GG_{AB}(v)$. In particular, $v$ lies on the left of
$\GG_{AB}(v')$, hence on the right of $\GG_{BA}(v')$ (oriented from 
$\tdel_A$ to $\tdel_B$). As we have
$\td_B(v') \geq \td_B(v)$, we see that $v$ and $v'$ play a completely
symmetric role, upon exchanging the roles of $A$ and $B$ and viewing
Figure~\ref{fig:non-crossing-bigeod2} upside-down. We denote by $w'$
the vertex at which $\GG_{AB}(v)$ and $\GG_{AB}(v')$ merge when going
towards $\tdel_B$. Note that it is possible that $\GG_{AB}(v)$ and
$\GG_{AB}(v')$ have intermediate contacts at vertices of
$\td_A$-latitude strictly included between $\td_A(v')$ and
$\td_A(v)$.

We now consider the closed region $\DD_{AB}(v,v')$ delimited by
$\GG_{AB}(v)$ and $\GG_{AB}(v')$ (precisely, the region which is on
the right of $\GG_{AB}(v)$ and on the left of $\GG_{AB}(v')$, when
orienting them from $\tdel_B$ to $\tdel_A$), which we prune at $w$ and
$w'$ to remove their (possibly infinite) common parts towards
$\tdel_A$ and $\tdel_B$. Note that $\DD_{AB}(v,v')$ is connected but
its interior may be disconnected when there are intermediate contacts
between $\GG_{AB}(v)$ and $\GG_{AB}(v')$. As a degenerate case, it is
possible to have $\GG_{AB}(v)=\GG_{AB}(v')$, and then $\DD_{AB}(v,v')$
consists of a segment joining $v=w'$ to $v'=w$.

\begin{lem}[Diangle lemma]
  \label{lem:diangle}
  $\DD_{AB}(v,v')$ is a bigeodesic diangle of nonnegative exceedance
  $\td_A(v)-\td_A(v')$.
\end{lem}

\begin{proof}
  We have to check that $\DD_{AB}(v,v')$ satisfies the axioms of
  Section~\ref{sec:bigeo-diangles}. First we observe that it is by
  construction a finite map with one boundary-face (the interior of a
  closed cycle on $\tilde{M}$ always contains finitely many vertices,
  edges and faces).

  Then, to make the correspondence with the notations of
  Section~\ref{sec:bigeo-diangles}, we take $w_{12}=v$, $w_{21}=v'$,
  $v_1=w$ and $v_2=w'$ (see again Figure~\ref{fig:diangle}) and the
  corners $c_1,c_{12},c_2,c_{21}$ are selected in a natural
  manner. The boundary intervals $[c_1,c_2]$ and $[c_2,c_1]$ are
  geodesic since they correspond to parts of the bigeodesics
  $\GG_{AB}(v)$ and $\GG_{AB}(v')$.  The boundary intervals
  $[c_{12},c_2]$ and $[c_{21},c_1]$ are strictly geodesic since
  $\GG_{AB}(v)$ and $\GG_{AB}(v')$ are actually \emph{leftmost}
  bigeodesics: no geodesic between $v$ and $w'$ can enter into
  $\DD_{AB}(v,v')$, and similarly between $v'$ and $w$. Finally, $w$
  (resp.\ $w'$) is by definition the only vertex common to
  $[c_{21},c_1]$ and $[c_1,c_2]$ (resp.\ $[c_{12},c_2]$ and
  $[c_2,c_1]$).
\end{proof}

So far, our construction depends on the choice of geodesic vertices
$v$ and $v'$ satisfying the aforementioned properties that $v'$ is on
the right of $\GG_{AB}(v)$ and that $\td_A(v) \geq \td_A(v')$.
However, given two latitudes $r \geq r'$ such that the sets
$I_{AB}(r)$ and $I_{AB}(r')$, as defined in
Lemma~\ref{lem:fingeodvert}, are both nonempty, there exists a
canonical choice of such vertices. Indeed, we may consider the
\emph{leftmost} element $v$ of $I_{AB}(r)$, defined as the only vertex
$v \in I_{AB}(r)$ such that the region on the left of $\GG_{AB}(v)$
(again oriented from $\tdel_B$ to $\tdel_A$) contains no other element
of $I_{AB}(r)$. Similarly, we choose $v'$ to be the \emph{rightmost}
element of $I_{AB}(r')$. Clearly $v'$ is on the right of
$\GG_{AB}(v)$, which actually passes through the leftmost
element of $I_{AB}(r')$. This choice of $v$ and $v'$ makes
$\DD_{AB}(v,v')$ the largest possible, as in the case of
triply pointed maps discussed in Section~\ref{sec:gluinginvtp}, and we
will always encounter such maximal diangles in the following.

\subsection{Equilibrium vertices and the triangle lemma}
\label{sec:trianglelemma}

In the previous subsection, we have only considered the pair
$(\tdel_A,\tdel_B)$. Let us now add $\tdel_C$ in the game: our purpose
is to construct a bigeodesic triangle $\TT_{ABC}$ which,
interestingly, is canonical in the sense that it is entirely
determined by the triplet of distinguished ideal corners
$(\tilde{x}_A,\tilde{x}_B,\tilde{x}_C)$ of $\tilde{S'}$. Indeed, as
discussed in Section~\ref{univcovsec}, this triplet distinguishes
the triplet $(\tdel_A,\tdel_B,\tdel_C)$ of ideal vertices/faces of
$\tilde{M}$.

Recall the definition~\eqref{eq:dAB} of the distance $\td_{AB}$
between $\tdel_A$ and $\tdel_B$, and define $\td_{BC}$ and $\td_{CA}$
similarly. Inspired by the equilibrium conditions~\eqref{eq:rdtp},
we define $r_A$, $r_B$ and $r_C$ by
\begin{equation}
  \label{eq:rdtpbis}
  \td_{AB} = r_A + r_B, \quad \td_{BC}= r_B+r_C, \quad \td_{CA}=r_C+r_A.
\end{equation}
Note that $r_A$, $r_B$ and $r_C$ may now be negative, since the
``renormalized'' distances $\td_{AB}$, $\td_{BC}$ and $\td_{CA}$ may
be negative in the presence of ideal faces. From the very definition
of Busemann functions and from the bipartiteness of $\tilde{M}$, we
get that the quantity $\td_A(v)+\td_B(v)$ has the same parity for all
$v$, which is also necessarily the parity of $\td_{AB}$.  We
immediately deduce that $\td_{AB}+\td_{BC}+\td_{CA}$ is even, hence
$r_A$, $r_B$ and $r_C$ are integers.

We claim that the sets $I_{AB}(r_A)$ (as defined in
Lemma~\ref{lem:fingeodvert}), $I_{BC}(r_B)$ and $I_{CA}(r_C)$ are
always nonempty: this is true when $\tdel_A,\tdel_B,\tdel_C$ are all
ideal faces, since the bigeodesics between them are infinite and
therefore pass through geodesic vertices of any latitude; this is also
true when $\tdel_A,\tdel_B,\tdel_C$ are all ideal vertices, as
$I_{AB}(r_A)$ projects to the set $S_{AB}$ considered in
Figure~\ref{fig:threepointdecomp} and similarly for the other sets;
the other cases are left to the reader.

Furthermore, even though we had to choose a reference point along
$\tdel_A$ to define the Busemann function $\td_A$ when $a>0$, and
similarly for $\tdel_B$ and $\tdel_C$, the sets $I_{AB}(r_A)$,
$I_{BC}(r_B)$ and $I_{CA}(r_C)$ \emph{do not depend} on these choices.
Indeed, as we change the reference point along $\tdel_A$, say, $\td_A$ is
changed by a constant, but $r_A$ gets changed by the same constant.
We also have the identification $I_{AB}(r_A)=I_{BA}(r_B)$.

We now define $v_{AB}$ to be the rightmost element of $I_{AB}(r_A)$
(again orienting from $\tdel_B$ to $\tdel_A$), and define $v_{BC}$ and
$v_{CA}$ similarly by permuting $A,B,C$ cyclically. We call $v_{AB}$,
$v_{BC}$ and $v_{CA}$ \emph{equilibrium vertices}, as they satisfy
\begin{equation}
  \begin{split}
    \td_A(v_{AB})&=\td_A(v_{CA})=r_A,\\ \td_B(v_{AB})&=\td_B(v_{BC})=r_B,\\
    \td_C(v_{BC})&=\td_C(v_{CA})=r_C.
  \end{split}
\label{eqvert}
\end{equation}
By the previous paragraph, the equilibrium vertices are intrinsic to
$\tilde{M}$.

Consider now the leftmost bigeodesics $\GG_{AB}(v_{AB})$,
$\GG_{BC}(v_{BC})$ and $\GG_{CA}(v_{CA})$. We will show that they
delimit a region $\TT_{ABC}$ which is a bigeodesic triangle. For this,
we first state some technical lemmas.

\begin{lem}
  \label{lem:crossing}
  Let $v$ be a vertex strictly to the left of the leftmost bigeodesic
  $\GG_{AB}(v_{AB})$. Then, we have $\td_C(v)>r_C$.
\end{lem}

\begin{proof}
  Since $\tilde{x}_C$ is on the right of $\GG_{AB}(v_{AB})$, any
  geodesic from $v$ towards $\tdel_C$ must cross $\GG_{AB}(v_{AB})$,
  at a vertex denoted $w$ such that $\td_C(v)>\td_C(w)$. As $w$ is a
  geodesic vertex between $\tdel_A$ and $\tdel_B$, we have
  $\td_A(w)+\td_B(w)=\td_{AB}=r_A+r_B$, and therefore we have either
  $\td_A(w) \leq r_A$ or $\td_B(w) \leq r_B$. In the former case, the
  definition of $\td_{CA}$ implies that
  $\td_C(w)+\td_A(w) \geq \td_{CA}=r_C+r_A$ hence $\td_C(w) \geq
  r_C$. The same conclusion holds in the latter case, using rather
  $\td_{BC}$.
\end{proof}

\begin{cor}
  \label{cor:trianglealternative}
  Unless we have $v_{AB}=v_{BC}=v_{CA}$, the vertices $v_{BC}$ and
  $v_{CA}$ are both strictly to the right of $\GG_{AB}(v_{AB})$.
\end{cor}

\begin{proof}
  Since $v_{BC}$ and $v_{CA}$ are both at $\td_C$-latitude $r_C$, they
  cannot be strictly on the left of $\GG_{AB}(v_{AB})$ by the previous
  lemma. If one of them, say $v_{BC}$, is on $\GG_{AB}(v_{AB})$, then
  it is equal to $v_{AB}$ since these two vertices are at the same
  $\td_B$-latitude $r_B$. But then, $v_{AB}=v_{BC}$ belongs to
  $I_{CA}(r_C)$, and is clearly the rightmost element $v_{CA}$ of that
  set.
\end{proof}

Reasoning as in the discussion of the diangle lemma, we deduce from
Corollary~\ref{cor:trianglealternative} that $\GG_{BC}(v_{BC})$ and
$\GG_{CA}(v_{CA})$ remain on the right of
$\GG_{AB}(v_{AB})$. Furthermore, the two bigeodesics
$\GG_{AB}(v_{AB})$ and $\GG_{CA}(v_{CA})$ merge at a vertex $v_A$ when
following them towards $\tdel_A$, but are disjoint before. Similarly,
we introduce the merging vertices $v_B$ and $v_C$ of the bigeodesics
going towards $\tdel_B$ and $\tdel_C$, respectively.

We now consider the cycle on $\tilde{M}$ obtained by following
$\GG_{AB}(v_{AB})$ from $v_A$ to $v_B$, then $\GG_{BC}(v_{BC})$ until
$v_C$, and finally $\GG_{CA}(v_{CA})$ until we return to $v_A$. It is
a simple counterclockwise cycle which delimits a region denoted
$\TT_{ABC}$. See the left of Figure~\ref{fig:ProcedureI} for an
illustration (ignoring the right of this figure for now). Note that
$\TT_{ABC}$ is reduced to a single vertex if and only if
$v_{AB}=v_{BC}=v_{CA}$. This situation happens when there exists a
vertex which is a geodesic vertex between $\tdel_A$ and $\tdel_B$,
between $\tdel_B$ and $\tdel_C$, and between $\tdel_C$ and $\tdel_A$,
all at the same time. Such a vertex, if it exists, is necessarily
unique by planarity\footnote{Two such vertices may exist on the sphere
  (consider the situation $v_{AB}=v_{BC}=v_{CA}$ and
  $v_{BA}=v_{CB}=v_{AC}$ on Figure~\ref{fig:threepointdecomp}), but
  here we are in the disk with $\tdel_A$, $\tdel_B$ and $\tdel_C$
  on the (ideal) boundary.}.

\begin{lem}[Triangle lemma]
  \label{lem:triangle}
  $\TT_{ABC}$ is a bigeodesic triangle.
\end{lem}

\begin{proof}
  To make the correspondence with the notations of
  Section~\ref{sec:bigeo-triangle}, we take $v_1=v_A$,
  $v_{12}=v_{AB}$, etc.\ (see again Figure~\ref{fig:triangle}), and the
  corners $c_1,c_{12},\ldots$ are selected in a natural manner.  It is
  then straightforward to check that $\TT_{ABC}$ satisfies all the
  axioms defining bigeodesic triangles, since it is delimited by
  leftmost bigeodesics, and every geodesic from $v_A$ to $v_B$ inside
  $\TT_{ABC}$ must pass through $v_{AB}$ as we chose it to be the
  rightmost element of $I_{AB}(r_A)$, and similarly for $v_{BC}$ and
  $v_{CA}$.
\end{proof}

We have constructed the bigeodesic triangle $\TT_{ABC}$ associated
with the triplet of ideal corners
$(\tilde{x}_A,\tilde{x}_B,\tilde{x}_C)$ of $\tilde{S}'$, but we can
adapt the construction to any other triplet. Note however that our
construction depends at several places (e.g.\ in
Lemma~\ref{lem:crossing}) on the fact that $\tilde{x}_A$,
$\tilde{x}_B$ and $\tilde{x}_C$ appear in counterclockwise order along
the ideal boundary of $\tilde{S}'$, and we shall therefore assume the
same order for other triplets. Specifically, we will consider later on
the triangles $\TT_{ACB'}$, $\TT_{ABB'}$ and $\TT_{BCB'}$
corresponding to such triplets.

\subsection{Decomposing a map of type I}
\label{sec:decI}

\begin{figure}[t]
  \centering
  \fig{.8}{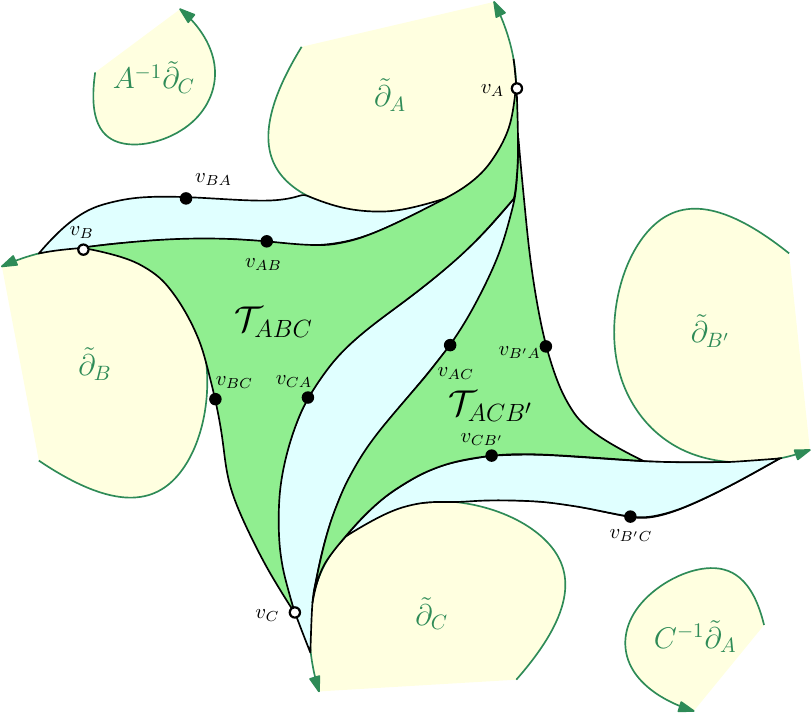}
  \caption{Illustration of the decomposition of a map of type
    I. Cutting along the leftmost bigeodesics launched from the
    equilibrium vertices (shown in black), we delimit the two geodesic
    triangles $\TT_{ABC}$ and $\TT_{ACB'}$ (shown in green), and three
    geodesic diangles of nonnegative exceedances (shown in light
    blue). Altogether, these five regions form the domain
    $\MM_{ABCB'}$ of Lemma~\ref{lem:nonoverlapstrongI}, containing
    exactly one preimage of each inner face of $M$.}
  \label{fig:ProcedureI}
\end{figure}

Let us consider the bigeodesic triangles $\TT_{ABC}$ and $\TT_{ACB'}$,
as constructed in the previous subsection. From the fact that
$\tilde{x}_A$, $\tilde{x}_B$, $\tilde{x}_C$ and $\tilde{x}_{B'}$ are
the four ideal corners of a fundamental domain of $\tilde{S}$, it is
tempting to identify $\TT_{ABC}$ and $\TT_{ACB'}$ with the two
triangles appearing in the assembling procedure. For this, we need to
make sure that they do not overlap. Recall the notations from
Section~\ref{sec:trianglelemma}, which we complete with notations
pertaining to the triangle $\TT_{ACB'}$: we let $r'_A$, $r'_C$ and
$r'_B$ be the integers such that
\begin{equation}
  \label{eq:rdtpter}
  \td_{AC} = r'_A + r'_C, \quad \td_{CB'}= r'_C+r'_B, \quad \td_{B'A}=r'_B+r'_A,
\end{equation}
and we let $v_{AC}$, $v_{CB'}$ and $v_{B'A}$ be the corresponding
equilibrium vertices, see Figure~\ref{fig:ProcedureI}. Then, the
triangles $\TT_{ABC}$ and $\TT_{ACB'}$ do not overlap if the leftmost
bigeodesic $\GG_{AC}(v_{AC})$ remains on the right of
$\GG_{AC}(v_{CA})$ (oriented from $\tdel_C$ to $\tdel_A$). We may
ensure this by invoking the diangle lemma (Lemma~\ref{lem:diangle}),
and more precisely by identifying the region delimited by these two
bigeodesics with the diangle $\DD_{AC}(v_{CA},v_{AC})$ as defined
(mutatis mutandis) in Section~\ref{sec:univcovbigeod}. But, for this,
the assumption of the diangle lemma must be satisfied, namely we must
have $\td_A(v_{CA}) \geq \td_A(v_{AC})$, i.e.\ $r_A \geq r'_A$.

\begin{lem}
  \label{lem:latdiffexc}
  The difference $r_A-r'_A$ of $\td_A$-latitude between $v_{CA}$ and $v_{AC}$
  is equal to $a+c-b$.
\end{lem}

\begin{proof}
  By the equilibrium conditions~\eqref{eq:rdtpbis} and
  \eqref{eq:rdtpter}, we have
  \begin{equation}
    \label{eq:distdiff}
    r_A-r_A'=\frac{(\td_{AB}-\td_{B'A})-(\td_{BC}-\td_{CB'})}{2}.
  \end{equation}
  Therefore, we must compare $\td_{AB}$ to $\td_{B'A}$, and $\td_{BC}$
  to $\td_{CB'}$. For this, we will use the symmetries of $\tilde{M}$.
  Consider first the action of $A$ on the Busemann functions.
  By~\eqref{eq:1} and~\eqref{eq:2}, we have
  \begin{multline}
    \td_{B'A}=\min_v \left(\td_{B'}(v)+\td_A(v)\right) = \min_v\left(\td_{B'}(Av)+\td_A(Av)\right)\\
   = \min_v\left(\td_{B}(v)+k+\td_A(v)-2a\right)
   = \td_{AB} +k-2a.
 \end{multline}
 Similarly, considering now the action of $C$, the analog
 of~\eqref{eq:1} for $C$ and \eqref{eq:3} imply
 \begin{equation}
   \td_{BC}=\min_v \left(\td_{B}(v)+\td_{C}(v)\right) = \min_v\left(\td_B(Cv)+\td_{C}(Cv)\right)
   = \td_{CB'} - k + 2b -2c. 
 \end{equation}
  Plugging these relations into~\eqref{eq:distdiff} gives the wanted
  difference $a+c-b$.
\end{proof}

We conclude that $\TT_{ABC}$ and $\TT_{ACB'}$ do not overlap when
$a+c-b \geq 0$. However, as we are trying to find a decomposition of
$M$ (and not just $\tilde{M}$), we actually want the stronger property
that the \emph{projections} of $\TT_{ABC}$ and $\TT_{ACB'}$ on $M$ do
not overlap. This property is ensured by the following:

\begin{lem}
  \label{lem:nonoverlapstrongI}
  Let $v_{BA}:=A^{-1}v_{B'A}$ and $v_{B'C}:=C^{-1}v_{BC}$, and let
  $\MM_{ABCB'}$ be the domain delimited by the bigeodesics
  $\GG_{AB}(v_{BA})$, $\GG_{BC}(v_{BC})$, $\GG_{CB'}(v_{B'C})$ and
  $\GG_{B'A}(v_{B'A})$, pruned from their common parts (see again
  Figure~\ref{fig:ProcedureI}). Then, $\MM_{ABCB'}$ contains exactly
  one preimage of each inner face of $M$.

  Furthermore, when $M$ is of type I, then both triangles $\TT_{ABC}$
  and $\TT_{ACB'}$ are contained in $\MM_{ABCB'}$, and their
  complement consists of the three geodesic diangles
  $\DD_{AC}(v_{CA},v_{AC})$, $\DD_{AB}(v_{BA},v_{AB})$ and
  $\DD_{B'C}(v_{CB'},v_{B'C})$, with nonnegative exceedances equal to
  $a+c-b$, $b+a-c$ and $c+b-a$, respectively.
\end{lem}

\begin{proof}
  The first claim means that $\MM_{ABCB'}$ is essentially a
  fundamental domain for the action of $\mathrm{Aut}(p)$. Indeed, the
  bigeodesics delimiting $\MM_{ABCB'}$ can be viewed as paths
  connecting the ideal points $\tilde{x}_A$, $\tilde{x}_B$,
  $\tilde{x}_C$ and $\tilde{x}_{B'}$: the bigeodesic
  $\GG_{AB}(v_{BA})$ connects $\tilde{x}_A$ and $\tilde{x}_B$ and the
  bigeodesic $\GG_{B'A}(v_{B'A})$ is (upon reversing its orientation)
  its image by $A$, connecting $\tilde{x}_A$ and
  $\tilde{x}_{B'}$. Similarly, $\GG_{CB'}(v_{B'C})$ connects
  $\tilde{x}_C$ and $\tilde{x}_{B'}$ and $\GG_{BC}(v_{BC})$ is (upon
  reversing its orientation) its image by $C$ connecting $\tilde{x}_C$
  and $\tilde{x}_B$. This pattern mimics precisely that of the four
  sides of the square $S_\varnothing$ in
  Figure~\ref{fig:Univcovcarre}, even though the topology of the
  quadrangle $\MM_{ABCB'}$ is not necessarily that of a square as it
  may have pinch points if two of its boundaries come into
  contact. The similitude with $S_\varnothing$ (whose sides do reach
  the ideal points) may be further improved by adding to
  $\MM_{ABCB'}$ the common parts of its boundary geodesics so as to
  eventually reach $\tilde{x}_A$, $\tilde{x}_B$, $\tilde{x}_C$ and
  $\tilde{x}_{B'}$ (possibly after infinitely many steps).  We are
  still left with a final (but somewhat irrelevant) slight difference
  with the situation of Section~\ref{univcovsec}: when, say,
  $\partial_A$ is a boundary-face, the projection on the sphere
  $p(\GG_{AB}(v_{BA}))=p(\GG_{B'A}(v_{B'A}))$ actually never reaches
  the puncture $x_A$, but rather wraps eventually around $\partial_A$
  forever. This issue can be fixed by stopping the path at the first
  time it hits $\partial_A$, replacing the final part with a segment
  entering inside $\partial_A$ to reach $x_A$ (and doing similar fixes
  at $\partial_B$ and $\partial_C$ if needed). All in all, the above
  differences do not concern the inner faces of $M$, which therefore
  lift to unique preimages in $\MM_{ABCB'}$.

  We now turn to the second claim. We have seen that $\TT_{ACB'}$ is
  on the right of $\TT_{ABC}$ when $a+c-b\geq 0$, and each of
  these triangles has one side in common with $\MM_{ABCB'}$ (namely  along
  $\GG_{BC}(v_{BC})$ for $\TT_{ABC}$ and along $\GG_{B'A}(v_{B'A})$ for 
  $\TT_{ACB'}$). Checking that both triangles are contained in $\MM_{ABCB'}$
  therefore boils down to checking that their other sides are well placed, namely that
  the boundary $\GG_{AB}(v_{AB})$ of $\TT_{ABC}$ is on the
  right of boundary $\GG_{AB}(v_{BA})$ of $\MM_{ABCB'}$ and that the boundary $\GG_{CB'}(v_{B'C})$ 
  of $\MM_{ABCB'}$ is on
  the right of the boundary $\GG_{CB'}(v_{CB'})$ of $\TT_{ACB'}$.
  But this can be done exactly in the same way as for proving that
  $\GG_{AC}(v_{AC})$ is on the right of $\GG_{AC}(v_{CA})$, via the
  diangle lemma. Indeed, the reasoning done at the beginning of this
  subsection---including Lemma~\ref{lem:latdiffexc}---pertains to the
  quadruplet of ideal corners
  $(\tilde{x}_A,\tilde{x}_B,\tilde{x}_C,\tilde{x}_{B'})$, and relies
  on the key relations
  $\tilde{x}_{B'}=A\tilde{x}_B=C^{-1}\tilde{x}_B$. But, at a
  fundamental level, $A$, $B$ and $C$ play a completely symmetric
  role, and redoing our reasoning with the quadruplets
  $(\tilde{x}_B,\tilde{x}_C,\tilde{x}_{A},\tilde{x}_{C'})$ and
  $(\tilde{x}_{B'},\tilde{x}_A,\tilde{x}_C,\tilde{x}_{A'})$, with
  $\tilde{x}_{C'}:=B\tilde{x}_C=A^{-1}\tilde{x}_C$ and
  $\tilde{x}_{A'}:=B'\tilde{x}_A=C^{-1}\tilde{x}_A$, we find that the
  remaining pieces of the puzzle of Figure~\ref{fig:ProcedureI} are
  the bigeodesic diangles $\DD_{AB}(v_{BA},v_{AB})$ and
  $\DD_{B'C}(v_{CB'},v_{B'C})$ of respective exceedances $b+a-c$ and
  $c+b-a$, which are indeed nonnegative since $M$ is assumed of type
  I.
\end{proof}

\begin{prop}
  \label{prop:typeIinv}
  The procedure which, to the map $M$ of type I, associates the two
  bigeodesic triangles $\TT_{ABC}$ and $\TT_{ACB'}$, and the three
  bigeodesic diangles $\DD_{AC}(v_{CA},v_{AC})$,
  $\DD_{AB}(v_{BA},v_{AB})$ and $\DD_{B'C}(v_{CB'},v_{B'C})$, is
  the inverse of the assembling procedure I.
\end{prop}

\begin{proof}
  By comparing Figures~\ref{fig:Univcovcarre_assembling} and
  \ref{fig:ProcedureI}, it is plain that disassembling a map $M$ of
  type I by cutting its universal cover $\tilde{M}$ along leftmost
  bigeodesics as in Figure~\ref{fig:ProcedureI}, then reassembling the
  pieces following procedure I as in
  Figures~\ref{fig:assembling_partial} and
  \ref{fig:Univcovcarre_assembling}, restores $M$ after projecting
  $\tilde{M}$ on $S$ by $p$.

  It remains to check that, conversely, if we assemble two triangles
  and three diangles together, then disassemble the result, we recover
  the original pieces. 
We thus start with two triangles and three diangles, and
  perform a partial gluing, as described in the Figure~\ref{fig:assembling_partial} of Section~\ref{univcovsec}.
We note that the obtained ``bigeodesic quadrangle'' $\MM$ is of the same form as that, $\MM_{ABCB'}$, displayed on
  Figure~\ref{fig:ProcedureI}. Working directly on the universal cover, we then have to glue 
  copies $\MM_\wl$ of $\MM$ along the scheme of Figure~\ref{fig:Univcovcarre_assembling}.

  Let us for now forget the decomposition interpretation of Figure~\ref{fig:ProcedureI} described in
  its caption and reinterpret it instead as the result of the procedure described in Figure~\ref{fig:Univcovcarre_assembling},
  once completed by a gluing of all the red and blue intervals facing each other.
We may then view the light blue and green domains in this figure as representing the copy $\MM_\varnothing$ of $\MM$, 
  together with its diangle/triangle components. We may also view the vertices $v_{AB}, \ldots$ as the associated attachment points of $\MM_\varnothing$ and of its internal components. 
  
  With this new interpretation
  of the figure, we already know from Lemma~\ref{lem:minpathnew} that the four sides of the quadrangle $\MM_\varnothing$ 
  lie along four geodesic paths in $\tilde{M}$: the bigeodesic denoted $\GG$ in Section~\ref{tightsec}, its image by $A$, 
  and a symmetric bigeodesic $\GG'$ (launched from the attachment point between the
copies $\MM_\varnothing$ and $\MM_{\overline{\cl}}$ and going towards
the ideal corners $\tilde{x}_C$ and $\tilde{x}_{B'}$) and its image by $C$. Clearly, $\GG$ is a bigeodesic
between $\tdel_A$ and $\tdel_B$ and it is in fact the leftmost bigeodesic $\GG_{AB}(v_{BA})$ launched from $v_{BA}$ (which is de facto a geodesic vertex). This is a straightforward consequence of the fact that, from $v_{BA}$ to $\tdel_A$ (respectively to 
$\tdel_B$), $\GG$ is glued to only red segments on its left. Similarly, $\GG'$ is the leftmost bigeodesic 
 $\GG_{CB'}(v_{B'C})$. Since $v_{B'A}=A v_{AB}$ and $v_{BC}=C v_{B'C}$, the two other sides of $\MM_\varnothing$ 
 are the leftmost bigeodesics  $\GG_{B'A}(v_{B'A})$ and $\GG_{BC}(v_{BC})$.
 From the definition of bigeodesic diangles and triangles, the paths passing via $v_{AB}$, $v_{CA}$, $v_{AC}$
 and $v_{CB'}$ are clearly leftmost bigeodesics within $\MM_\varnothing$ and coalesce with either $\GG$ or
 $\GG'$ outside of $\MM_\varnothing$, hence they are leftmost bigeodesics in $\tilde{M}$ which we
 thus identify as $\GG_{AB}(v_{AB})$, $\GG_{CA}(v_{CA})$, $\GG_{AC}(v_{AC})$ and $\GG_{CB'}(v_{CB'})$.
 
 To recover the original interpretation of Figure~\ref{fig:ProcedureI}, it remains to show that $v_{AB}$, $v_{BC}$
 and $v_{CA}$ are actually the equilibrium vertices for $\tdel_A$, $\tdel_B$ and
 $\tdel_C$, while  $v_{AC}$, $v_{CB'}$
 and $v_{B'A}$ are the equilibrium vertices for $\tdel_A$, $\tdel_C$ and
 $\tdel_{B'}$.  Since ${\cal{T}}_{ABC}$ is a bigeodesic triangle, we deduce that $v_{AB}$ and $v_{CA}$ have
 the same $\td_A$-latitude, $v_{AB}$ and $v_{BC}$ the same $\td_B$-latitude and $v_{BC}$ and $v_{CA}$ the
 same $\td_C$-latitude.  Since these vertices are all geodesic vertices (so that, e.g.\
 $\td_A(v_{AB})+\td_B(v_{AB})=\td_{AB}=r_A+r_B$), we deduce that they obey the relation 
 \eqref{eqvert} imposed on equilibrium vertices. Otherwise stated, the vertices
  $v_{AB}$, $v_{BC}$ and $v_{CA}$ belong to the respective sets
  $I_{AB}(r_A)$, $I_{BC}(r_B)$ and $I_{CA}(r_C)$, as defined in
  Section~\ref{sec:trianglelemma}.  Furthermore, since these vertices
  are the three attachment points of a bigeodesic triangle, they are
  necessarily the rightmost elements of their respective sets, hence
  they are indeed the equilibrium vertices associated with
  $(\tdel_A,\tdel_B,\tdel_C)$, and $\TT_{ABC}$ indeed coincides with
  the triangle constructed in
  Section~\ref{sec:trianglelemma}. Performing the same reasoning on
  the triangle denoted $\TT_{ACB'}$, we eventually recover precisely
  the original decomposition interpretation of Figure~\ref{fig:ProcedureI}, as described in its caption.
  Proposition~\ref{prop:typeIinv} follows.
\end{proof}

We conclude this section with two remarks. First, when $M$ has no
inner face, $\tilde{M}$ has only ideal faces and is actually a tree.
We find that the triangles $\TT_{ABC}$ and $\TT_{ACB'}$ are reduced to
single vertices, while the diangles are reduced to segments, thereby
inverting the assembling of Figure~\ref{fig:assembling_noface} for
type I.  Second, in the case $a=b=c=0$, the current disassembling
procedure coincides with that of
Section~\ref{sec:gluinginvtp}. Indeed, we only have ideal vertices in
this case, so that $\td_{AB}$ is just the graph distance between
$\tdel_A$ and $\tdel_B$ in $\tilde{M}$ and coincides with the graph
distance $d_{AB}$ between $\partial_A$ and $\partial_B$ in $M$, and
similarly for the other pairs of ideal vertices. The decomposition of
$\tilde{M}$ which we perform here just projects to the decomposition
of $M$ performed in Section~\ref{sec:gluinginvtp} (note that, in that
section, the vertices denoted $v_{AB}$, $v_{BC}$ and $v_{CA}$ are the
projections of those which we consider here, while $v_{BA}$, $v_{CB}$
and $v_{AC}$ are the projections of $v_{B'A}$, $v_{CB'}$ and $v_{AC}$
respectively).

\subsection{Decomposing a map of type II}
\label{sec:decII}

\begin{figure}[h!]
  \centering
  \fig{.8}{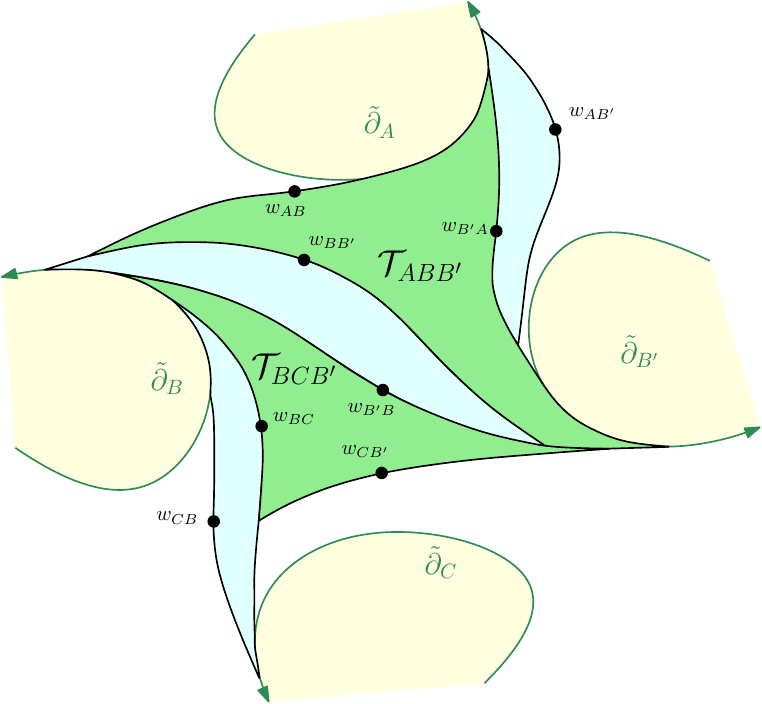}
  \caption{Illustration of the decomposition of a map of type II. The
    two geodesic triangles $\TT_{ABB'}$ and $\TT_{BCB'}$ (shown in
    green) and the three bigeodesic diangles (shown in light blue) form
    the domain $\MM'_{ABCB'}$ of
    Lemma~\ref{lem:nonoverlapstrongII}.}
  \label{fig:ProcedureII}
\end{figure}

Suppose now that we have a map $M$ of type II. Without loss of
generality, we may assume that $\partial_B$ is the longest boundary,
i.e.\ we have $b \geq a+c$. Then, the decomposition of the previous
subsection might fail, since now
$\td_A(v_{CA})-\td_A(v_{AC})=r_A-r'_A=a+c-b \leq 0$, hence it is not
possible in general to apply the diangle lemma (it could happen that
the bigeodesics $\GG_{AC}(v_{CA})$ and $\GG_{AC}(v_{AC})$ cross each
other).

Then, the trick is to ``perform a flip'' and, instead of the triangles
$\TT_{ABC}$ and $\TT_{ACB'}$, to consider rather the triangles
$\TT_{ABB'}$ and $\TT_{BCB'}$, see Figure~\ref{fig:ProcedureII}. The
equilibrium vertices of $\TT_{ABB'}$ (resp.\ $\TT_{BCB'}$) are denoted
$w_{AB}$, $w_{BB'}$ and $w_{B'A}$ (resp.\ $w_{BC}$, $w_{CB'}$ and
$w_{B'B}$).  We now have the following counterpart of
Lemma~\ref{lem:nonoverlapstrongI}:

\begin{lem}
  \label{lem:nonoverlapstrongII}
  Let $w_{AB'}:=Aw_{AB}$ and $w_{CB}:=Cw_{CB'}$, and let
  $\MM'_{ABCB'}$ be the domain delimited by the bigeodesics
  $\GG_{AB}(w_{AB})$, $\GG_{BC}(w_{CB})$, $\GG_{CB'}(w_{CB'})$ and
  $\GG_{B'A}(w_{AB'})$, pruned from their common parts (see again
  Figure~\ref{fig:ProcedureII}). Then, $\MM'_{ABCB'}$ contains exactly
  one preimage of each inner face of $M$.

  Furthermore, when $M$ is of type II, then both triangles
  $\TT_{ABB'}$ and $\TT_{BCB'}$ are contained in $\MM'_{ABCB'}$, and
  their complement consists of the three geodesic diangles
  $\DD_{BB'}(w_{B'B},w_{BB'})$, $\DD_{AB'}(w_{B'A},w_{AB'})$ and
  $\DD_{BC}(w_{CB},w_{BC})$, with nonnegative exceedances equal to
  $b-a-c$, $2a$ and $2c$, respectively.
\end{lem}

\begin{proof}
  The first claim is similar to that of
  Lemma~\ref{lem:nonoverlapstrongI}, and is proved in the same way.

  For the second claim, we apply again the diangle lemma, and all
  boils down to proving that
  \begin{equation}
    \label{eq:typeIIconds}
    \begin{split}
      \td_B(w_{B'B})-\td_B(w_{BB'})&=b-a-c,\\
      \td_A(w_{B'A})-\td_A(w_{AB'})&=2a,\\
      \td_C(w_{BC})-\td_C(w_{CB})&=2c.
    \end{split}
  \end{equation}
  For the first relation, by considering the equilibrium conditions in
  $\TT_{ABB'}$ and $\TT_{BCB'}$ we find that
  \begin{equation}
    \td_B(w_{B'B})-\td_B(w_{BB'})=\frac{(\td_{BC}-\td_{AB})-(\td_{CB'}-\td_{B'A})}{2}
  \end{equation}
  and we observe that the right-hand side is nothing but the opposite
  of that of~\eqref{eq:distdiff}. Thus, by Lemma~\ref{lem:latdiffexc},
  it is equal to $b-a-c$ as wanted. For the second relation
  of~\eqref{eq:typeIIconds}, we simply note that
  $\td_A(w_{AB})=\td_A(w_{B'A})$ by the definition of equilibrium
  vertices, and that $\td_A(w_{AB'})=\td_A(w_{AB})-2a$
  by~\eqref{eq:1}. For the third relation, we proceed in the same way,
  with $C$ instead of $A$.
\end{proof}

\begin{prop}
  \label{prop:typeIIinv}
  The procedure which, to the map $M$ of type II, associates the two
  bigeodesic triangles $\TT_{ABB'}$ and $\TT_{BCB'}$, and the three
  bigeodesic diangles $\DD_{BB'}(w_{B'B},w_{BB'})$,
  $\DD_{AB'}(w_{B'A},w_{AB'})$ and $\DD_{BC}(w_{CB},w_{BC})$, is the
  inverse of the assembling procedure II.
\end{prop}

The proof is entirely similar to that of
Proposition~\ref{prop:typeIinv}. Note that, when $M$ has no inner
face, $\tilde{M}$ has only ideal faces and is actually a tree.  We
find that the triangles $\TT_{ABB'}$ and $\TT_{BCB'}$ are reduced to
single vertices, while the diangles are reduced to segments, thereby
inverting the assembling of Figure~\ref{fig:assembling_noface} for
type II. The proof of Theorem~\ref{thm:main} is now complete.

\section{Equivalence with the Eynard-Collet-Fusy formula, and limiting
statistics in random maps with boundaries}
\label{sec:ecf}

In this section, we discuss the relation of our work with the
Eynard-Collet-Fusy (ECF) formula for (quasi-)bipartite maps with three
boundaries---see \cite[Proposition~3.3.1]{Eynard16} and
\cite{CoFu12}---and we show that it entails interesting properties of
the statistics of distances and areas in random maps and their scaling
limits.

Let $L_1,L_2,L_3$ be positive integers or half-integers whose sum is
an integer.  Let $G_{L_1,L_2,L_3}$ be the generating function of
essentially bipartite planar maps with three (non necessarily tight)
rooted boundary-faces of degrees $2L_1,2L_2,2L_3$, counted with a
weight $t$ per vertex and a weight $g_{2k}$ per inner face of degree
$2k$. Here, a boundary-face is said \emph{rooted} if one its incident
corners is distinguished.  We also let $R$ be the generating series
defined at \eqref{eq:Req}, and we introduce the notation
\begin{equation}
  \alpha(2L) :=
  \frac{(2L)!}{\lfloor L \rfloor!   \lfloor L-\frac{1}{2}
    \rfloor!}, \qquad L \in \frac{1}{2} \Z .
\end{equation}
The ECF formula states that 
\begin{equation}
  \label{eq:ColletFusy}
  G_{L_1,L_2,L_3} = \alpha(2L_1) \alpha(2L_2) \alpha(2L_3)
  \cdot R^{L_1+L_2+L_3}\frac{d \ln(R)}{d t}\, .
\end{equation}

We will show in Section~\ref{sec:rederivation} how one can recover
this formula from Theorem~\ref{thm:Tabc}. To this end, we first
establish in Section~\ref{sec:struct-outerm-minim} some facts about
the structure of minimal cycles homotopic to the boundaries in general
pairs of pants. For the record, and since this will be useful for the
probabilistic considerations of Section~\ref{sec:scalinglimits}, we
also recall the formula for the generating function $G_{L_1,L_2}$ of
essentially bipartite \emph{annular maps}, namely maps with two rooted
boundary-faces of lengths $2L_1$ and $2L_2$, with $L_1,L_2$ positive
integers or half-integers whose sum is an integer:
\begin{equation}
  \label{eq:ColletFusy2}
  G_{L_1,L_2} =\frac{ \alpha(2L_1) \alpha(2L_2)  }{L_1+L_2}  \cdot R^{L_1+L_2}\, .
\end{equation}
This formula appears in various places in the literature, for instance
it is the case $r=2$ of~\cite[Theorem~1.1]{CoFu12}, see also
\cite[Proposition~4]{BuddRMGFFnotes} or
\cite[Theorem~3.12]{CurienSFnotes}, and
\cite[Equation~(2.1)]{BouttierHDR} for a derivation based on the slice
decomposition.

\subsection{The structure of outermost minimal separating cycles}\label{sec:struct-outerm-minim}

Let $M$ be a planar map with three boundary-faces
$\partial_1,\partial_2,\partial_3$, that are \emph{not} supposed to be
tight.  We consider the problem of finding cycles homotopic to the
boundaries $\partial_i$ of $M$, with minimal length.  The optimization
problems of finding shortest paths with certain topological
constraints on surfaces have been investigated in the literature on
effective geometry and computer science. In particular, some of the
ideas used in this section are similar to \cite{CdVLa07,CdVEr10}. We
mention that the discussion below generalizes easily to maps with more
than three boundaries.

As in Section \ref{sec:universal-cover-maps}, we may and will assume
that $M$ is a map on the triply punctured sphere $S$, and we denote
the punctures by $x_1,x_2,x_3$.  For $i\in \{1,2,3\}$, let
$\CC^{(i)}_{\min}(M)$ be the set of cycles in $M$ that are freely
homotopic to the contour of $\partial_i$ in the punctured sphere $S$,
and that have minimal possible length. This minimal length will be
denoted by $2\ell_i(M)$, where $\ell_i(M)$ is a positive integer or
half-integer. Note that, when $M$ is essentially bipartite,
$2\ell_i(M)$ has the same parity as the degree of $\partial_i$.  Our
goal is to collect a number of facts about the structure of
$\CC^{(i)}_{\min}(M)$. In particular, we will see that it carries a
natural order relation which makes it a lattice. 

As in the preceding sections, it will be useful to work on the universal cover $\tilde{M}$ of $M$ introduced in
Section \ref{sec:universal-cover-maps}, which is a map on the universal cover $p:\tilde{S}\to S$. We will however use a slightly different index notation, replacing $\partial_A,\partial_B,\partial_C$ with $\partial_1,\partial_2,\partial_3$, $x_A,x_B,x_C$ with $x_1,x_2,x_3$. 
The automorphisms $A,B,C$ are renamed $A_1,A_2,A_3$, and the distinguished ideal corners $\tilde{x}_A,\tilde{x}_B,\tilde{x}_C$ are renamed $\tilde{x}_1,\tilde{x}_2,\tilde{x}_3$, see for instance Figure \ref{fig:Univcovcarre} for a reminder of the former notation
(we will not use $\tilde{x}_{B'}$ here). 

The boundary-faces of $M$ are lifted in $\tilde{M}$ to faces of
infinite degrees, similarly to Figure \ref{fig:Univcovcarremap}, but whose contours are now not necessarily geodesic, and
not even necessarily simple curves, since the boundaries of $M$ are
not assumed to be tight.  From the ``concrete'' construction of
$\tilde{M}$ in Section \ref{sec:universal-cover-maps}, for
$i\in \{1,2,3\}$, we can naturally distinguish a particular lift of
$\partial_i$ by choosing $\tdel_i$ to be the infinite face of
$\tilde{M}$ that is incident to the ideal boundary point
$\tilde{x}_i$.  Note that $\tdel_i$ is invariant under the
automorphism $A_i$.  If $c^i$ is a cycle that is homotopic to the
contour $\hat{\partial}_i$ of $\partial_i$, then we can reason exactly
as in the beginning of the proof of
Lemma~\ref{lem:tightn-bound-result-new}, which did not make use of the
fact that the boundaries are tight. Namely, we let $\gamma$ be a path
from some arbitrary point $x$ on $\hat{\partial}_i$ to some point $z$
on $c^i$ such that $\gamma c^i\gamma^{-1}$ and $\hat{\partial}_i$ are
homotopic as loops rooted at $x$, and then lift those two paths to
obtain a path $\tilde{c}^i_0$ from a vertex $\tilde{z}$ in $\tilde{M}$
to $A_i\tilde{z}$ that is a lift of $c^i$. The concatenation of the
paths $A_i^n\tilde{c}^i_0,n\in \Z$, seen up to increasing
reparametrization,
is then a biinfinite path
$\tilde{c}^i$ whose projection via $p$ is a path that circles
indefinitely around $c^i$, and which we call a {\em biinfinite lift}
of $c^i$. Moreover, $\tilde{c}^i$ is invariant under $A_i$. Let $U_0$
be the finite set of words $\wl$ such that $\tilde{c}^i_0$ visits the
domains $S_\wl$, in the former notation of Section
\ref{univcovsec}. Then by invariance under $A_i$, $\tilde{c}^i$ visits
$A_i^nS_\wl,n\in \Z,\wl\in U_0$. Now, from the way in which domains
are connected together, we see that all domains of the form
$A_i^nS_\varnothing,n\in \Z$ must be visited, so that $\tilde{c}^i$
remains at bounded ``distance'' from $\tdel_i$ (where distance is
measured in terms of number of domains $S_\wl$ to cross).  In this
sense, $\tilde{c}^i$ converges to the ideal boundary point
$\tilde{x}_i$, that is incident to the infinite face
$\tdel_i$. Consequently, if $j\neq i$ and $c^i,c^j$ are two cycles
that are respectively freely homotopic to the contours of $\partial_i$
and $\partial_j$, their biinfinite lifts $\tilde{c}^i,\tilde{c}^j$
defined as above converge to two different ideal boundary points.

Now, 
if $\hat{c}^i$ is another biinfinite lift of $c^i$,
there is an automorphism $W$ such that
$\hat{c}^i=W\tilde{c}^i$. If $W\in \{A_i^n,n\in \Z\}$, we have
$\hat{c}^i=\tilde{c}^i$, and otherwise, $\hat{c}^i$ is a distinct
infinite path that is invariant under the automorphisms
$WA_i^nW^{-1},n\in \Z$. In the latter case, this infinite path visits
domains of the form $WA_i^nS_{ \wl }$ where $n\in \Z$ and $\wl$
belongs to a finite set of words, which implies that $\hat{c}^i$
converges to the ideal corner $W \tilde{x}_i$, distinct from
$\tilde{x}_i$. For this reason, we can single out the biinfinite lift
$\tilde{c}^i$ constructed above, which converges to the distinguished
ideal boundary vertex $\tilde{x}_i$, and call it the \emph{canonical
  biinfinite lift} of $c^i$.

Now fix $i\in \{1,2,3\}$, let $c^i\in \CC^{(i)}_{\min}(M)$, and let
$\hat{c}^i$ be any biinfinite lift of $c^i$ in $\tilde{M}$, passing
through an arbitrary point $\hat{z}$ projecting to a point of
$c^i$. Similarly to Lemma \ref{sec:infin-geod-busem} above, the
minimality of the length of $c^i$, and Proposition 2.5 and Lemma 2.1
in \cite{CdVEr10}, imply the following result.  
 
 \begin{lem}\label{lem:biinfgeod}
 The path $\hat{c}^i$ is a biinfinite geodesic in $\tilde{M}$,
 separating $\tilde{S}$ into two connected components.  
 \end{lem}

 Note that we can view the two components 
 of $\tilde{M}\setminus \hat{c}^i$ as the ``left'' and ``right'' 
 component, since $\tilde{M}$ is oriented, and we can assume that
 $c^i$ circles counterclockwise around the puncture $x_i$.   
 Given the fact that the
 latter (which we view as a point ``outside'' the surface)
 belongs to the region of the complement of $c^i$ located to 
 its left, we call the left region of $\tilde{M}\setminus
 \hat{c}^i$ the outer domain of $\hat{c}^i$, and the right region the
 inner domain.  Note that
 both domains of $\hat{c}^i$ determine $\hat{c}^i$ as their boundaries, by Jordan's
 theorem. 

 We now define a partial order relation on $\CC^{(i)}_{\min}(M)$. Let $c^i_1,c^i_2\in
 \CC^{(i)}_{\min}(M)$, and $\tilde{c}^i_1,\tilde{c}^i_2$ be their
 canonical biinfinite lifts. We write $c^i_1\preceq^{(i)} c^i_2$ if the outer domain of
 $\tilde{c}_1^i$  
 is included in the outer domain of $\tilde{c}_2^i$. The fact that
 this indeed defines a partial order is easy and left to the reader. 
 
 \begin{lem}\label{latticeprop}
 If $c_1^i,c_2^i$ are elements of $\CC^{(i)}_{\min}(M)$, and if $\tilde{c}_1^i,\tilde{c}_2^i$ are
 their canonical lifts, then
 the intersection and union of their outer domains are simply
 connected and bounded by two paths which we denote by
 $\tilde{c}_1^i\wedge \tilde{c}_2^i$ and $\tilde{c}_1^i\vee \tilde{c}_2^i$ . In turn, these two paths
 project via $p$ to two cycles $c_1^i\wedge c_2^i$ and $c_1^i\vee c_2^i$ on $M$, which are the infimum and supremum of
 $\{c_1^i,c_2^i\}$ for the order $\preceq^{(i)}$. In particular,
 $(\CC^{(i)}_{\min}(M),\preceq^{(i)})$ is a lattice.    
 \end{lem}

 Again, the proof of this statement is easy and relies on the observation that
 two consecutive
 intersections of the paths $\tilde{c}^i_1$ and $\tilde{c}^i_2$
 must arise in increasing order for the parametrization of both paths,
 and be linked by arcs of same lengths, by the geodesic property. 

 As a consequence, since $(\CC^{(i)}_{\min}(M),\preceq^{(i)})$ is clearly a finite lattice, it
 admits a smallest element $c_{(i)}$, called the outermost minimal
 cycle homotopic to the contour of $\partial_i$. It also admits a
 maximal element, although we are not going to use this 
 one in the sequel. As a final observation, we state the
 following.   

 \begin{lem}\label{lem:nevercross}
   For every $i, j\in \{1,2,3\}$, the cycles $c_{(i)},c_{(j)}$
   do not cross each other, but may however have edges in common. For
   $i=j$ this means that the cycle $c_{(i)}$ cannot be self-crossing,
   but may possibly visit some edges more than once. 
 \end{lem}

\begin{rem}\label{sec:struct-outerm-minim-1}
Though intuitively clear, this lemma has some subtelty to it. In
particular, the fact that $c_{(i)}$ may visit some edges twice does
happen, see for instance the contour of $\partial_B$ on the example
displayed in Figure \ref{fig:Univcovcarremap}.  
The proof of the lemma consists in showing that the outer domains of
two distinct biinfinite lifts of $c_{(i)}$ and $c_{(j)}$ cannot
overlap. For $i=j$, this entails that the self-contacts of
$c_{(i)}$, if there are any, can occur only ``from the inner
side''. More precisely, the outer region of $c_{(i)}$, which is
comprised of the edges of $\tilde{M}$ that can be attained from
$\partial_i$ without crossing $c_{(i)}$, can be seen as a map with two
boundaries, one of which is the contour of $\partial_i$, and the other
being a \emph{simple} boundary that results from cutting along
$c_{(i)}$. This remark will be useful in the next section.  
\end{rem}

 \begin{proof}
   In this proof, we let $A=A_i$ to simplify the notation.  If
   $c_{(i)}$ and $c_{(j)}$ cross each other, then we can lift them
   into biinfinite geodesics $\tilde{c}_{(i)},\hat{c}_{(j)}$ in the
   universal cover that also cross each other, where we choose
   the first lift to be the canonical lift of $c_{(i)}$, which is
   invariant under $A$. Note
   that $\hat{c}_{(j)}$ is not necessarily the canonical lift of
   $c_{(j)}$, as it is determined by the choice of a lift of an
   intersection point of $c_{(i)}$ and $c_{(j)}$. 
   Necessarily, these two geodesic
   biinfinite  lifts must converge to two distinct ideal boundary points.
   Therefore, whenever $\hat{c}_{(j)}$ enters the outer region of
   $\tilde{c}_{(i)}$, it has to eventually leave it. This
   reasoning holds also when $i=j$ and $c_{(i)}$ has a self-crossing,
   upon noting that $\tilde{c}_{(i)}$ and $\hat{c}_{(i)}$
   denote two lifts of $c_{(i)}$ which are necessarily different, since they are
   simple paths that cross each other.
   
   So there is then a subpath of $\hat{c}_{(j)}$ within the outer
   region of $\tilde{c}_{(i)}$, say between the two points $\tilde{x}$
   and $\tilde{z}$, meaning that $\hat{c}_{(j)}$ is entirely
   contained in the open outer region of $\tilde{c}_{(i)}$ between
   these points. Now assume without loss of generality that
   $\tilde{z}$ belongs to the portion of $\tilde{c}_{(i)}$ between
   $A^n\tilde{x}$ and $A^{n+1}\tilde{x}$ for some $n\geq 0$, and
   distinct from $A^n\tilde{x}$. If $n=0$ we arrive at a
   contradiction, because we can follow the arc of $\hat{c}_{(j)}$
   between points $\tilde{x}$ and $\tilde{z}$, then the arc of
   $\tilde{c}_{(i)}$ between $\tilde{z}$ and $A\tilde{x}$ to form a
   new geodesic arc $\tilde{c}$ between $\tilde{x}$ and $A\tilde{x}$
   that is contained in the (closed) outer domain of
   $\tilde{c}_{(i)}$, and therefore projects via $p$ to a cycle that
   is strictly smaller than $c_{(i)}$ in $\mathcal{C}^{(i)}_{\min}(M)$.
   
   \begin{figure}
  \centering
  \fig{1}{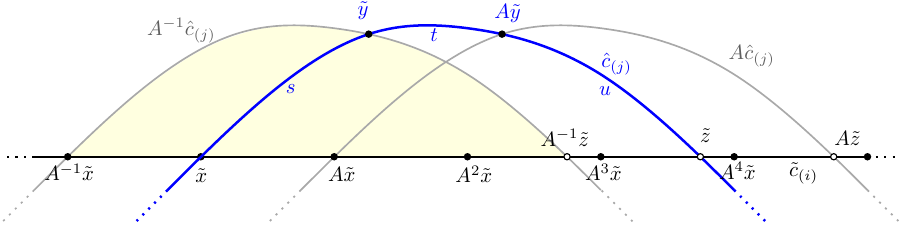}
  \caption{Illustration of the proof of Lemma
    \ref{lem:nevercross}. The bottom black line is the biinfinite
    geodesic $\tilde{c}_{(i)}$, whose outer domain, represented as the
    upper-half of the picture, is visited by the blue path
    $\hat{c}_{(j)}$. The proof consists in showing that the portion of
    the blue arc between $\tilde{y}$ and $A\tilde{y}$ projects via $p$
    to a cycle of $\mathcal{C}^{(i)}_{\min}(M)$ which is strictly
    smaller than its minimal element, a contradiction.  }
  \label{fig:noncrossingouter}
\end{figure}
   
Next, let us assume that $n\geq 1$. Let $2\ell=2\ell_i(M)$ be the
length of the cycle $c_{(i)}$, that is in particular the length of the
arc of $\tilde{c}_{(i)}$ between $A^m\tilde{x}$ and $A^{m+1}\tilde{x}$
for every $m\in \Z$. In particular, the length of the arc of
$\tilde{c}_{(i)}$ between $\tilde{x}$ and $\tilde{z}$ equals
$2n\ell+\ell'$ for some $\ell'\in (0,2\ell]$, and this is also the
distance between its extremities since $\tilde{c}_{(i)}$ is a geodesic
path.

Then the arc of $\hat{c}_{(j)}$ between $\tilde{x}$ and $\tilde{z}$
enters the Jordan domain (see the yellow domain in Figure
\ref{fig:noncrossingouter}) formed by the arcs of $\tilde{c}_{(i)}$
and $A^{-1}\hat{c}_{(j)}$ between the points $A^{-1}\tilde{x}$ and
$A^{-1}\tilde{z}$, and has to leave it through some point $\tilde{y}$
since $\tilde{z}$ is not in this domain. Since we assumed that the arc
of $\hat{c}_{(j)}$ between $\tilde{x}$ and $\tilde{z}$ is entirely
contained in the outer domain of $\tilde{c}_{(i)}$, it must be that
$\tilde{y}$ belongs to the arc of $A^{-1}\hat{c}_{(j)}$ between
$A^{-1}\tilde{x}$ and $A^{-1}\tilde{z}$. Now by applying the
automorphism $A$, $A\tilde{y}$ has to be the intersection point of the
arc of $\hat{c}_{(j)}$ between $\tilde{x}$ and $\tilde{z}$ with the
arc of $A\hat{c}_{(j)}$ between $A\tilde{x}$ and $A\tilde{z}$. In
particular, $\hat{c}_{(j)}$ contains an arc (between $\tilde{y}$ and
$A\tilde{y}$) that is strictly contained in the outer domain of
$\tilde{c}_{(i)}$. Since one of its extremity is the image of the
other by $A$, this arc projects to a cycle $c$ of $M$ that is
homotopic to the boundary $\partial_i$. So if we can show that its
length is $2\ell$ (in fact, we will show that this length can be at
most $2\ell$, which is even better!), we will obtain that $c$ is in
$\mathcal{C}^{(i)}_{\min}(M)$ but $c \prec^{(i)} c_{(i)}$, a contradiction.
   
So let $t$ be the length of the arc of $\hat{c}_{(j)}$ between
$\tilde{y}$ and $A\tilde{y}$, and let $s,u$ be the lengths of the arcs
between $\tilde{x}$ and $\tilde{y}$, and between $A\tilde{y}$ and
$\tilde{z}$ respectively. Then $u$ is also the length of the arc of
$A^{-1}\hat{c}_{(j)}$ between $\tilde{y}$ and $A^{-1}\tilde{z}$.  

Now the length of the concatenation of the arc of $\hat{c}_{(j)}$
between $\tilde{x}$ and $\tilde{y}$, the arc of
$A^{-1}\hat{c}_{(j)}$ between $\tilde{y}$ and $A^{-1}\tilde{z}$, and
the arc of $\tilde{c}_{(i)}$ between $A^{-1}\tilde{z}$ and $\tilde{z}$
equals $s+u+2\ell$, and has to be greater than or equal to the
distance between its extremities $\tilde{x}$ and $\tilde{z}$, which is
$2n\ell+\ell'$ as mentioned above. Therefore, we obtain $s+u\geq
2(n-1)\ell+\ell'$. On the other hand, since $\hat{c}_{(j)}$ is a
geodesic path, this distance $2n\ell+\ell'$ is also equal to $s+t+u$,
which is at least $t+2(n-1)\ell+\ell'$. So we obtain that $t\leq
2\ell$, as wanted.   
 \end{proof}

\subsection{A re-derivation of the ECF formula for pairs of pants}\label{sec:rederivation}

We now use the above discussion to introduce a bijective
decomposition of planar maps with three boundary-faces which are not
necessarily tight. Let $M$ be such a map, with boundaries denoted
$\partial_1,\partial_2,\partial_3$ as before. For every $i=1,2,3$, we
let $c_{(i)}$ be the minimal element of
$(\mathcal{C}^{(i)}_{\min}(M),\preceq^{(i)})$, as defined in the
previous subsection.  By Lemma~\ref{lem:nevercross}, these cycles
cannot cross, so they split the map $M$ into four parts
$M^{(0)},M^{(1)},M^{(2)},M^{(3)}$ where, for $i=1,2,3$, $M^{(i)}$ is
the part of $M$ delimited by the (non self-crossing) cycle $c_{(i)}$
and containing the face $\partial_i$, while $M^{(0)}$ is the remainder
of the map delimited by the three cycles
$c_{(1)},c_{(2)},c_{(3)}$. By Remark \ref{sec:struct-outerm-minim-1},
for $i\in \{1,2,3\}$, we may view $M^{(i)}$ as
a map with two boundaries, one being given by the contour of
$\partial_i$, and the other being a simple boundary resulting from
cutting along $c_{(i)}$. 
  
In order to be consistent with the ECF formula, the boundary-faces of
$M$ are assumed to be rooted. This induces a canonical rooting of the
boundary-faces of $M^{(0)}$ by, say, considering the leftmost shortest
path starting at the root of $\partial_i$ and ending on $c_{(i)}$.

To characterize the resulting maps, we need an extra definition: a
boundary-face is said \emph{strictly tight} if its contour is the
unique cycle of minimal length in its homotopy class. A strictly tight
boundary-face is to a tight boundary-face what a (red) strictly geodesic
boundary interval is to a (blue) geodesic boundary interval, as we defined in
Section~\ref{sec:geoddefs}.

\begin{prop}\label{prop:minimalbij}
  Let $L_1,L_2,L_3,l_1,l_2,l_3$ be positive integers or half-integers.
  Then, the mapping $M\mapsto(M^{(0)},M^{(1)},M^{(2)},M^{(3)})$ is a
  bijection between:
  \begin{itemize}
  \item the set of planar maps $M$ which have three rooted
    boundary-faces of lengths $2L_1$, $2L_2$, $2L_3$ and whose minimal
    separating cycles have lengths $2\ell_i(M)=2l_i$,
  \item and the set of quadruplets $(M^{(0)},M^{(1)},M^{(2)},M^{(3)})$
    made of a planar map $M^{(0)}$ with three rooted tight
    boundary-faces of lengths $2l_1,2l_2,2l_3$, and of three annular
    maps $M^{(1)},M^{(2)},M^{(3)}$ where, for $i=1,2,3$, the map
    $M^{(i)}$ has a rooted boundary of length $2L_i$ and a strictly
    tight unrooted boundary of length $2l_i$.
  \end{itemize}
  The map $M$ is essentially bipartite if and only if
  $M^{(0)},M^{(1)},M^{(2)},M^{(3)}$ are essentially bipartite. (In
  this case, $M$ exists if and only if $L_1+L_2+L_3$, $L_1-l_1$,
  $L_2-l_2$ and $L_3-l_3$ are all nonnegative integers.)
\end{prop}

\begin{proof}
  The tightness properties of the boundaries of
  $M^{(0)},M^{(1)},M^{(2)},M^{(3)}$ result from the fact that
  $c_{(i)}$ is the minimal element of $\mathcal{C}^{(i)}_{\min}(M)$
  for $i=1,2,3$.

  The mapping is a bijection since we may conversely (re)assemble a
  map $M$ from a quadruplet.  Note that there are a priori
  $(2l_1)\cdot(2l_2)\cdot(2l_3)$ ways to perform the (re)assembling,
  but only one of them is consistent with the rootings. Indeed, in
  $M^{(i)}$ we consider the leftmost shortest path starting with the
  root and ending on the unrooted boundary, and this singles out the
  position at which we must align the root on the $i$-th boundary of
  $M^{(0)}$.
\end{proof}

We refer to the areas (number of faces) of the annular maps
$M^{(1)},M^{(2)},M^{(3)}$ as the \emph{exterior areas} of $M$, denoted
by $A_1(M),A_2(M),A_3(M)$ respectively, and to the area of the map
$M^{(0)}$ as the \emph{interior area} of $M$, denoted by $A_0(M)$. In
this way, the exterior area $A_i(M)$ ($i=1,2,3$) is the minimal area
bounded by the contour of $\partial_i$ and by a cycle homotopic to it
of minimal possible length\footnote{As was pointed to one of the
  authors by Marco Mazzucchelli, this area is related to the notion of
  \emph{flat topology} (here ``flat'' stands for the musical symbol
  $\flat$) on homology classes introduced by Whitney and Federer in
  geometric measure theory, see for instance the Introduction in
  \cite{MaNe16}.}.

For the purposes of the next section, we also consider the case where
$M$ is an annular map whose two boundary-faces $\partial_1,\partial_2$
are rooted and not necessarily tight. The contours of the two
boundaries are now homotopic, hence there is now a single set
$\mathcal{C}_{\min}(M)$ of separating cycles of minimal length (this
length is denoted $2\ell(M)$). In this set we may find two ``extremal
cycles'', closest to $\partial_1$ and to $\partial_2$ respectively. By
splitting $M$ along these two cycles, we obtain three pieces
$M^{(0)},M^{(1)},M^{(2)}$, where $M^{(0)}$ is a map with two rooted
tight boundary-faces both of length $2\ell(M)$, while $M^{(1)}$ and
$M^{(2)}$ have one rooted boundary-face and another strictly tight
unrooted boundary-face of length $2\ell(M)$. This decomposition yields
a bijection analogous to that of Proposition~\ref{prop:minimalbij}.
We define the interior and exterior areas of $M$ accordingly: for
$i=0,1,2$, we let $A_i(M)$ be the area of $M^{(i)}$.

We now recall known enumerative results about the annular maps with
(strictly) tight boundaries considered above. We mention that these
results may be obtained by specializing \cite[Theorem~34]{BeFu12b} or
\cite[Equations~(9.18) and (9.19)]{irredmaps}, which deal with the
more general setting of maps with girth/irreducibility constraints,
but we refer to~\cite[Section~2.2]{BouttierHDR} for a more elementary
presentation in the current setting. It uses a slice decomposition
expressed in the universal covers of the annular maps, similarly to
the present paper.

\begin{prop}[{see e.g. \cite[Theorem~2.1]{BouttierHDR}}]
  \label{prop:enumannular}
  Let $L,l$ be positive integers or half-integers such that $L-l$ is a
  nonnegative integer. Then, the generating function of essentially
  bipartite annular maps with a rooted boundary-face of degree $2L$
  and a strictly tight unrooted boundary-face of degree $2l$, counted
  with a weight $g_{2k}$ per inner face of degree $2k$ and a weight
  $t$ per vertex not incident to the unrooted boundary, is given by
  $\binom{2L}{L-l}R^{L-l}$, where $R$ is as usual defined by
  \eqref{eq:Req}.

  The generating function of essentially bipartite annular maps with
  two tight rooted boundary-faces of degree $2l$, counted with a
  weight $g_{2k}$ per inner face of degree $2k$ and a weight $t$ per
  vertex is equal to $2l R^{2l}$.
\end{prop} 

We now re-derive the ECF formula \eqref{eq:ColletFusy}.  Indeed, the
last two propositions imply that the generating function
$G_{L_1,L_2,L_3}$ of essentially bipartite planar maps with three
rooted boundary-faces of lengths $2L_1,2L_2,2L_3$ is equal to
\begin{equation*}
  \sum_{\substack{2l_1,2l_2,2l_3>0\\L_i-l_i \in \Z}}\binom{2L_1}{L_1-l_1}R^{L_1-l_1}
  \binom{2L_2}{L_2-l_2}R^{L_2-l_2}\binom{2L_3}{L_3-l_3}R^{L_3-l_3}
  (2l_1)(2l_2)(2l_3)T_{l_1,l_2,l_3}
\end{equation*}
where the factors $2l_i$ account for the extra rooting of the faces in
the tight maps counted by $T_{l_1,l_2,l_3}$. Using
Theorem~\ref{thm:Tabc} and the hypergeometric identity
$\sum_{2l>0}(2l)\binom{2L}{L-l} =\alpha(2L)$,
we recover precisely~\eqref{eq:ColletFusy}.

Similarly, the generating function of essentially
bipartite annular maps with two rooted boundary-faces of lengths
$2L_1,2L_2$ is equal to
\begin{equation}
  \label{eq:annulG}
  G_{L_1,L_2} = 
  \sum_{\substack{2l>0\\L_1-l \in \Z}} \binom{2L_1}{L_1-l}R^{L_1-l}
  \binom{2L_2}{L_2-l}R^{L_2-l} (2l) R^{2l} 
\end{equation}
and we recover~\eqref{eq:ColletFusy2} by another hypergeometric
identity.

\subsection{Scaling limits of separating loop statistics}\label{sec:scalinglimits}

In this section we show how our results can be used to deduce
statistical properties of random annular maps and pairs of pants with a large area. This will be done by deriving a scaling limit result for the minimum cycle lengths, as well as for the exterior and interior areas of the associated decomposition. For simplicity, we focus on the simplest case of bipartite quadrangulations, for which $g_{2k}=g\delta_{k,2}$, although our results should have extensions to much more general models of random maps. 

Let us recall some classical probability densities: 
\begin{equation}
  p_a(x)=\frac{1}{\sqrt{2\pi a}}e^{-x^2/2a}\, ,\qquad x\in \R
\end{equation}
the Gaussian density (of variance $a>0$),
\begin{equation}
  q_x(a)=\frac{x}{a}p_a(x)=\frac{x}{\sqrt{2\pi a^3}}e^{-x^{2}/2a}\, ,\qquad a> 0
\end{equation}
the stable-$1/2$ density (with parameter $x>0$), and 
\begin{equation}
  r_a(x)=x\sqrt{\frac{2\pi}{a}}p_a(x)=\frac{x}{a}e^{-x^2/2a}\, ,\qquad x\geq 0
\end{equation}
the size-biased Gaussian absolute value density, also known as Rayleigh density (with parameter $a>0$). All the random variables considered below will be defined on some common probability space $(\Omega,\FF,\P)$. 

\begin{thm}\label{thcylinders}
  Let $M^n$ be a uniformly random quadrangulation with two rooted
  boundaries of lengths $2L_1^n$ and $2L_2^n$ and $n$ inner faces,
  where the integers $L_i^n,i\in \{1,2\}$ satisfy
  $L_i^n\sim \LL_i\sqrt{2n}$ as $n\to\infty$ for some
  $\LL_1,\LL_2\in (0,\infty)$.  Then we have the following convergence in
  distribution for the rescaled minimal half-length of a separating
  cycle and for the rescaled exterior/interior areas of $M^n$:
\begin{equation}\label{convergencecylinder}
\left(\Big(\frac{2}{n}\Big)^{1/4}\ell(M^n),\frac{A_1(M^n)}{n},\frac{A_0(M^n)}{\sqrt{2n}}\right)
\longrightarrow (\RR,\AA,\BB).
\end{equation}
Here $\RR$ and $\AA$ are independent random variables, $\RR$ follows a Rayleigh law of parameter $\LL_{\mathrm{eff}}=(\LL_1^{-1}+\LL_2^{-1})^{-1}$ and $\AA$ has 
density $q_{\LL_1}(a)q_{\LL_2}(1-a)/q_{\LL_1+\LL_2}(1)$ for $a\in (0,1)$. Finally, the random variable $\BB$ has conditional density $q_{\RR}$ given $(\RR,\AA)$, and in particular it is independent of $\AA$. 
\end{thm}

\begin{rem}
  Note that this theorem implies in particular that $A_0(M^n)/ n\to 0$ in probability, and so  $(A_1(M^n)+A_2(M^n))/n\to 1$ in probability. One can also be more explicit by computing the Laplace transform of the Rayleigh random variable, which yields
  \begin{equation}
  \E[e^{-u \BB}]=\sqrt{2\pi}(2u \LL_{\mathrm{eff}})^{1/4}e^{u \LL_{\mathrm{eff}}}\Pi(\sqrt{2u \LL_{\mathrm{eff}}})\, ,
\end{equation}
where $\Pi(x)=\int_x^\infty p_1(y)d y$ is the Gaussian tail distribution function. 
\end{rem}

\begin{thm}\label{thpants}
Let $M^n$ be a uniformly random quadrangulation with three rooted boundaries of lengths $2L_1^n,2L_2^n$ and $2L_3^n$ and $n$ inner faces, where the integers $L_i^n,i\in \{1,2,3\}$ satisfy $L_i^n\sim \LL_i\sqrt{2n}$ as $n\to\infty$ for some $\LL_1,\LL_2,\LL_3\in (0,\infty)$. 
Then we have the following convergence in distribution for the rescaled minimal half-lengths of separating cycles in each homotopy class and for the rescaled exterior/interior areas of $M^n$:
\begin{equation}
\left(\Big(\frac{2}{n}\Big)^{1/4}\ell_i(M^n),\frac{A_i(M^n)}{n}\right)_{i\in \{1,2,3\}}
\longrightarrow(\RR_i,\AA_i)_{i\in \{1,2,3\}}\, .
\end{equation}
Here $\RR_1,\RR_2,\RR_3$ are independent random variables, respectively with Rayleigh distribution of parameter $\LL_1$, $\LL_2$ and $\LL_3$, and $(\AA_1,\AA_2,\AA_3)$ is a random vector, independent of $(\RR_1,\RR_2,\RR_3)$, 
with density 
\begin{equation}\label{densitypants}
\left(\int_0^1\frac{q_{\LL_1+\LL_2+\LL_3}(x)}{2\sqrt{1-x}}dx\right)^{-1}\frac{q_{\LL_1}(a_1)q_{\LL_2}(a_2)q_{\LL_3}(a_3)}{2\sqrt{1-a_1-a_2-a_3}}
\end{equation} on the simplex $\{(a_1,a_2,a_3)\in (0,\infty)^3:a_1+a_2+a_3< 1\}$. 
\end{thm}

\begin{rem}
  That \eqref{densitypants} is indeed a probability density is an easy
  exercise using the semigroup property of the stable densities. The
  reason why we put a constant $2$ in the denominator is that it
  allows to view it as the conditional probability density function of
  four independent random variables $(\xi_0,\xi_1,\xi_2,\xi_3)$ where
  $\xi_i$ follows a stable(1/2) law with parameter $\LL_i$ for
  $i\in \{1,2,3\}$, and $\xi_0$ follows a Beta(1,1/2) random variable,
  given the singular event that $\sum_{i=0}^3\xi_i=1$. Note also that
  the integral of the normalizing constant can be computed from the
  explicit form of $q_x(a)$, which after a change of variables
  $y=(1-x)/x$ yields
  $\int_0^1q_{\mathcal{L}}(x)dx/\sqrt{1-x}=\exp(-\mathcal{L}^2/2)$.
\end{rem}

\begin{rem}
Note that in both statements, we have the remarkable property that the areas of the regions cut by the minimal length curves homotopic to the boundaries are independent of these respective lengths. 
We also note that in the setting of Theorem \ref{thpants}, the quantity $\min(\ell_1(M),\ell_2(M),\ell_3(M))$ is also the length $\ell_{\min}(M)$ of the shortest non contractible cycle in $M$. An immediate consequence of this observation is that $(2/n)^{1/4}\ell_{\min}(M^n)$ converges in distribution to a random variable with Rayleigh distribution of parameter $\LL_{\mathrm{eff}}=(\LL_1^{-1}+\LL_2^{-1}+\LL_3^{-1})^{-1}$. Remarkably, this extends word by word  the conclusion of Theorem \ref{thcylinders}, and asks the question whether this further generalizes to maps with four boundaries or more. 
\end{rem}

\begin{rem}
It is tempting to believe that these two theorems have consequences for Brownian surfaces, that are the scaling limits in the Gromov-Hausdorff sense of (say) random quadrangulations with a fixed topology \cite{bettinelli16,BeMi20+}. In particular, we expect that for the Brownian annulus ($k=2$), which is known \cite{bettinelli16} to be homeomorphic to a two-punctured sphere, there is a unique cycle of minimal length separating the two boundaries, and that this cycle has a Rayleigh distribution with parameter $(\LL_1^{-1}+\LL_2^{-1})^{-1}$. By letting the size of the second boundary $\LL_2$ go to infinity, we naturally expect to find the \emph{infinite Brownian disk} with boundary length $\LL_1$ introduced in \cite{BaMiRa19}, a random metric space homeomorphic to the complement of the open unit disk in the plane. We conjecture that the shortest non-contractible loop in this space has length distributed as a Rayleigh law of parameter $\LL_1$. This would be relevant in work by Riera~\cite{Riera2022} on the isoperimetric profile of the Brownian plane, but will be investigated elsewhere. 
\end{rem}

Let us prove these results. Since we are focusing on the case of quadrangulations, from now on we will restrict the generating function
$R=R(t,(g_{2k},k\geq 1))$ to the special case $g_{2k}=g\delta_{k,2}$, 
yielding explicitly
\begin{equation}\label{Rquad}
R=\frac{1-\sqrt{1-12\, gt}}{6g}\, .
\end{equation} 
This implies the following well-known asymptotic
enumeration formulas: 
\begin{equation}\label{asympenum}
[g^n]R|_{t=1}\sim \frac{12^n}{\sqrt{\pi n^3}}\, ,\qquad [g^n]\frac{d
  \ln(R/t)}{d t}\Big|_{t=1}\sim\frac{12^n}{2\sqrt{\pi n}}\, .
  \end{equation}
From now on we will always implicitly assume that the vertex parameter $t$ is set to $1$. 
The coefficients of $R$ and its powers admit some convenient probabilistic representations that we recall quickly here. Let $P_n(k)=2^{-n}\binom{n}{(n+k)/2}$ be the probability that a simple random walk starting from $0$ equals $k$ at time $n$. It satisfies a local limit theorem (the summation index being due to parity reasons)
\begin{equation}\label{LLthP}\sum_{k\in 2\Z+n}|\sqrt{n}P_n(k)-2p_1(k/\sqrt{n})|\underset{n\to\infty}{\longrightarrow} 0\, .
\end{equation}
Let also $Q_k(n)=(k/n)P_n(k)$ be the probability that the simple random walk first hits $-k$ at time $n$. Then 
\begin{equation}\label{LLthQ}\sum_{n\in 2\N+k}|k^2Q_k(n)-2q_1(n/k^2)|\underset{k\to\infty}{\longrightarrow} 0\, .
\end{equation}
Then \eqref{Rquad} and Lagrange's inversion formula imply that for every $n,k\geq 0$,  
\begin{equation}\label{eqFHT}
[g^n]R^k=12^n 2^k Q_k(2n+k)\, .
\end{equation}

\begin{proof}[Proof of Theorem \ref{thcylinders}]
The number of annular quadrangulations $\QQ(n;L_1^n,L_2^n)$ with boundaries of lengths $2L_1^n,2L_2^n$ and with $n$ inner quadrangles is given by~\eqref{eq:ColletFusy2}: letting $\un{L}^n=L_1^n+L_2^n$ and $\un{\LL}=\LL_1+\LL_2$, this is 
\begin{align}\label{CFcylinders}
\#\QQ(n;L_1^n,L_2^n)=\frac{\alpha(2L_1^n)\alpha(2L_2^n)}{\un{L}^n}[g^n]R^{\un{L}^n}\underset{n\to\infty}{\sim} \frac{12^n\, 8^{\un{L}^n}}{n}\cdot\frac{\sqrt{\LL_1\LL_2}}{\pi\, \un{\LL}}q_{\LL}(1)\, ,
\end{align}
where the asymptotic formula is obtained from \eqref{eqFHT} and by applying \eqref{LLthQ} as well as the easy asymptotics $\alpha(2l)\sim 4^l\sqrt{l/\pi}$, also a consequence of \eqref{LLthP}.
Now, by the same discussion as that leading to~\eqref{eq:annulG},
the number of quadrangulations $M\in \QQ(n;L_1^n,L_2^n)$ that have $\ell(M)=l$, $A_0(M)=n_0$, $A_1(M)=n_1$, hence $A_2(M)=n_2:=n-n_0-n_1$, is given by
\begin{equation}
[g^{n_1}]\binom{2L_1^n}{L_1^n-l}R^{L_1^n-l}[g^{n_2}]\binom{2L_2^n}{L_2^n-l}R^{L_2^n-l}[g^{n_0}]2lR^{2l}.
\end{equation}
We rewrite this in ``probabilistic'' form as 
\begin{equation}
12^n\, 8^{\un{L}^n}2l P_{2L_1^n}(-2l)P_{2L_2^n}(-2l)Q_{L_1^n-l}(2n_1+L_1^n-l)Q_{L_2^n-l}(2n_2+L_2^n-l)Q_{2l}(2n_0+2l).
\end{equation}
Letting now $n_1=\lfloor n a\rfloor$ for some $a\in (0,1)$, $l=\lfloor (n/2)^{1/4}\lambda\rfloor$ and $n_0=\lfloor b\sqrt{2n}\rfloor$, for some $\lambda,b>0$, we can use the local limit theorems to get the following asymptotics: 
\begin{align}
2l P_{2L_1^n}(-2l)P_{2L_2^n}(-2l) &\underset{n\to\infty}{\sim} \frac{\sqrt{\LL_1\LL_2}}{\pi\un{\LL}} \Big(\frac{2}{n}\Big)^{1/4}r_{\LL_{\mathrm{eff}}}(\lambda)\, ,\\
Q_{L_1^n-l}(2n_1+L_1^n-l)Q_{L_2^n-l}(2n_2+L_2^n-l)& \underset{n\to\infty}{\sim}\frac{1}{n^2}q_{\LL_1}(a)q_{\LL_2}(1-a)\, ,\\
Q_{2l}(2n_0+2l) &  \underset{n\to\infty}{\sim}\frac{1}{\sqrt{2n}}q_\lambda(b)\, .
\end{align}
Taking a quotient with \eqref{CFcylinders}, this implies that for $l,n_1,n_0$ as above, 
\begin{multline}
  \P(\ell(M^n)=l,A_1(M^n)=n_1,A_0(M^n)=n_0)  
  \underset{n\to\infty}{\sim} \\
 \Big(\frac{2}{n}\Big)^{1/4}r_{\LL_{\mathrm{eff}}}(\lambda) \frac{1}{n}\frac{q_{\LL_1}(a)q_{\LL_2}(1-a)}{q_{\un{\LL}}(1)} \frac{1}{\sqrt{2n}}q_\lambda(b) \, . 
 \end{multline}
 Since the function of $\lambda,a,b$ appearing in the right hand side (after removing the factors involving $n$) is a probability density function on $\R_+\times (0,1)\times (0,\infty)$, we easily conclude by Scheff\'e's lemma. 
 \end{proof}

 \begin{proof}[Proof of Theorem \ref{thpants}]
The number of quadrangulations $\QQ(n;L_1^n,L_2^n,L_3^n)$ with three boundaries of perimeters $2L_1^n,2L_2^n,2L_3^n$ and with $n$ inner quadrangles is given by the ECF formula: letting 
$\un{L}^n=L_1^n+L_2^n+L_3^n$ and $\un{\LL}=\LL_1+\LL_2+\LL_3$, this is 
\begin{equation}
\label{CFpants}
\begin{split}
\#\QQ(n;L_1^n,L_2^n,L_3^n) &=\alpha(2L_1^n)\alpha(2L_2^n)\alpha(2L_3^n)[g^n]R^{\un{L}^n}\frac{d \ln R}{dt}\\
&\underset{n\to\infty}{\sim} 4^{\un{L}^n}\frac{\sqrt{L_1^nL_2^nL_3^n}}{\pi^{3/2}}[g^n]R^{\un{L}^n}\frac{d \ln R}{dt}
\end{split}
\end{equation}
The coefficient to extract in the right-hand side is given by a convolution of the form 
\begin{equation}
  \begin{split}
[g^n]R^{k}\frac{d \ln R}{dt} & = \sum_{m=0}^n[g^m]R^{k}[g^{n-m}]\frac{d \ln R}{d t}\\
& = 12^n\, 2^{k} \sum_{m=0}^nQ_{\un{L}^n}(2m+k)[g^{n-m}]\frac{1}{12^{n-m}}\frac{d \ln R}{dt}\, , \label{sumRk}
\end{split}
\end{equation}
where we have used again the probabilistic representation for the coefficients of $R^k$. For our present purposes we should take $k=\un{L}^n\sim \un{\LL}\sqrt{2n}$, but later we will also need the asymptotic behaviour of the same quantity, where $k$ is of smaller order $n^{1/4}$. So let us start with this simpler case, assuming that $k=k(n)$ is bounded by $Kn^{1/4}$ for some $K>0$.

Note that, by the asymptotics \eqref{asympenum}, the coefficients $c_m=[g^{m}]12^{-m}\frac{d \ln R}{dt}$ involving the logarithmic derivative of $R$ are uniformly bounded by some constant $C$, and equivalent to $1/2\sqrt{\pi m}$ as $m\to\infty$. So if we fix $\beta\in (1/2,1)$, we can rewrite using \eqref{asympenum} the sum arising in \eqref{sumRk} as
\begin{equation}
  \sum_{m=0}^{n^{\beta}}Q_{k}(2m+k)\frac{(1+\epsilon_n)}{2\sqrt{\pi n}}+r_n
\end{equation}
where $\epsilon_n$ is a sequence depending only on $n$ and converging to $0$, and $r_n\geq 0$ is a remainder term which is bounded by $C\sum_{m>n^{\beta}}Q_{k}(2m+k)$. Our choice of $\beta$ and the fact that $k\leq Kn^{1/4}$ then implies that $r_n\to 0$. So in this case, 
\begin{equation}\label{cassimple}
[g^n]R^{k}\frac{d \ln R}{dt} \sim \frac{12^n\, 2^{k} }{2\sqrt{\pi n}}\, .
\end{equation}
Now, for the slightly more delicate case where $k=\un{L}^n$, we rewrite the sum in \eqref{sumRk} as
\begin{equation*}
\frac{1}{(\un{L}^n)^2}\left[\sum_{m=0}^n c_{n-m}\bigg((\un{L}^n)^2Q_{\un{L}^n}(2m+\un{L}^n)-2q_1\Big(\frac{2m+\un{L}^n}{(\un{L}^n)^2}\Big)\bigg)+ 2\sum_{m=0}^n c_{n-m}q_1\Big(\frac{2m+\un{L}^n}{(\un{L}^n)^2}\Big)\right]
\end{equation*}
and note that by boundedness of $(c_m)$ and the local limit theorem \eqref{LLthQ}, the first sum converges to $0$ in absolute value. It remains to deal with the second sum. We introduce some $\beta\in (1/2,1)$ and split the sum according to whether $m\leq n-n^\beta$ or $n-n^\beta<m\leq n$. In the first case we can use the asymptotics \eqref{asympenum} and a comparison with an integral to obtain
\begin{equation}
\sum_{m=0}^{n-n^{\beta}} c_{n-m}q_1\Big(\frac{2m+\un{L}^n}{(\un{L}^n)^2}\Big) \sim \sqrt{n}\int_0^1\frac{q_1(a/\un{\LL}^2)}{2\sqrt{\pi(1-a)}}da\, ,
\end{equation}
while the sum over $n-n^\beta<m\leq n$ is clearly $O(n^{-\beta})$ by bounding the coefficients $c_{n-m}$ by $C$ and using the fact that $q_1$ is bounded near $1/\un{\LL}^2$. Putting things together, we obtain 
\begin{align}
\label{CFpants2}
\#\QQ(n;L_1^n,L_2^n,L_3^n) 
&\underset{n\to\infty}{\sim} (8n)^{1/4}\frac{12^n\, 8^{\un{L}^n}}{\pi^2}\sqrt{\LL_1\LL_2\LL_3}\int_0^1\frac{q_{\un{\LL}}(t)}{2\sqrt{1-t}}dt\, .
\end{align}

By Propositions \ref{prop:minimalbij} and \ref{prop:enumannular} and Theorem \ref{thm:Tabc}, the number of quadrangulations $M\in \QQ(n;L_1^n,L_2^n,L_3^n)$ such that $\ell_i(M)=l_i$ and  $A_i(M)=n_i$ for $i\in \{1,2,3\}$ is equal to (letting $n_0=n-n_1-n_2-n_3$)
\begin{multline}
\prod_{i=1}^3[g^{n_i}](2l_i)\binom{2L_i^n}{L_i^n-l_i}R^{L_i^n-l_i}\times [g^{n_0}]R^{\un{l}}\frac{d \ln R}{dt} = \\
12^n\, 8^{\un{L}^n}\prod_{i=1}^3(2l_i)P_{2L_i^n}(-2l_i)Q_{L_i^n-l_i}(2n_i+L_i^n-l_i)\times \frac{1}{12^{n_0}2^{\un{l}}}[g^{n_0}]R^{\un{l}}\frac{d \ln R}{dt}
\end{multline}
where $\un{l}=l_1+l_2+l_3$ and where we have used once again the probabilistic representation of the coefficients. 
Here the first two extracted coefficients count the number of annular quadrangulations with a rooted boundary of perimeter $2L_i^n$, and a strictly tight boundary of length $2l_i$, and the last one counts the number of (quadrangulated) pairs of pants with tight boundaries of perimeters $2l_i,i\in \{1,2,3\}$. Proposition \ref{prop:minimalbij} states that the boundaries of these pairs of pants should be marked, and we have absorbed the corresponding factors $2l_i$ in the product before. 
We let $l_i=\lfloor(n/2)^{1/4}\lambda_i\rfloor$ and $n_i=\lfloor na_i\rfloor$ for some $\lambda_i>0$ and $a_i>0$, $i\in \{1,2,3\}$,  such that $a_0=1-(a_1+a_2+a_3)>0$. Then, the local limit theorems give the asymptotics
\begin{align}
2l_i P_{2L_i^n}(-2l_i) &\underset{n\to\infty}{\longrightarrow} \sqrt{\frac{2\LL_i}{\pi}} r_{\LL_i}(\lambda_i)\, ,\\
Q_{L_i^n-l_i}(2n_i+L_i^n-l_i)& \underset{n\to\infty}{\sim}\frac{1}{n}q_{\LL_i}(a_i)\, .
\end{align}
Together with \eqref{cassimple}, this implies by taking a quotient with \eqref{CFcylinders} that for $l_i,n_i$ as above, 
\begin{multline}
\P(\ell_i(M^n)=l_i,A_i(M^n)=n_i,i\in \{1,2,3\}) \underset{n\to\infty}{\sim}  \\
\prod_{i=1}^3 \Big(\frac{2}{n}\Big)^{1/4}r_{\LL_i}(\lambda_i) \times \frac{1}{n^3}\Big(\int_0^1\frac{q_{\un{\LL}}(a)}{2\sqrt{1-a}}d a\Big)^{-1}\frac{q_{\LL_1}(a_1)q_{\LL_2}(a_2)q_{\LL_3}(a_3)}{2\sqrt{1-a_1-a_2-a_3}}  \, . 
 \end{multline}
 Since the function of $\lambda_i,a_i$ appearing in the right hand side (after removing the factors involving $n$) is a probability density function on $(\R_+)^3 \times \{(a_1,a_2,a_3)\in (0,\infty)^3:a_1+a_2+a_3< 1\}$, we conclude by Scheff\'e's lemma. 
 \end{proof}

\section{Conclusion}\label{sec:conc}

In this work, we have enumerated bijectively essentially bipartite
planar maps with three tight boundaries, relying on a geometric
decomposition of these objects in terms of elementary pieces with
certain geodesic boundaries.

Let us mention a number of natural extensions of the present work that
we plan on studying in the future. The most natural extension consists
in considering maps with more than three boundaries and/or higher
genus. We first remark that the discussion of
Section~\ref{sec:struct-outerm-minim}, hence
Proposition~\ref{prop:minimalbij}, extend easily to arbitrary
topologies: a map $M$ of genus $g$ having $n$ boundary-faces which are
not necessarily tight can be decomposed bijectively into a tuple
$(M^{(0)},M^{(1)},\ldots,M^{(n)})$, where $M^{(0)}$ is a map of genus
$g$ with $n$ tight boundary-faces, and where $M^{(1)},\ldots,M^{(n)}$
are ``funnels'', namely annular maps with one strictly tight
boundary-face, as defined in Section \ref{sec:rederivation}. This
decomposition is closely related to the Joukowsky transform considered
in \cite[Section~3.1.3.1]{Eynard16}. Using enumerative results coming
from topological recursion, we can show that the generating function
of essentially bipartite maps of genus $g$ with $n$ tight
boundary-faces of prescribed lengths $2\ell_1,\ldots,2\ell_n$ is a
quasi-polynomial generalizing the lattice count polynomial of
\cite{Norbury2010} (which we recover when setting the weights for
inner faces to zero). This will be discussed in a forthcoming paper.

Furthermore, it would be interesting to address the problem of the
enumeration of maps of genus $g$ with $n$ tight boundaries by a
bijective approach.  One might think at first that our decomposition
into bigeodesic triangles and diangles could easily be extended
without fundamental changes beyond the case $(g,n)=(0,3)$ considered
in the present paper. A closer look however shows that a number of new
technical questions arise in the general case: for instance,
controlling the exceedances of the diangles is not as simple as for
pairs of pants where these exceedances are entirely fixed by the
boundary lengths. More important, making sure that the bigeodesics
used in the decomposition do not cross, and therefore lead to
independent building blocks, is more challenging for more boundaries
or higher genus.

Still, we hope that the tight pairs of pants introduced in
this work, 
or small variations thereof, will serve as new elementary pieces intervening in the
decomposition of such maps with higher topological complexity. Indeed, pants
decompositions are the canonical way to describe all Riemann surfaces,
by cutting them along separating cycles. In the context of maps, in
order to get a canonical decomposition, 
one needs to specify along which cycles we cut. In this
respect, minimal separating cycles are the natural candidates,
but we must specify which of these minimal cycles we choose among certain ordered sets $\mathcal{C}$ of such cycles. In Section \ref{sec:struct-outerm-minim}, we explained why the choice 
of ``outermost'' 
elements in the ordered set $\mathcal{C}$ was crucial to avoid possible crossings of the 
various cutting cycles. If we now wish to split a map into two components, the ordering of the set $\mathcal{C}$ of separating cycles is reversed 
when viewed from both components and choosing its outermost
element from both sides therefore produces some overlap between the
components, hence a decomposition into non independent elements. To
avoid such overlap, we must instead choose innermost elements from
both sides (so that the overlapping region now becomes an independent
building block), or at least from one side. We then face again the
problem that, if we choose only innermost elements, different cycles
around a pair of pants may cross each other. A probable solution is to
consider a mixed prescription, with both outermost and innermost
elements, which would involve as building blocks pairs of pants with,
say, one strictly tight boundary and two tight ones. The question of
their enumeration is therefore an issue that we hope to better
understand in the future.

As we noticed above, Theorems \ref{thcylinders} and \ref{thpants}
imply that minimal separating cycles in random maps with the topology of the annulus or of the pair of pants admit Rayleigh statistics in the scaling
limit, with a parameter that depends in a similar and simple way on
the boundary lengths. One can naturally wonder whether these
statistics also arise for more boundaries or in higher
genera. However, for four boundaries or more, the minimal separating
cycles are not necessarily separating only one boundary from all the
others, and it is likely that more complicated statistics would
arise.

Another direction of study would be to control distances between the
boundaries. This program was achieved in the case of three  
boundary-vertices in \cite{threepoint} for planar quadrangulations and
in \cite{FuGu14} for general  
planar maps. As discussed in Appendix~\ref{sec:wlm}, the results
of~\cite{threepoint} provide an explicit expression for the
generating function $X_{s,t}$ (respectively $Y_{s,t,u}$) of balanced
bigeodesic diangles  
(respectively bigeodesic triangles) with all inner faces of degree $4$ and
with, say red intervals of lengths
$s'\leq s$ and $t'\leq t$ (respectively of lengths $s'\leq s$, $t'\leq
t$ and $u'\leq u$).  
Together with the generating function  
$R_s$ for elementary slices with (red) right boundary of length $s'<s-1$
($s\geq 1$), known since the very introduction of slices in
\cite{hankel}, it seems 
that we have all the ingredients (at least for quadrangulations) for a
proper refined enumeration of pairs of pants with 
a control on the (properly defined) geodesic distances between their
boundary-faces or boundary-vertices. Indeed,  
the lengths $s,t,u$ above, characterizing the various building blocks,
eventually fix the desired distances. 

A final framework where our method is likely to apply is that of
planar irreducible maps or maps with girth constraints, for which an
interesting connection with Weil-Petersson volumes was recently
pointed out by Budd~\cite{Budd2022b,Budd2022a}. Recall that the
\emph{girth} is the length of the shortest cycle in the map and that a
map is \emph{$d$-irreducible} if its girth is at least $d$ and any
cycle of length $d$ is the contour of an inner face. In
\cite{irredmaps}, a slice decomposition was devised to enumerate such
families of maps with one or two boundaries. We expect that this
decomposition may be extended to three boundaries along lines similar
to those of the present paper.

\appendix

\section{A slice-theoretic enumeration of triply pointed maps}
\label{sec:slicetp}

\paragraph{Recursion relation for $\boldsymbol{R}$.}
Call $R$ the generating function of elementary slices (i.e.\ tight slices of width $1$). Let us show that $R$ 
satisfies the recursion relation \eqref{eq:Req}, which determines it uniquely as a formal power series in $t$ 
and in the $g_{2k}$'s. We use the notations of Figure~\ref{fig:tightslice} 
for tight slices, specialized to the case where the interval $[c',c'']$ reduces to a single oriented ``base edge'' $e$. 
Assuming that the slice is not reduced to $e$, we may consider its base face $f$ 
incident to $e$ on its left, and look at the (clockwise) contour path of $f$ from $v'$ 
(incident to $c'$) to $v''$ (incident to $c''$): this path has length $2k-1$ if $f$ has degree $2k$.
Calling $v$ the apex of the slice (vertex incident to $c$), we may record the relative distances 
$\ell_i=d(v',v)-d(v_i,v)$ for the successively visited vertices $v_i$, $i=0,\ldots,2k-1$, along the contour path from 
$v_0=v'$ to $v_{2k-1}=v''$. The sequence $\left(\left(i,\ell_i\right)\right)_{0\leq i\leq 2k-1}$ defines a directed path 
$\mathcal{P}$ of length $2k-1$ in 
$\Z^2$ from $(0,\ell_0)=(0,0)$ to $(2k-1,\ell_{2k-1})=(2k-1,1)$, with $k$ ascending steps 
with $\ell_{i}-\ell_{i-1}=+1$ and $k-1$ descending steps with $\ell_{i}-\ell_{i-1}=-1$. Let us now, for each visited vertex $v_i$, 
cut the slice along the leftmost geodesic $\mathcal{G}_i$ from $v_i$ to $v$. This results into a decomposition of the slice into $k$ 
components which are elementary slices, in correspondence with the $k$ ascending steps of $\mathcal{P}$. 
More precisely, to each step 
with $\ell_{i}-\ell_{i-1}=+1$ is associated an elementary slice delimited by $\mathcal{G}_{i-1}$ and $\mathcal{G}_i$, 
whose base edge connects $v_{i-1}$  (at distance $d(v,v')-\ell_{i-1}$ from $v$) to $v_i$ (at distance $d(v,v')-\ell_{i-1}-1$ from $v$) 
and whose apex is the first meeting point of $\mathcal{G}_{i-1}$ and $\mathcal{G}_i$ towards $v$.
  As for a step with $\ell_{i}-\ell_{i-1}=-1$, it does not give rise to any component in the decomposition since
$\mathcal{G}_i$ starts by following (counterclockwise) the contour 
of $f$ from $v_i$  (at distance $d(v,v')-\ell_{i}$ from $v$) to $v_{i-1}$ (at distance $d(v,v')-\ell_{i}-1$ from $v$)
and then merges with $\mathcal{G}_{i-1}$, so that no faces lie in-between $\mathcal{G}_{i-1}$ and $\mathcal{G}_i$.

Starting conversely from the directed path $\mathcal{P}$ above, viewed as a sequence of edges all colored
in red, and from the $k$ elementary slice components $S_1,\ldots, S_k$, we may recover the original slice by: (i) 
gluing the (blue) base edge of $S_j$ to (and above) the (red) edge
associated with the
$j$-th ascending step 
of $\mathcal{P}$, then (ii) gluing each blue boundary edge 
of a slice $S_j$ to the the first available red edge, if any, facing it on its left (this edge may belong to the red boundary of a
preceding slice component or be associated with a descending step of $\mathcal{P}$) and finally (iii) closing $\mathcal{P}$ by adding 
an extra base edge $e$ so as to form the base face $f$ of degree $2k$. 

Once translated in the language of generating functions, the above bijective decomposition
yields the relation \eqref{eq:Req}, where the first term $t$ accounts for the elementary slice reduced to a single edge
and the $k$-th term in the sum accounts for elementary slices with a base face of degree $2k$, with
$g_{2k}$ the weight of this face, the factor ${2k-1\choose k}$ the number of
possible oriented paths $\mathcal{P}$ of length $2k-1$ with $k$ ascending steps, 
and the factor $R^k$ the generating function for the $k$ elementary slice components.

\paragraph{Proof of the relation $\boldsymbol{T_{0,0,0}=d \ln(R/t)/dt}$.}
Take an elementary slice not reduced to a single edge. Upon gluing its two intervals $[c,c']$ and $[c',c]$ (see again Figure~\ref{fig:tightslice} for the notations), we get a bipartite planar map with both a marked oriented edge $e$ 
(corresponding to the interval $[c',c'']$ of length $1$ oriented from $c'$ to $c''$) and a marked vertex $v$ 
(corresponding to the vertex incident to $c$) which is closer to the endpoint than to the origin of $e$. 
Conversely, starting from a bipartite planar map with a marked oriented edge $e$ and a marked vertex $v$ closer to the endpoint than to the origin of $e$, the elementary slice leading to this marked map by the above gluing is easily recovered by cutting the map along the leftmost geodesic towards $v$ starting with $e$. In the generating function $R-t$ of elementary slices not reduced to a single edge, 
the vertex incident to $c'$ receives a weight $t$ while that incident to $c$ receives no weight. 
We immediately deduce that $(R-t)/t$ is the generating function of planar bipartite maps with a marked oriented edge $e$ and 
a marked vertex $v$ closer to the endpoint than to the origin of $e$, where neither $v$ nor the origin $v'$ of $e$ (necessarily distinct from $v$) receive the weight $t$.
Alternatively, by first choosing $v'$ then $e$, $(R-t)/t$ is the generating function of planar bipartite maps with two distinct marked vertices $v$ and $v'$ (which receive no weight $t$)
and a marked edge $e$ incident to $v'$ and whose other extremity is at distance $d(v',v)-1$ from $v$. If we now wish to compute instead the generating
function $B$ of doubly pointed maps, without the marked edge $e$, we may proceed as follows: consider, for a map with two marked distinct vertices $v$ and $v'$
the (non-empty) counterclockwise sequence of edges from $v'$ to a vertex at distance $d(v',v)-1$ from $v$, and cut the map along the leftmost geodesics
towards $v$ starting with these successive edges. This results in a non-empty cyclic sequence of a particular type of elementary slices,
all not reduced to a single edge, which 
are such that all the non-boundary edges incident to the vertex $v'$ incident to the corner $c'$ lead to vertices at a distance larger than 
$d(v',v)$ from $v$, the vertex incident to $c$. Call $N$ the generating function of these particular elementary slices (with the same weighting convention
as for regular tight slices). We deduce 
the relation $B=-\ln(1-N/t)$ (note that in $B$, maps are counted with symmetry factors: a planar map with
two marked distinct vertices may have a $k$-fold symmetry by ``rotating'' around the axis of the two marked vertices. It then receives the weight $1/k$).
As for the maps counted by $(R-t)/t$, the additional marked edge $e$ provides an origin for the cyclic sequence so that the above
cutting now results in a non-empty linear sequence of the same particular tight slices. We now deduce the relation
$(R-t)/t=N/t/(1-N/t)$, from which we eventually get $B=\ln(R/t)$.
Since in $B$ the two marked vertices have no weight, taking a derivative with respect to $t$ in $B$ amounts to the marking of a third vertex on the map, 
distinct from the already marked ones. We deduce that $T_{0,0,0}=dB/dt=d \ln(R/t)/dt=d\ln R/dt-t^{-1}$.

\section{Connection with well-labeled maps}\label{sec:wlm}
We discuss here the connection between our decomposition into bigeodesic diangles and triangles 
and another decomposition introduced in \cite{threepoint} to characterize the three-point function of planar maps.
We restrict our discussion to the case of quadrangulations, i.e.\ maps whose all inner faces have degree $4$,
and to the case where the three boundaries are boundary-vertices.
As first noted in \cite{Miermont09}, such triply pointed planar quadrangulations may be bijectively encoded by so-called
planar \emph{well-labeled maps}, which are maps whose vertices carry integer labels with the constraint that
\begin{itemize}
\item{the difference of labels between any two neighboring vertices is $0$ or $\pm 1$.}
\end{itemize}
For convenience, the corners of a well-labeled map receive the label of their incident vertex.

\medskip
More precisely, as shown in \cite{Miermont09}, and in \cite{threepoint} in the specific case that we consider here, one may establish a one-to-one correspondence between
planar quadrangulations with three distinct vertices $v_A$, $v_B$
and $v_C$ and planar well-labeled maps with (generically) three faces $f_A$, $f_B$ and $f_C$
satisfying the additional constraint that 
\begin{itemize}
\item{C1. the frontier between any two faces of the map (i.e. the set of vertices and edges incident to both faces) is non-empty and the minimum label
on this frontier is $0$.}
\end{itemize}
In the non-generic case where one of the three boundary-vertices is a geodesic vertex between the other two, one of the faces in the 
well-labeled map degenerates into a single vertex, and some of the arguments presented below must be adapted.

\begin{figure}
  \centering
  \includegraphics[width=1.\textwidth]{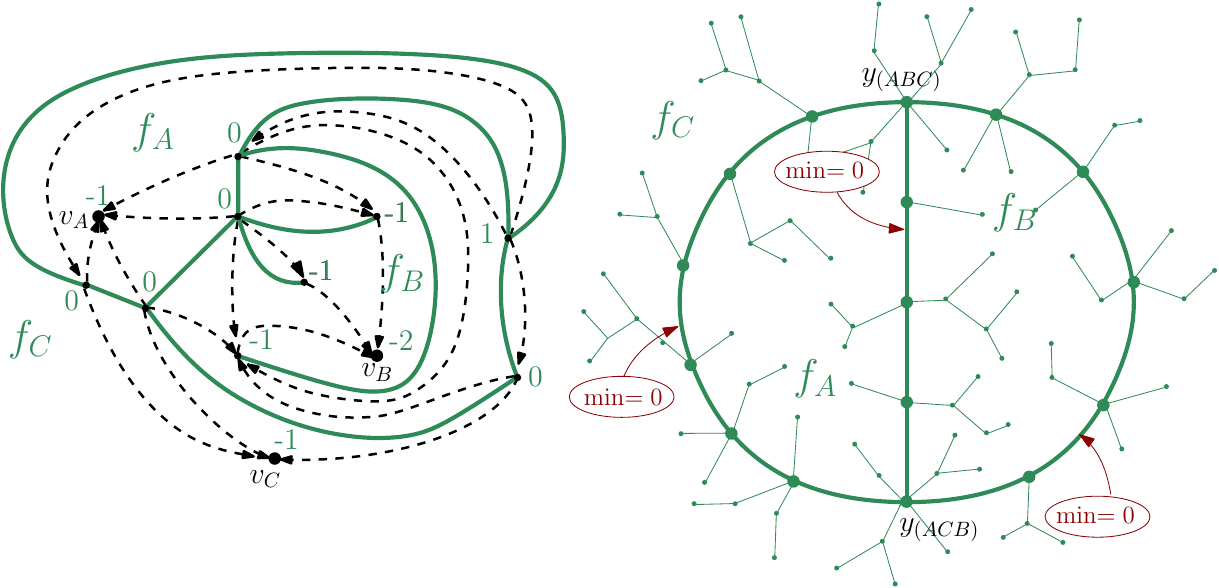}
  \caption{Left: The construction of a quadrangulation with three boundary-vertices $v_A$, $v_B$ and $v_C$ from a well-labeled
  map with three faces $f_A$, $f_B$ and $f_C$, as obtained by connecting each corner to its successor. Here $r_B=2$ and $r_A=r_C=1$.
  Right: Schematic picture of a $3$-face well-labeled map satisfying C1. Its skeleton is indicated by thick edges. Each frontier between two given faces
  carries minimal label $0$.}
  \label{fig:WL}
\end{figure}

\medskip
Given a planar well-labeled map with three faces satisfying C1, the associated triply pointed quadrangulation is easily recovered as follows: 
calling $1-r_i$ the minimum label among vertices incident to the face $f_i$ ($i\in\{A,B,C\}$), 
we add a new vertex $v_i$ with label $-r_i$ in this face. Within each face, we then connect each corner with label $\ell$ to its \emph{successor}, 
which is the first encountered corner with label $\ell-1$ when going counterclockwise around
the face (i.e. with the face on the left). See the left of Figure~\ref{fig:WL} for an example. We finally
remove the labels as well as the original edges of the well-labeled map. In particular, the vertices of the
quadrangulation are identified with those of the well-labeled map, plus the three added vertices $v_A$, $v_B$ and $v_C$. 

An important property relating the well-labeled map to its associated quadrangulation is the following: 
\begin{itemize}
\item{Any vertex $v$ incident to $f_i$ ($i\in\{A,B,C\}$ ) with label $\ell(v)$ is at a distance $r_i+ \ell(v)$ from $v_i$ in the
quadrangulation (the property extends trivially to $v_i$ itself since $\ell(v_i)=-r_i$).}
\end{itemize}
From this property, it is then easily shown that those vertices of the frontier between $f_i$ and $f_j$ which carry the (minimal) label $0$ 
are precisely the geodesic vertices between $v_i$ and $v_j$ (for $i\neq j\in\{A,B,C\}$) at distance $r_i$ from $v_i$ and $r_j$ from $v_j$.

\medskip
The generic topology of a planar well-labeled map with three faces satisfying C1 is shown on the right of Figure~\ref{fig:WL}. 
Its \emph{skeleton}, obtained by iteratively removing all the leaves of the map so 
that all remaining vertices have degree at most $2$, has exactly two $3$-valent vertices $y_{(ABC)}$ and $y_{(BAC)}$ and three linear branches
between them made of $2$-valent vertices, each branch corresponding to a frontier between two faces. 
From C1, each of the three branches carries a minimal label $0$. The full well-labeled map is made of this skeleton and a number
of attached well-labeled subtrees.

\medskip
We may now perform a canonical decomposition of the map as in~\cite[Section~4.3]{threepoint}. Namely, consider the branch at the frontier between, say $f_A$ and $f_B$ and call 
$v_{AB}$ (respectively $v_{BA}$) the vertex with label $0$ closest to $y_{(ABC)}$ (respectively closest to $y_{(ACB)}$).
We define $v_{BC}$, $v_{CB}$ and $v_{CA}$, $v_{AC}$ by cyclic permutation. 
Clearly, from the above discussion, $v_{AB}$ and $v_{BA}$ are the extremal elements of $S_{AB}$ in the sense of Figure \ref{fig:threepointdecomp}.
We may then cut the map at all the $v_{ij}$ vertices, resulting in five well-labeled tree components: the first two components, containing one of the vertices $y_{(ABC)}$ or $y_{(BAC)}$, will be referred to as \emph{Y-diagrams} and called $Y_{(ABC)}$ and $Y_{(BAC)}$ accordingly.
\begin{figure}[t!]
  \centering
  \includegraphics[width=.9\textwidth]{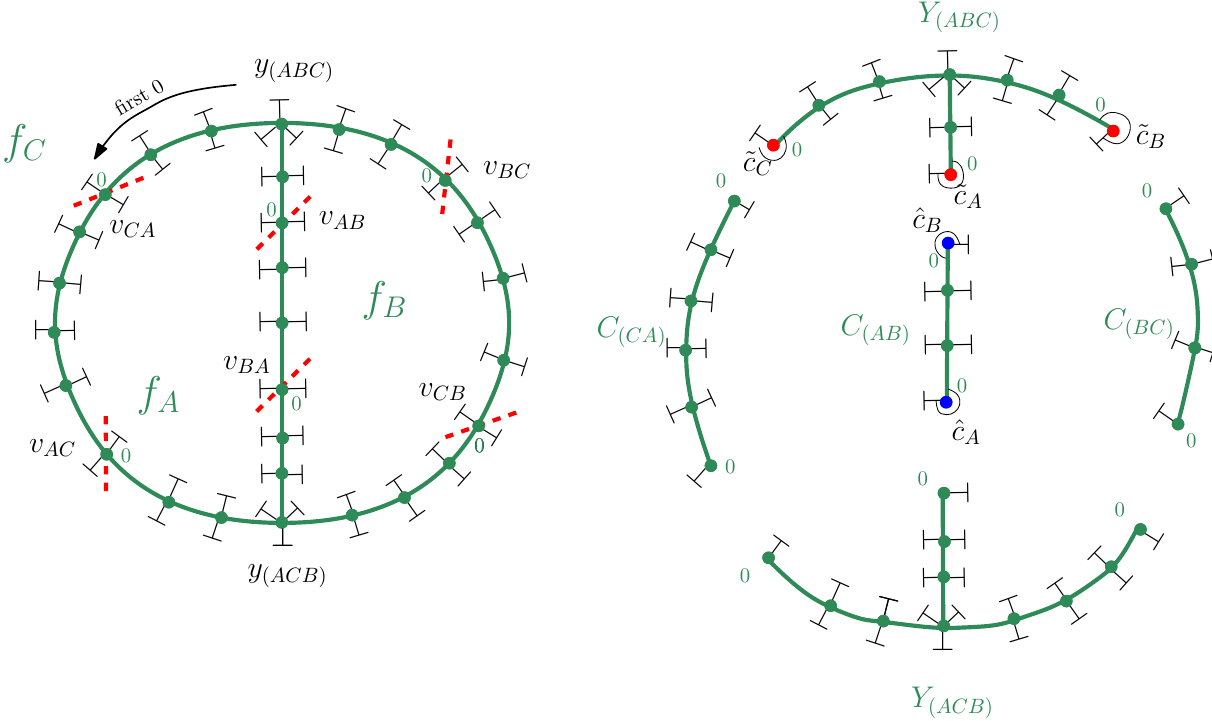}
  \caption{Schematic picture of the decomposition of a well-labeled map with three faces satisfying C1 into five well-labeled components, two Y-diagrams and three chains, 
  by cutting each branch of the skeleton at their extremal vertices labeled $0$ as shown by dashed red lines. The $\perp$ signs indicate
  attached well-labeled subtrees. The cutting singles out two corners $\hat{c}_A$ and $\hat{c}_B$ on the chain $C_{(AB)}$. 
  Similarly, three corners $\tilde{c}_A,\tilde{c}_B,\tilde{c}_C$ are singled out on the Y-diagram $Y_{(ABC)}$.}
  \label{fig:WLdecomp}
 \end{figure}
The last three components, lying in-between $v_{ij}$ and $v_{ji}$ 
for some $i\neq j\in \{A,B,C\}$, will be called \emph{chains} and denoted $C_{(ij)}$. In the cutting process, we must specify to which component
we attach the subtrees incident to the cutting vertices $v_{ij}$. We use the convention shown on the left of Figure~\ref{fig:WLdecomp}. 
This is dictated by the fact that we wish to retain in, say the chain $C_{(AB)}$ the subtree incident to $v_{BA}$ which follows clockwise 
the leftmost corner incident to $v_{BA}$ in the face $f_A$ of the original well-labeled map, as this subtree may carry the successor of this corner.
This choice of cutting singles out de facto two corners at the extremities of the chain $C_{(AB)}$:  
a corner $\hat{c}_A$ at $v_{BA}$ and a corner $\hat{c}_B$ at $v_{AB}$, see the right of Figure~\ref{fig:WLdecomp}. Note that we may identify
those vertices of the chain originally in $f_A$ (respectively $f_B$) as the vertices lying 
in-between $\hat{c}_A$ and $\hat{c}_B$ (respectively in-between $\hat{c}_B$ and $\hat{c}_A$) when going clockwise around the chain. Similarly, the
cutting process marks three corners $\tilde{c}_A,\tilde{c}_B,\tilde{c}_C$ preceding the retained subtrees at the end of the branches of Y-diagram 
$Y_{(ABC)}$, see the right of Figure~\ref{fig:WLdecomp}. By
construction, the well-labeled chains have nonnegative labels on the
unique path linking their extremal corners, the latter having label zero,
while for Y-diagrams, all labels along the paths linking the three extremal corners are
strictly positive, except for the corners themselves, which have label
zero. 
 \begin{figure}
  \centering
  \includegraphics[width=.5\textwidth]{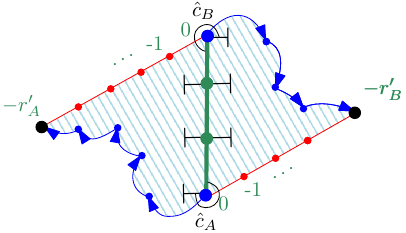}
  \caption{Schematic picture of the balanced bigeodesic diangle encoded by the well-labeled chain $C_{(AB)}$ (see text).
  Its attachment points are the big blue dots.}
  \label{fig:Chaintodiang}
 \end{figure}
\medskip
 
We now claim that the two Y-diagrams and the three chains resulting from the decomposition of the well-labeled map
precisely encode the two bigeodesic triangles and the three bigeodesic diangles resulting from our decomposition
of the associated triply pointed quadrangulation. To see this, 
consider for instance the well-labeled chain $C_{(AB)}$. We may associate to this chain a balanced bigeodesic diangle with attachment points $v_{AB}$
and $v_{BA}$ as follows: calling $1-r'_A$ (respectively $1-r'_B$) the minimal label of the chain between $\hat{c}_A$ and $\hat{c}_B$ (respectively
between $\hat{c}_B$ and $\hat{c}_A$) when going clockwise around the chain (with $r'_A\leq r_A$ and $r'_B\leq r_B$), we attach to 
$\hat{c}_B$ a new branch, called the $A$-branch, made of $r'_A$ vertices with labels decreasing from $-1$ to $-r'_A$ 
and to $\hat{c}_A$ a new branch, the $B$-branch, made of $r'_B$ vertices with labels decreasing from $-1$ to $-r'_B$. These branches are represented as
red lines in Figure~\ref{fig:Chaintodiang}.  
Each corner of the chain, except those incident to the new vertices of the two added red branches, is then connected to its successor, possibly lying on the newly added red branch
(note that going counterclockwise around the external face corresponds to going clockwise around the chain). 
The resulting object is a map with one boundary-face and four boundary intervals, alternating 
between blue (geodesic) intervals, corresponding to the sequence of successors of $\hat{c}_A$ and that of $\hat{c}_B$, and red (strictly geodesic) intervals, 
corresponding to the $A$- and the $B$-branch. This is nothing but a bigeodesic diangle with attachment points $v_{AB}$ and $v_{BA}$,
which is moreover balanced, with a blue and a red interval of the same length $r'_A$ and the other two of the same length $r'_B$. 
We can repeat the process
to build balanced bigeodesic diangles from the chains $C_{(BC)}$ and $C_{(CA)}$. As for the well-labeled Y-diagrams, say for instance $Y_{(ABC)}$: 
calling $1-r''_A$
the minimal label between $\tilde{c}_A$ and $\tilde{c}_C$, 
we attach a new red $A$-branch
of length $r''_A$
to the corner $\tilde{c}_C$. We do the same by cyclic permutations of
the letters $A,B,C$. This gives rise to the 
red lines in Figure~\ref{fig:Ydiagtotriang}. Finally, we  
connect as before each corner not along the added red branches to its successor
(possibly lying on the newly added red branches). 
This clearly creates a bigeodesic triangle: the fact that the attachment points of this triangles are 
``red'' points in our terminology is due to the fact that we chose the extremal $0$ labels on each frontier so that any geodesic between $v_i$ and $v_j$ ($i\neq j\in\{A,B,C\}$) within the triangle, which must cross a vertex with label $0$, has to pass via the appropriate attachment vertex. 
The other triangle is obtained similarly from $Y_{(ACB)}$. 
 \begin{figure}
  \centering
  \includegraphics[width=.5\textwidth]{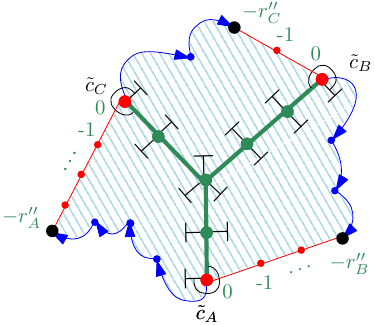}
  \caption{Schematic picture of the bigeodesic triangle encoded by the well-labeled Y-diagram $Y_{(ABC)}$ (see text).
  Its attachment points are the big red dots.}
  \label{fig:Ydiagtotriang}
 \end{figure}

\medskip
We now claim that gluing the three bigeodesic diangles and two bigeodesic triangles associated to the three well-labeled chains and the two
well-labeled Y-diagrams according to our procedure I clearly reproduces the quadrangulation
associated with the well-labeled map at hand before its decomposition. This simply results from the fact that the sequences of successors within the full well-labeled map
after gluing match precisely with the sequences of successors within each of its five well-labeled components after identification of the glued blue and red intervals, see Figure~\ref{fig:WLdecompbis}. The paths along which the bigeodesic diangles and triangles are glued induce paths in the resulting quadrangulation, that correspond precisely to leftmost geodesics launched from 
the cutting points $v_{ij}$ ($i\neq j \in \{A,B,C\}$) towards the vertices $v_i$ at the extremity of the added red $i$-branches of the various well-labeled components which have the smallest label, necessarily equal to $-r_i$.   

\medskip
To conclude, we have a bijective correspondence between (i) balanced bigeodesic diangles and well-labeled chains, and (ii) bigeodesic triangles
and well-labeled Y-diagrams. With this correspondence, our decomposition of planar quadrangulations with three boundary-vertices
matches precisely that of \cite{threepoint} for the associated well-labeled maps with three faces.

\begin{figure}
  \centering
  \includegraphics[width=.9\textwidth]{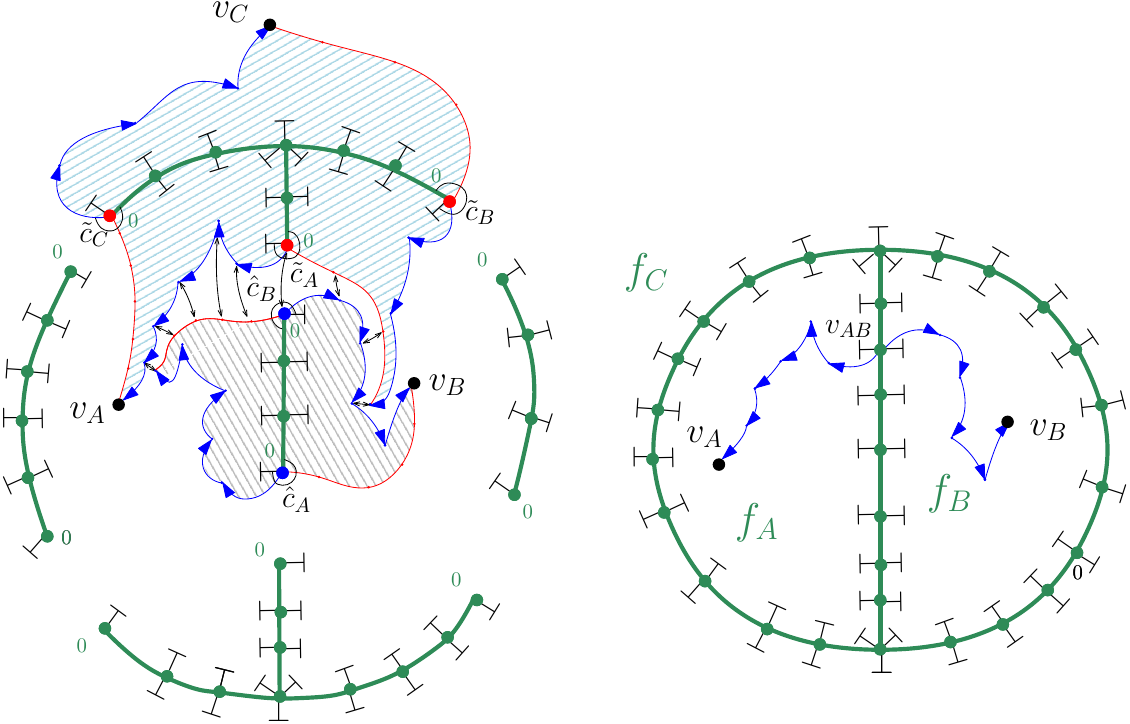}
  \caption{The gluing of the bigeodesic triangle encoded by the Y-diagram $Y_{ABC}$ and the bigeodesic diangle encoded by 
  the chain $C_{(AB)}$ occurs along leftmost geodesics launched from $v_{AB}$ in the original quadrangulation and correspond to sequences of successors
  both in the well-labeled components (left) and in the well-labeled map (right).}
  \label{fig:WLdecompbis}
 \end{figure}
 \begin{figure}
  \centering
  \includegraphics[width=.3\textwidth]{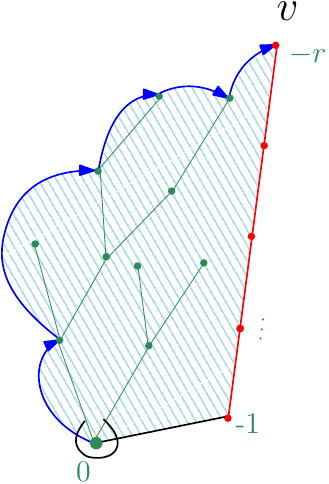}
  \caption{The elementary slice encoded by a well-labeled planted tree with root label $0$ (see text).}
  \label{fig:Rslice}
 \end{figure}
As a direct enumerative consequence, we identify $X$ and $Y$ as the generating functions of properly weighted 
well-labeled chains and Y-diagrams. More precisely, if we let $t$ and $g=g_4$ be the inner vertex and face weights in the 
triply pointed quadrangulation, each vertex of a well-labeled chain or Y-diagram receives the weight $t$ and each edge
the weight $g$. To include the possible degenerate cases (for instance the case of well-labeled maps with
two faces obtained whenever one of the boundary-vertices is geodesic between the other two),  the vertex map, with label $0$, must be considered as a well-labeled Y-diagram as well as 
a well-labeled chain.  
Viewed as well-labeled object generating functions, $X$ and $Y$ are easily obtained as the power series in $t$ solutions 
(see \cite{threepoint} for a detailed derivation) of
\begin{equation}
R=t+3 g\, R^2\ , \quad X=t+\frac{1}{t}\, g\, R^2\, X\left(1+\frac{1}{t^2}\, g\, R^2\, X\right)\, \quad Y=t+\frac{1}{t^6}\, g^3\, R^6\, X^3\, Y\ ,
\end{equation}
where $R$ is the generating function of well-labeled planted trees with root label $0$. Note that $R$ matches our definition \eqref{eq:Req} for $g_{2k}=g\, \delta_{k,2}$,
i.e.\ is also the generating function of elementary slices with $4$-valent inner faces only. 
That well-labeled planted trees encode elementary slices is obtained along the same lines as before: calling $1-r$ the smallest label in the tree,
we attach to the root-corner a branch of length $r$ with vertices having decreasing labels $-1,\cdots,-r$ as in Figure~\ref{fig:Rslice}. Connecting each corner not incident
to one of the new added vertices to its successor
creates a map with a single boundary face, with $4$-valent inner faces, having a blue (geodesic) interval from the extremity $v$ of the added branch to the root vertex counterclockwise around the map, and a geodesic interval from the root vertex to $v$, whose portion which 
goes from the new added vertex with label $-1$ to $v$ is strictly geodesic (hence represented in red). This is precisely an elementary slice. 

A simple parametrization of $X$ and $Y$ is obtained by introducing the power series $x$ solution of 
\begin{equation}
x=\frac{g\, R^2}{t}\, (1+x+x^2) 
\end{equation}
as it allows to write
\begin{equation}
  \label{eq:XYform}
X=t\, \frac{1-x^3}{1-x}\ , \quad Y=\frac{t}{1-x^3}\ .
\end{equation}

\newpage

\bibliographystyle{myhalpha}
\bibliography{triskell}

\end{document}